\documentclass[11pt]{article}
\usepackage{amsmath,amssymb}
\usepackage{graphicx}
\usepackage{amsfonts}
\usepackage{lmodern}             

\usepackage{amsthm}
\newtheorem{thm}{Theorem}[section]

\newtheorem{prop}[thm]{Proposition}

\newtheorem{define}[thm]{Definition}
\newtheorem{assume}[thm]{Assumption}

\usepackage[USenglish]{babel} 
\usepackage[T1]{fontenc}
\usepackage[ansinew]{inputenc}
\usepackage{mathrsfs}
\usepackage{stmaryrd} 
\usepackage{epstopdf} 
\usepackage{verbatim}
\usepackage{MnSymbol}  
\usepackage{accents}   
\usepackage{bbm}			 
\usepackage{color}     
\usepackage{authblk}   
\usepackage{enumitem}
\usepackage[colorlinks=true, linkcolor=blue, citecolor=blue, urlcolor=blue]{hyperref} 
\usepackage{cleveref}  
\usepackage{placeins}  

\usepackage{tikz, standalone}
\usetikzlibrary{positioning, shapes.geometric, calc}
\definecolor{morange}{RGB}{255,127,14}
\definecolor{mblue}{RGB}{31,119,180}
\definecolor{mred}{RGB}{214,39,40}
\definecolor{mpurple}{RGB}{148,103,189}
\definecolor{mgreen}{RGB}{44,160,44}

\numberwithin{equation}{section}
\numberwithin{figure}{section}

\title{Learning deterministic and stochastic forced Hamiltonian systems}
\date{}
\author[1,2]{Benedikt Brantner\thanks{\texttt{benedikt.brantner@ipp.mpg.de}}}
\author[2,3]{Tomasz M. Tyranowski\thanks{\texttt{t.m.tyranowski@utwente.nl}}}

\affil[1]{\small Max-Planck-Institut f\"ur Plasmaphysik \authorcr Boltzmannstra{\ss}e 2, 85748 Garching, Germany\vspace{.5em}}
\affil[2]{\small Technische Universit\"{a}t M\"{u}nchen, Zentrum Mathematik \authorcr Boltzmannstra{\ss}e 3, 85748 Garching, Germany\vspace{.5em}}
\affil[3]{\small University of Twente, Department of Applied Mathematics \authorcr PO Box 217, 7500AE Enschede, The Netherlands\vspace{.1em}}

\begin{document}

\maketitle

\begin{abstract}
We develop a geometric framework for learning deterministic and stochastic forced Hamiltonian systems with neural networks. Motivated by the Lagrange-d'Alembert principle and the theory of variational integrators, we introduce the notion of a Lagrange-d'Alembert map and establish a $C^r$ convergence theorem for first-order one-step methods. Building on these results, we propose Generalized Forced Hamiltonian Neural Networks (GFHNNs), a class of structure-preserving neural networks obtained by concatenating Lagrange-d'Alembert-Euler maps, and prove a universal approximation theorem for this architecture. We further extend the framework to parameter-dependent systems, leading to Parametric Generalized Forced Hamiltonian Neural Networks (PGFHNNs). By interpreting the multiple Stratonovich integrals appearing in the Stratonovich-Taylor expansion as parameters, the same framework can be applied to stochastic forced Hamiltonian systems whenever information about the underlying Wiener process is available. Our numerical experiments demonstrate that the proposed geometric architectures provide significantly improved long-time stability and accuracy compared to non-geometric residual neural networks, while requiring substantially less training data to achieve a comparable level of performance.
\end{abstract}

\section{Introduction}
\label{sec:intro}

Forced Hamiltonian systems arise naturally in many areas of science and engineering, including robotics, plasma physics, control, and molecular dynamics. In contrast to conservative Hamiltonian systems, forced systems exchange energy with their environment through external inputs and dissipative effects. Preserving the underlying geometric structure of such systems is often essential for obtaining accurate long-time numerical simulations. The main goal of this work is to design a structure-preserving neural network architecture for learning the flow of a parametric forced Hamiltonian system

\begin{align}
\label{eq: Parametric forced Hamiltonian system}
\dot q = \frac{\partial H}{\partial p}(q,p;\mu), \qquad\qquad \dot p = -\frac{\partial H}{\partial q}(q,p;\mu) + f(q,p;\mu),
\end{align}

\noindent
defined on the cotangent bundle $T^*Q \simeq \mathbb{R}^n \times \mathbb{R}^n$ of the configuration space $Q \simeq \mathbb{R}^n$, where $H=H(q,p;
\mu)$ and $f= f(q,p; \mu)$ are the parameter-dependent Hamiltonian and external forcing functions, respectively. More broadly, we propose a geometric framework for learning parameter-dependent Lagrange-d'Alembert maps

\begin{equation}
\label{eq: Parameter dependent Lagrange-d'Alembert map}
\varphi_\mu : T^*Q \longrightarrow T^*Q,
\end{equation}

\noindent
which arise as time-$T$ maps of parametric forced Hamiltonian systems (see Definition~\ref{thm: Definition of parametric Lagrange-d'Alembert maps}).

The use of neural networks for modeling solutions of differential equations dates back to the late 1980s and early 1990s. The potential of neural networks as function approximators was already recognized in \cite{LeeKang1990}. In \cite{MeadeFernandez1994} and \cite{LagarisLikasFotiadis1998}, the authors demonstrated how the flow of ordinary differential equations (ODEs) can be approximated using feedforward neural networks, while \cite{ChenChen1995} provided a more theoretical treatment of neural-network-based methods for solving differential equations.

Around the same time, efforts emerged to incorporate structural properties into neural-network architectures. An early example is \cite{DecoBrauer1995}, where volume preservation and symplecticity are discussed. Since then, a wide variety of geometric and physical structures have been considered, leading to the development of Hamiltonian Neural Networks (HNNs, \cite{GreydanusDzambaYosinski2019}), Symplectic ODE-Nets (SympODENs, \cite{ZhongDeyChakraborty2019}), Lagrangian Neural Networks (LNNs, \cite{CranmerGreydanusHoyerBattagliaSpergelHo2020}), Symplectic Neural Networks (SympNets, \cite{JinZhangZhuTangKarniadakis2020}), port-Hamiltonian Neural Networks (pHNNs, \cite{DesaiMattheakisSondakProtopapasRoberts2021}), Dissipative Hamiltonian Neural Networks (D-HNNs, \cite{Sosanya2022}), Generalized Hamiltonian Neural Networks (GHNNs, \cite{HornKorenGHNN}), Generalized Lagrangian Neural Networks (GLNNs, \cite{XiaoZhangTang2024}), Discrete Forced Lagrangian Neural Networks (DFLNNs, \cite{HansenCelledoniTapley2025}), Parametric Generalized Hamiltonian Neural Networks (PGHNNs, \cite{Horn2026PHD,HornKorenPGHNN}), Stochastic Generating Function Neural Networks (SGFNNs, \cite{ChenWang2025}), as well as approaches for learning non-canonical Hamiltonian systems \cite{ChenMatsubara2021, CourtesFranckKrausNavoretTremant2025}, among many others.

In the following, we review the most relevant literature in greater detail. The purpose is to provide a brief overview of existing neural-network-based integration schemes and to identify the gaps that remain to be addressed.

\subsection{Non-geometric neural network integrators}

Multilayer perceptrons (MLPs) are a class of neural networks that have been used for solving ordinary differential equations (ODEs) by approximating the solution as a neural network output (see, e.g., \cite{ChenChen1995} for an early example). By virtue of the universal approximation theorem, MLPs can approximate continuous functions to arbitrary accuracy \cite{HornikStinchcombeWhite1989}. However, like traditional numerical integrators \cite{HLWGeometric}, these methods are generally non-geometric and do not preserve geometric structures, such as symplecticity, that may be present in Hamiltonian systems. Consequently, they may exhibit poor long-time behavior and fail to accurately reproduce important qualitative features of the underlying dynamics \cite{GreydanusDzambaYosinski2019}.

\subsection{Learning the governing equations: Hamiltonian Neural Networks}

The success of geometric integrators in preserving qualitative features of Hamiltonian dynamics over long time intervals (see \cite{HLWGeometric,McLachlanQuispel,SanzSerna} and the references therein) naturally motivates the incorporation of similar geometric principles into data-driven modeling. In recent years, the development of neural network architectures that preserve geometric structure while learning governing equations from data has emerged as a highly active area of research.

One of the earliest and most influential examples is the Hamiltonian Neural Network (HNN) architecture introduced in \cite{GreydanusDzambaYosinski2019}. HNNs use data consisting of phase-space coordinates and their time derivatives to learn an approximation of the underlying Hamiltonian function. The governing equations are then reconstructed through Hamilton's equations, ensuring that the learned vector field is Hamiltonian. During inference, trajectories are generated by integrating the learned Hamiltonian system, typically with a symplectic integrator.

\subsection{Learning the symplectic flow map}

Similarly to HNNs, Symplectic Neural Networks (SympNets, \cite{JinZhangZhuTangKarniadakis2020}) are designed to learn Hamiltonian dynamics from data. Unlike HNNs, however, they do not learn the governing equations, i.e. the underlying vector field, but instead learn the flow map directly. More precisely, given data consisting of pairs $((q_k,p_k),(q_{k+1},p_{k+1}))$ generated by a Hamiltonian system, the neural network aims to learn the discrete flow map

\begin{equation}
 F_{\Delta t}:T^*Q \longrightarrow T^*Q, \quad\quad (q_k,p_k)\longmapsto(q_{k+1},p_{k+1}),
\end{equation}

\noindent
for a fixed time step $\Delta t$.

H\'enon Neural Networks (H\'enonNets, \cite{BurbyHenonNets2020}) follow a closely related philosophy. Rather than learning the governing equations, H\'enonNets learn the flow map directly through compositions of elementary H\'enon-like symplectic transformations. This construction yields an explicit symplectic architecture that is well suited for approximating Hamiltonian flows. Although SympNets and H\'enonNets were originally introduced from different perspectives, \cite{HornKorenGHNN} showed that both architectures can be interpreted within the framework of Generalized Hamiltonian Neural Networks (GHNNs) (see \Cref{fig:GHNN}). More precisely, SympNets and H\'enonNets arise as particular compositions of symplectic integrators associated with separable Hamiltonians. This observation provides a unified geometric interpretation of several structure-preserving neural network architectures and establishes a direct connection between geometric numerical integration and machine learning.

It is also worth mentioning that the ``scheme learning'' approach proposed in \cite{CourtesFranckKrausNavoretTremant2025} can be interpreted from the same perspective. In that work, a non-canonical Hamiltonian system is learned through a prescribed structure-preserving numerical scheme.

\subsection{The extension to parametric systems: Parametric Generalized Hamiltonian Neural Networks}

Starting from GHNNs, the extension to parameter-dependent systems is straightforward \cite{Horn2026PHD,HornKorenPGHNN}. The resulting architecture is known as a Parametric Generalized Hamiltonian Neural Network (PGHNN) and is illustrated in \Cref{fig:PGHNN}. The parameters are incorporated by augmenting the inputs of the neural networks. Consequently, the separable Hamiltonian used in GHNNs,

\begin{equation}
\tilde H(q,p) = \tilde T(p)+\tilde U(q),
\end{equation}
is replaced by the parameter-dependent Hamiltonian

\begin{equation}
\tilde H(q,p,\mu) = \tilde T(p,\mu)+\tilde U(q,\mu),
\end{equation}
where $\mu\in\mathbb{R}^{l}$ denotes a vector of parameters, and $\tilde T, \tilde U:\mathbb{R}^{n+l}\to\mathbb{R}$ are feedforward neural networks.

\subsection{The extension to forced systems}

Various approaches to constructing structure-preserving neural network architectures for learning forced and dissipative systems have been proposed in the literature. Many of these methods were designed for the Lagrangian counterpart of \Cref{eq: Parametric forced Hamiltonian system}, namely the forced Euler-Lagrange equations on the tangent bundle $TQ$.

Generalized Lagrangian Neural Networks (GLNNs) were proposed in \cite{XiaoZhangTang2024} as an extension of Lagrangian Neural Networks \cite{CranmerGreydanusHoyerBattagliaSpergelHo2020} to non-conservative systems. By replacing the Euler-Lagrange equations with generalized Euler-Lagrange equations, GLNNs allow dissipative and externally forced dynamics to be represented within a Lagrangian framework. In contrast to the present work, GLNNs learn the governing equations by approximating a Lagrangian function and an external force with neural networks. Our approach instead learns Lagrange-d'Alembert maps directly through compositions of explicit structure-preserving numerical integrators. Consequently, GFHNNs (Section~\ref{sec: Generalized Forced Hamiltonian Neural Networks}) and PGFHNNs (Section~\ref{sec: Parametric Generalized Forced Hamiltonian Neural Networks}) learn flow maps directly rather than the underlying differential equations and are naturally rooted in the Lagrange-d'Alembert principle.

Approaches utilizing the Lagrange-d'Alembert principle are described in \cite{HansenCelledoniTapley2025} and \cite{Havens2021}. Discrete Forced Lagrangian Neural Networks (DFLNNs, \cite{HansenCelledoniTapley2025}) are based on the discrete Lagrange-d'Alembert principle and the forced discrete Euler-Lagrange equations. Their method learns a discrete Lagrangian and a discrete forcing term from data and is trained by minimizing the residual of the forced discrete Euler-Lagrange equations. Trajectories are subsequently generated by numerically solving the learned discrete Euler-Lagrange equations. By contrast, our approach represents the dynamics directly through compositions of explicit Lagrange-d'Alembert-Euler maps, avoiding the need to solve learned discrete Euler-Lagrange equations during prediction.

Forced Variational Integrator Networks (FVINs) were proposed in \cite{Havens2021} as an extension of Variational Integrator Networks \cite{Saemundsson2020} to non-conservative mechanical systems, and were further extended to Lagrangian systems on Lie groups \cite{DuruisseauxLeok2023}. Based on the discrete d'Alembert principle, FVINs employ Verlet-type variational integrators and neural-network approximations of the potential energy and forcing terms to learn forced mechanical dynamics. Similar to the present work, FVINs are rooted in forced variational mechanics. However, FVINs assume a prescribed mechanical structure in which the kinetic energy is quadratic in the velocities with a constant mass matrix, and are based on a single variational-integrator update, whereas GFHNNs and PGFHNNs learn the kinetic energy, potential energy, and forcing functions directly and are constructed as arbitrary compositions of Lagrange-d'Alembert-Euler blocks, which considerably increases their expressive power and enables the development of universal approximation results for general forced Hamiltonian systems.

On the Hamiltonian side, Dissipative Hamiltonian Neural Networks (D-HNNs) were proposed in \cite{Sosanya2022} as an extension of Hamiltonian Neural Networks to dissipative systems. D-HNNs learn both a Hamiltonian function and a Rayleigh dissipation function, thereby separating the conservative and dissipative components of the dynamics. Like HNNs, however, D-HNNs learn the governing differential equations rather than the flow map itself. In contrast, GFHNNs and PGFHNNs are flow-map-based models constructed from structure-preserving numerical integrators and naturally incorporate general forcing terms within the Lagrange-d'Alembert framework. Forced systems have also been studied from a port-Hamiltonian perspective in \cite{DesaiMattheakisSondakProtopapasRoberts2021, ZhongDeyChakraborty2019, ZhongDeyChakraborty2020}.

\subsection{The extension to stochastic Hamiltonian systems}

While neural networks have been applied to stochastic differential equations (see \cite{ChenXiu2024, DridiDrumetzFablet2021, Dietrich2023, LiuXiao2019, KongSunZhang2020, WangYao2021, Xu2024, zhu2024dyngma} and the references therein), structure-preserving learning of stochastic geometric systems remains largely unexplored. To the best of our knowledge, the existing results are limited to stochastic Hamiltonian systems without forcing.

A quadrature-based extension of HNNs to stochastic Hamiltonian systems was proposed in \cite{ChengWang2024}. Their approach learns the drift and diffusion Hamiltonians directly from data by combining neural-network approximations with numerical quadrature and moment-based denoising. In contrast, the present work does not learn the governing Hamiltonians. Instead, we learn Lagrange-d'Alembert maps directly through compositions of structure-preserving numerical integrators.

Recently, Stochastic Generating Function Neural Networks (SGFNNs) for learning stochastic Hamiltonian systems from observational data were introduced in \cite{ChenWang2025}. Their approach learns a stochastic generating function whose associated flow map is symplectic by construction and employs an autoencoder architecture to infer latent random variables corresponding to the unobservable noise. While both SGFNNs and the present work aim to learn structure-preserving flow maps, the underlying geometric frameworks are fundamentally different. SGFNNs rely on stochastic generating functions and recover the flow map through generating-function relations, resulting in an implicit symplectic map. Consequently, trajectory generation requires solving nonlinear equations during inference, similarly to implicit geometric integrators. In contrast, our approach is based on the Lagrange-d'Alembert principle and on approximating Lagrange-d'Alembert maps through compositions of explicit Lagrange-d'Alembert-Euler maps. As a result, GFHNNs and PGFHNNs learn the flow map directly and can generate trajectories using only forward evaluations of the network architecture, without the need for iterative nonlinear solvers.

\subsection{Outline of the paper}

The remainder of this paper is organized as follows.

In Section~\ref{sec:Forced Hamiltonian systems and Lagrange-d'Alembert maps}, we review forced Hamiltonian systems and the Lagrange-d'Alembert principle, and introduce the notion of a Lagrange-d'Alembert map. We further review Lagrange-d'Alembert integrators and establish the approximation results that form the mathematical foundation of the proposed neural network architectures.

In Section~\ref{sec: Generalized Forced Hamiltonian Neural Networks}, we introduce Generalized Forced Hamiltonian Neural Networks (GFHNNs) and establish a universal approximation theorem for this architecture.

Section~\ref{sec: Parametric Generalized Forced Hamiltonian Neural Networks} extends the construction to parameter-dependent systems through Parametric Generalized Forced Hamiltonian Neural Networks (PGFHNNs) and proves a corresponding universal approximation theorem.

Section~\ref{sec: Learning time-dependent and stochastic systems} discusses the application of the PGFHNN framework to time-dependent and stochastic forced Hamiltonian systems.

Numerical experiments for autonomous, time-dependent, and stochastic forced Hamiltonian systems are presented in Section~\ref{sec: Numerical experiments}.

Finally, Section~\ref{sec: Summary} summarizes the main findings and outlines directions for future research.

\section{Forced Hamiltonian systems and Lagrange-d'Alembert maps}
\label{sec:Forced Hamiltonian systems and Lagrange-d'Alembert maps}

We begin by reviewing forced Hamiltonian systems, their flows, and Lagrange-d'Alembert integrators, and by establishing several results needed later in the paper. In particular, we introduce the notion of a Lagrange-d'Alembert map, prove a $C^r$ convergence theorem for first-order one-step methods, and extend the framework to parametric forced Hamiltonian systems.

\subsection{Time-dependent forced Hamiltonian systems}
\label{sec:Time-dependent forced Hamiltonian systems}

For simplicity, in this work we assume that the configuration space $Q \simeq \mathbb{R}^n$ of the systems under consideration is a vector space. The evolution of a time-dependent (non-autonomous) forced Hamiltonian system takes place on the cotangent bundle $T^*Q \simeq \mathbb{R}^n \times \mathbb{R}^n$ and is governed by the differential equations

\begin{align}
\label{eq: Time-dependent forced Hamiltonian system}
\dot q = \frac{\partial H}{\partial p}(q,p,t), \qquad\qquad \dot p = -\frac{\partial H}{\partial q}(q,p,t)+f(q,p,t),
\end{align}

\noindent
where $H: T^*Q \times \mathbb{R} \longrightarrow \mathbb{R}$ is the Hamiltonian function and $f_H: T^*Q \times \mathbb{R} \longrightarrow T^*Q$ is a fiber-preserving mapping representing external non-conservative forcing, given in coordinates by $f_H(q,p,t) = (q, f(q,p,t))$ (see \cite{MarsdenRatiuSymmetry, MarsdenWestVarInt}). The corresponding flow $F_{t,t_0}: T^*Q \longrightarrow T^*Q$ depends on both the final time $t$ and the initial time $t_0$. We refer to the flows of forced Hamiltonian systems as \emph{Lagrange-d'Alembert flows}. In order to guarantee the existence, uniqueness, and suitable regularity of the flow $F_{t,t_0}$, the Hamiltonian and the external force have to satisfy appropriate conditions.

\begin{assume}[{\bf Standing assumptions on the forced Hamiltonian system}]
\label{ass:standing}
Let $T>0$ and let $r\ge 0$ be an integer. Throughout this work, we consider a forced Hamiltonian system with the Hamiltonian $H: T^*Q \times [0,T] \longrightarrow \mathbb{R}$ and the forcing term
$f: T^*Q \times [0,T] \longrightarrow \mathbb{R}^n$, and assume that the following conditions hold:
\begin{enumerate}[label=(\roman*), ref=(\roman*)]
    \item\label{ass:H-reg} $H \in C^{r+2}(T^*Q\times [0,T])$ and $f \in C^{r+1}(T^*Q\times [0,T], \mathbb{R}^n)$;
	\item\label{ass:hyperregular} $H$ is hyperregular, that is, for all $t \in [0,T]$, the fiber derivative
		\[
		\mathbb{F}H(\cdot,\cdot,t): T^*Q \ni (q,p) \longmapsto \bigg(q,\frac{\partial H}{\partial p}(q,p,t) \bigg) \in TQ
		\]
		is a diffeomorphism;
	\item\label{ass:global-flow} For every $(q_0,p_0)\in T^*Q$ and every $t_0\in[0,T]$, the solution of \eqref{eq: Time-dependent forced Hamiltonian system} with the initial condition $(q(t_0),p(t_0))=(q_0,p_0)$ exists on the entire interval $[0,T]$. Equivalently, the associated flow map
\[
F_{t,t_0}:T^*Q\longrightarrow T^*Q
\]
is well defined for all $t,t_0\in[0,T]$.
\end{enumerate}
\end{assume}

\noindent
Assumption~\ref{ass:standing}~\ref{ass:H-reg} guarantees the local existence and uniqueness of solutions to \Cref{eq: Time-dependent forced Hamiltonian system}, with the flow $F_{t,t_0}$ being jointly $C^{r+1}$ in all variables $(t,t_0,q,p)$ (for a proof, see, e.g., \cite{CoddingtonLevinsonBook,HaleODE1969}). While this assumption is sufficient for the purposes of the present work, it could be relaxed if necessary. In particular, some applications may require less regularity in $t$. In that case, additional conditions must be imposed, for instance, a Lipschitz condition in $(q,p)$ that is uniform in time on $\frac{\partial H}{\partial q}$, $\frac{\partial H}{\partial p}$, and $f$. Assumption~\ref{ass:standing}~\ref{ass:hyperregular} is imposed for simplicity and clarity of exposition, ensuring that \Cref{eq: Time-dependent forced Hamiltonian system} admits an equivalent Lagrangian formulation and that type-I generating functions exist (see Section~\ref{sec:Lagrange-d'Alembert principle}). This assumption can also be relaxed if necessary; however, for degenerate Hamiltonians one must verify which boundary conditions are admissible and choose an appropriate type of generating function. Assumption~\ref{ass:standing}~\ref{ass:global-flow} is imposed to guarantee that solutions with initial conditions in a compact set remain in a compact set over the time interval $[0,T]$. This property will be required later in the proof of the universal approximation theorem.

\subsection{Lagrange-d'Alembert principle}
\label{sec:Lagrange-d'Alembert principle}

Unlike for canonical Hamiltonian systems, the flow $F_{t,t_0}:T^*Q \longrightarrow T^*Q$ for \Cref{eq: Time-dependent forced Hamiltonian system} is in general not symplectic. However, forced Hamiltonian systems have an underlying variational principle, the so-called \emph{Lagrange-d'Alembert principle}. Denote by $C^1([t_a,t_b],T^*Q)$ the space of all $C^1$ paths in $T^*Q$, and define the phase space action functional $\mathcal{B}: C^1([t_a,t_b],T^*Q) \longrightarrow \mathbb{R}$ by

\begin{equation}
\label{eq: Phase space action functional}
\mathcal{B}[q(\cdot),p(\cdot)] = \int_{t_a}^{t_b} \big( p(t)\cdot \dot q(t)-H(q(t),p(t),t)\big)\,dt.
\end{equation}

\noindent
The Lagrange-d'Alembert principle states that a curve $(q(t),p(t))$ in $T^*Q$ satisfies \Cref{eq: Time-dependent forced Hamiltonian system} for $t\in[t_a,t_b]$ if and only if it satisfies the variational equation

\begin{equation}
\label{eq: Lagrange-d'Alembert principle}
\delta \mathcal{B}[q(\cdot),p(\cdot)] + \int_{t_a}^{t_b} f(q(t),p(t),t)\cdot \delta q(t) \, dt = 0
\end{equation}

\noindent
for all variations $\delta q(t)$ with fixed endpoints, $\delta q(t_a)=\delta q(t_b)=0$, and all variations $\delta p(t)$. This result was proved, for example, in \cite{MarsdenRatiuSymmetry} and \cite{MarsdenWestVarInt}. The Lagrange-d'Alembert principle generalizes Hamilton's principle for canonical Hamiltonian systems, and provides the intrinsic geometric structure of forced Hamiltonian systems. It can be further used to define the so-called type-I generating function $S$ and the type-I exact discrete forces $f^\pm$ as

\begin{subequations}
\label{eq: Generating function and exact discrete forces}
\begin{align}
\label{eq: Generating function and exact discrete forces 1}
S(q_a,q_b; t_a, t_b) &= \int_{t_a}^{t_b} \Big(\bar p(t)\cdot\dot{\bar{q}}(t) - H(\bar q(t),\bar p(t),t)\Big)\,dt, \\
\label{eq: Generating function and exact discrete forces 2}
f^+(q_a,q_b; t_a, t_b) &= \int_{t_a}^{t_b}  f(\bar q(t),\bar p(t),t)\cdot \frac{\partial \bar q(t)}{\partial q_b}\,dt,  \\
\label{eq: Generating function and exact discrete forces 3}
f^-(q_a,q_b; t_a, t_b) &= \int_{t_a}^{t_b}  f(\bar q(t),\bar p(t),t)\cdot \frac{\partial \bar q(t)}{\partial q_a}\,dt,
\end{align}
\end{subequations}

\noindent
where $(\bar q(t; q_a,q_b,t_a,t_b),\bar p(t; q_a,q_b,t_a,t_b))=F_{t,t_a}(q_a,p_a)$ is the trajectory of \Cref{eq: Time-dependent forced Hamiltonian system} satisfying the boundary conditions $\bar q(t_a; q_a,q_b,t_a,t_b)=q_a$ and $\bar q(t_b; q_a,q_b,t_a,t_b)=q_b$. Local existence of these functions, for $q_b$ sufficiently close to $q_a$ and $t_b$ sufficiently close to $t_a$, is guaranteed by \cite[Corollary~7.4.6]{MarsdenRatiuSymmetry}. Instead of specifying the motion directly in terms of positions and momenta, the generating function and exact discrete forces capture the relationship between the initial and final states $(q_b,p_b)=F_{t_b,t_a}(q_a,p_a)$ via the equations

\begin{align}
\label{eq: Lagrange-d'Alembert flow via generating function and forces}
p_a = -D_1 S(q_a,q_b; t_a, t_b)-f^-(q_a,q_b; t_a, t_b), \quad\qquad p_b = D_2 S(q_a,q_b; t_a, t_b)+f^+(q_a,q_b; t_a, t_b).
\end{align}

\noindent
This result was proved in \cite[Lemma~1.6.2]{MarsdenWestVarInt} and \cite[Section~3.2.4]{MarsdenWestVarInt}, and together with an application of the implicit function theorem as in \cite[Theorem~7.4.5]{MarsdenRatiuSymmetry}, guarantees the differentiability of $S(q_a,q_b; t_a, t_b)$ and $\bar q(t; q_a,q_b,t_a,t_b)$ with respect to $q_a$ and $q_b$. A similar result for stochastic forced Hamiltonian systems was proved in \cite{KrausTyranowski2019}.

\subsection{Lagrange-d'Alembert maps}
\label{sec:Lagrange-d'Alembert maps}

A Hamiltonian map is a (symplectic) diffeomorphism of the phase space $T^*Q$ that arises as the time-$t$ evolution map of the flow of a canonical Hamiltonian system (see, e.g., \cite{McDuffSalamonBook,PolterovichBook2001}). In a similar spirit, we introduce the following definition.

\begin{define}[{\bf Lagrange-d'Alembert map}]
\label{thm: Definition of Lagrange-d'Alembert maps}
Let $r\ge 0$ be an integer. A diffeomorphism $\varphi:T^*Q \longrightarrow T^*Q$ is called a Lagrange-d'Alembert map of class $C^{r+1}$ if there exists a time-dependent forced Hamiltonian system \eqref{eq: Time-dependent forced Hamiltonian system} satisfying Assumption~\ref{ass:standing} for some $T>0$, with the associated Lagrange-d'Alembert flow $F_{t,t_0}$, such that $\varphi = F_{T,0}$. The set of all Lagrange-d'Alembert maps of class $C^{r+1}$ on $T^*Q$ is denoted by $LdA^{r+1}(T^*Q)$.
\end{define}

\paragraph{Remark.} Note that a Lagrange-d'Alembert map does not uniquely determine its generating forced Hamiltonian system, since different forcing terms and Hamiltonian functions (e.g., up to additive time-dependent terms or reparametrizations) can generate the same flow map. Furthermore, in light of Definition~\ref{thm: Definition of Lagrange-d'Alembert maps}, a Lagrange-d'Alembert map $\varphi$ can also be expressed in terms of a type-I generating function and type-I exact discrete forces \eqref{eq: Generating function and exact discrete forces}.

\subsection{Lagrange-d'Alembert integrators}
\label{sec: Lagrange-d'Alembert integrators}

Geometric numerical integration provides structure-preserving discretizations of dynamical systems. Among such methods, \emph{variational integrators} occupy a central role. These numerical schemes are based on discrete variational principles and provide a natural framework for the discretization of Lagrangian systems, including forced, dissipative, or constrained ones. These methods have the advantage that they are symplectic when applied to systems without forcing, and in the presence of symmetries, they satisfy a discrete version of Noether's theorem. For an overview of variational integration see \cite{MarsdenWestVarInt,HallLeokSpectral,KaneMarsden2000,LeokShingel,LeokZhang,OberBlobaum2016,RowleyMarsden,TyranowskiDesbrunLinearLagrangians}. Variational integrators were introduced in the context of finite-dimensional mechanical systems, but were later generalized to Lagrangian field theories \cite{MarsdenPatrickShkoller} and applied in a wide range of computational settings, for example in elasticity, electrodynamics, or fluid dynamics; see \cite{LewAVI,Pavlov,SternDesbrun,TyranowskiDesbrunRAMVI}.

A class of variational integrators derived from a discrete version of the Lagrange-d'Alembert principle \eqref{eq: Lagrange-d'Alembert principle} are the so-called \emph{Lagrange-d'Alembert integrators} \cite{MarsdenWestVarInt}. These integrators are constructed by specifying a discrete Lagrangian $L_d \approx S$ that approximates the exact generating function, together with discrete forces $f_d^\pm \approx f^\pm$ that approximate the exact discrete forces. For a given time step $\Delta t$, an approximate flow of the forced Hamiltonian system, $\widehat F_{t_{k+1}, t_k}: (q_k,p_k)\longmapsto(q_{k+1},p_{k+1})$, is implicitly given by the equations

\begin{align}
\label{eq: Lagrange-d'Alembert integrator}
p_k &= -D_1 L_d(q_k,q_{k+1}; t_k,t_{k+1})-f_d^-(q_k,q_{k+1}; t_k,t_{k+1}), \nonumber \\
p_{k+1} &= \phantom{-}D_2 L_d(q_k,q_{k+1}; t_k,t_{k+1})+f_d^+(q_k,q_{k+1}; t_k,t_{k+1}).
\end{align}

\noindent
An integrator derived in this way generates a discrete trajectory $((q_0,p_0), (q_1,p_1), \ldots)$ in $T^*Q$. The discrete equations~\eqref{eq: Lagrange-d'Alembert integrator} also follow from a discrete counterpart of the Lagrange-d'Alembert principle \eqref{eq: Lagrange-d'Alembert principle}. Let $t_k=t_a+k \Delta t$ for $k=0,1,\ldots,N$ with $\Delta t = (t_b-t_a)/N$, and define the discrete action functional

\begin{equation}
\label{eq: Discrete action functional}
\mathcal{B}_d[\{q_k\}_{k=0,\ldots,N}] = \sum_{k=0}^{N-1} L_d(q_k,q_{k+1}; t_k,t_{k+1}).
\end{equation}

\noindent
The \emph{discrete Lagrange-d'Alembert principle} states that the discrete system follows the trajectory $\{q_k\}_{k=0,\ldots,N}$ that satisfies

\begin{equation}
\label{eq: Discrete Lagrange-d'Alembert principle}
\delta \mathcal{B}_d + \sum_{k=0}^{N-1} \big( f_d^-(q_k,q_{k+1}; t_k,t_{k+1}) \cdot \delta q_k + f_d^+(q_k,q_{k+1}; t_k,t_{k+1}) \cdot \delta q_{k+1} \big) = 0,
\end{equation}

\noindent
for all variations $\{\delta q_k\}_{k=0,\ldots,N}$ vanishing at the endpoints, that is, $\delta q_0=\delta q_N=0$. This is equivalent to the system of equations

\begin{equation}
D_2L_d(q_{k-1},q_k;t_{k-1},t_k) + D_1L_d(q_{k},q_{k+1};t_{k},t_{k+1}) + f_d^+(q_{k-1},q_k;t_{k-1},t_k) + f_d^-(q_k,q_{k+1}; t_k,t_{k+1}) = 0,
\end{equation}

\noindent
for $k=1,\ldots,N-1$, which can be recast as the system~\eqref{eq: Lagrange-d'Alembert integrator} by introducing auxiliary momentum variables $p_k$ (see \cite[Section~3.2]{MarsdenWestVarInt} for more details). Two Lagrange-d'Alembert integrators can be composed, and in this way one obtains a new Lagrange-d'Alembert integrator. Let $\widehat F_1: (q_k,p_k)\longmapsto(q_{k+1},p_{k+1})$ be generated by $L_1(q_k,q_{k+1})$ and $f_1^\pm(q_k,q_{k+1})$, and let $\widehat F_2: (q_k,p_k)\longmapsto(q_{k+1},p_{k+1})$ be generated by $L_2(q_k,q_{k+1})$ and $f_2^\pm(q_k,q_{k+1})$, respectively, where we omit the time arguments for brevity. Then the composition $\widehat F=\widehat F_2 \circ \widehat F_1$ is also a Lagrange-d'Alembert integrator, with the discrete Lagrangian $L_d$ and discrete forces $f_d^\pm$ given by (see \cite[Example~3.2.4]{MarsdenWestVarInt} and \cite[Example~4.3]{WestPHD} for details)

\begin{align}
\label{eq: S, f+, and f- for the composition}
L_d(q_k,q_{k+1}; t_k,t_{k+1}) &= L_1(q_k,q_c) + L_2(q_c,q_{k+1}), \nonumber \\
f_d^+(q_k,q_{k+1}; t_k,t_{k+1})&= f_2^+(q_c,q_{k+1}) + \big(f_1^+(q_k,q_c) + f_2^-(q_c,q_{k+1}) \big)\cdot \frac{\partial q_c}{\partial q_{k+1}},\\
f_d^-(q_k,q_{k+1}; t_k,t_{k+1})&= f_1^-(q_k,q_c) + \big(f_1^+(q_k,q_c) + f_2^-(q_c,q_{k+1}) \big)\cdot \frac{\partial q_c}{\partial q_k}, \nonumber
\end{align}

\noindent
where the point $q_c=q_c(q_k,q_{k+1})$ is determined by the condition

\begin{equation}
\label{eq: Condition for qc(q_k,q_{k+1})}
D_2 L_1(q_k,q_c) + D_1 L_2(q_c,q_{k+1}) + f_1^+(q_k,q_c) + f_2^-(q_c,q_{k+1})=0.
\end{equation}

The simplest example of a Lagrange-d'Alembert integrator is an extension of the symplectic Euler scheme,

\begin{align}
\label{eq: General Lagrange-d'Alembert Euler scheme}
q_{k+1} &= q_k + \Delta t \frac{\partial H}{\partial p}(q_{k+1},p_k,t_k), \nonumber \\
p_{k+1} &= p_k+\Delta t \bigg[-\frac{\partial H}{\partial q}(q_{k+1},p_k,t_k)+ f(q_{k+1},p_k,t_k)\bigg],
\end{align}

\noindent
which we will call the Lagrange-d'Alembert-Euler scheme and denote by $\Phi^{H,f}_{t_{k+1}, t_k}: (q_k,p_k)\longmapsto(q_{k+1},p_{k+1})$. The forced Hamiltonian system \eqref{eq: Time-dependent forced Hamiltonian system} can also be solved by applying a splitting method. In this approach, the associated vector field is decomposed into a conservative Hamiltonian component and a non-conservative forcing component, namely,

\begin{equation}
\begin{aligned}
\left\{
\begin{aligned}
\dot{q} &= \phantom{-}\frac{\partial H}{\partial p}(q,p,t), \\
\dot{p} &=-\frac{\partial H}{\partial q}(q,p,t),
\end{aligned}
\right.
\qquad
\text{and} \qquad
\left\{
\begin{aligned}
\dot{q} &= 0, \\
\dot{p} &= f(q,p,t).
\end{aligned}
\right.
\end{aligned}
\end{equation}

\noindent
Applying the Lagrange-d'Alembert-Euler method to these subsystems yields the discrete flows $\Phi^{H}_{t_{k+1}, t_k} \equiv\Phi^{H,0}_{t_{k+1}, t_k}$ and $\Phi^{f}_{t_{k+1}, t_k} \equiv\Phi^{0,f}_{t_{k+1}, t_k}$, respectively. The Lie-Trotter splitting method \cite{HLWGeometric} is then defined by composition,

\begin{equation}
\label{eq:Lie-Trotter splitting method}
\Phi^{LT}_{t_{k+1}, t_k} = \Phi^{f}_{t_{k+1}, t_k} \circ\Phi^{H}_{t_{k+1}, t_k}.
\end{equation}

\noindent
Since both component flows are Lagrange-d'Alembert integrators, their composition is again a Lagrange-d'Alembert integrator.

Both the Lagrange-d'Alembert-Euler method and the Lie-Trotter method are first-order accurate. It is a standard result in numerical analysis that such methods are therefore convergent of order one. However, for our purposes, we require convergence in the strong $C^r$ topology on compact sets in order to formulate a universal approximation theorem for our neural networks in Section~\ref{sec: Universal approximation theorem for GFHNNs}.

\begin{thm}[{\bf $C^r$ convergence of first-order one-step methods}]\label{thm:Convergence of first-order one-step methods}
Let $r\ge 0$ be an integer and let $g \in C^{r+1}(\mathbb{R}^d \times [0,T],\mathbb{R}^d)$. Consider the initial value problem
\begin{equation}
\dot x(t) = g(x(t),t), \qquad x(0)=x.
\end{equation}
Denote by $F_{t,t_0}:\mathbb{R}^d \longrightarrow \mathbb{R}^d$ the associated flow, and assume that it is defined for all $t,t_0\in[0,T]$. Let $\hat F_{t+\Delta t,t} : \mathbb{R}^d \longrightarrow \mathbb{R}^d$ be a one-step method of class $C^r$ with respect to $x$ and with time step $\Delta t$, and define

\begin{equation}
x_{k+1} = \hat F_{t_{k+1},t_k}(x_k), \qquad t_k = k \Delta t, \qquad x_0 = x,
\end{equation}

\noindent
and $\Phi_{\Delta t}^k := \hat F_{t_{k},t_{k-1}}\circ\ldots\circ\hat F_{t_{1},t_0}$. Assume that for each compact $K \subset \mathbb{R}^d$ the following stability and consistency conditions hold:

\begin{enumerate}[label=(\roman*), ref=(\roman*)]
\item\label{ass:First-order consistency} There exists a constant $M_K>0$ such that 
\begin{equation}
\sup_{(x,t)\in K\times [0,T-\Delta t]} \left\| D^l \hat F_{t+\Delta t,t}(x) - D^l F_{t+\Delta t,t}(x) \right\| \le M_K \Delta t^2,
\qquad l=0,\dots,r,
\end{equation}

\noindent
where $D^l$ denotes the $l$-th derivative with respect to the $x$ variable;

\item\label{ass:stability} There exists a constant $L_K>0$ such that 

\begin{equation}
\sup_{(x,t)\in  K\times [0,T-\Delta t]} \left\|D \hat F_{t+\Delta t,t}(x)\right\| \le 1 + L_K \Delta t.
\end{equation}

\end{enumerate}

Then for every compact $K \subset \mathbb{R}^d$ there exists a constant $C_K>0$ such that

\begin{equation}
\label{eq:Global error estimate of a one-step method}
\max_{0 \le k \le \lfloor T/\Delta t\rfloor} \left\|\Phi_{\Delta t}^{k} - F_{t_k,0}\right\|_{C^r(K)} \le C_K \Delta t.
\end{equation}
\end{thm}

\begin{proof}
We only indicate the steps that go beyond the standard results from the theory of one-step methods. Let $K\subset \mathbb{R}^d$ be compact, and let $\Delta t >0$ be sufficiently small.

\begin{enumerate}[label=\alph*)]

\item By the classical convergence theory of one-step methods (see, e.g., \cite[Ch.~II]{HWODE1}), consistency~\ref{ass:First-order consistency} and stability~\ref{ass:stability} imply that there exists a constant $C_{K,0}>0$ such that

\begin{equation}
\label{eq:C^0 estimate}
\max_{0 \le k \le \lfloor T/\Delta t\rfloor}  \sup_{x\in K} \left|\Phi_{\Delta t}^{k}(x) - F_{t_k,0}(x)\right| \le C_{K,0} \Delta t
\end{equation}

\noindent
for all sufficiently small $\Delta t$.

\item Since $g \in C^{r+1}$, the flow $F_{t,t_0}$ is jointly of class $C^{r+1}$ in the variables $(t, t_0, x)$, and therefore the set

\begin{equation}
K_T:=\{ F_{t,0}(x) \;|\; t\in[0,T], x \in K \}
\end{equation}

\noindent
is compact. Moreover, in view of the $C^0$ estimate \eqref{eq:C^0 estimate}, there exists a compact set $\bar K_{T}$ such that $K_T \subset \bar K_{T}$ and $x_k \in \bar K_{T}$ for $k=0,1,\ldots,\lfloor T/\Delta t\rfloor$ and all sufficiently small $\Delta t >0$. In other words, both the exact and the numerical trajectories starting in the compact set $K$ remain in the compact set $\bar K_{T}$.

\item In addition, the derivatives of $F_{t,t_0}$ with respect to $x$ satisfy variational equations \cite[\S I.14]{HWODE1}. In particular, all derivatives $D^l F_{t,t_0}(x)$ for $l=0,1,\ldots,r$ are bounded on compact sets.

\item\label{Estimating the first derivative} Let $J_k(x) := D\Phi_{\Delta t}^k(x)$ and $J(x,t) := DF_{t,0}(x)$. Using the chain rule, we obtain

\begin{equation}
\label{eq:J_(k+1) from the chain rule}
J_{k+1}(x) = D\hat F_{t_{k+1},t_k}(x_k)\, J_k(x),
\end{equation}

\noindent
and

\begin{equation}
\label{eq:J(t_(k+1)) from the chain rule}
J(x,t_{k+1}) = DF_{t_{k+1},t_k}\big(F_{t_{k},0}(x)\big)\, J(x,t_k).
\end{equation}

\noindent
Define the error after $k$ steps by

\begin{equation}
\label{eq:Error of the first derivative}
E_k(x):=J_k(x)-J(x,t_k),  \qquad \text{for $k=1,2,\ldots$},
\end{equation}

\noindent
with $E_0(x):=0$. Using \Cref{eq:J_(k+1) from the chain rule}, \Cref{eq:J(t_(k+1)) from the chain rule}, and \Cref{eq:Error of the first derivative}, and adding and subtracting suitable intermediate terms, we obtain

\begin{align}
\label{eq:E_(k+1) recursion}
E_{k+1}(x) = \underbrace{D\hat F_{t_{k+1},t_k}(x_k) E_k(x)}_{(A_1)} &+ \underbrace{\big(D\hat F_{t_{k+1},t_k}(x_k)-DF_{t_{k+1},t_k}(x_k)\big) J(x,t_k)}_{(A_2)} \nonumber \\
           &+ \underbrace{\Big(DF_{t_{k+1},t_k}(x_k)-DF_{t_{k+1},t_k}\big(F_{t_{k},0}(x)\big)\Big)J(x,t_k)}_{(A_3)}.
\end{align}

\noindent
The term $(A_1)$ can be bounded using the stability assumption \ref{ass:stability} on the set $\bar K_{T}$. The term $(A_2)$ can be bounded using the consistency assumption \ref{ass:First-order consistency} on the set $\bar K_{T}$ together with the boundedness of $J(x,t)$ on $K\times [0,T]$, yielding $(A_2)=O(\Delta t^2)$. To bound the term $(A_3)$, note that

\begin{equation}
DF_{t_{k+1},t_k}(x) = I + \Delta t D_x g(x,t_k) + O(\Delta t^2),
\end{equation}

\noindent
where $I$ is the identity matrix. Since $J(x,t)$ is bounded on $K\times [0,T]$ and $D_x g$ is of class $C^r$ on $\bar K_{T} \times [0,T]$, the $C^0$ error estimate \eqref{eq:C^0 estimate} implies that $(A_3)=O(\Delta t^2)$. Altogether, we have

\begin{equation}
\left\|E_{k+1}(x)\right\| \le (1+ L_{\bar K_{T}}\Delta t) \left\|E_{k}(x)\right\| + B_K \Delta t^2, \qquad \text{for $k=0,1,\ldots$},
\end{equation}

\noindent
where $B_K>0$ is a constant. A discrete Gr\"{o}nwall argument then gives

\begin{equation}
\left\|E_{k+1}(x)\right\| \le e^{k L_{\bar K_{T}}\Delta t} \left\|E_{0}(x)\right\| + C_{K,1} \Delta t, \qquad \text{with } C_{K,1}=\frac{e^{L_{\bar K_{T}}T}-1}{L_{\bar K_{T}}} B_K.
\end{equation}

\noindent
Consequently, we get the $C^1$ error estimate

\begin{equation}
\label{eq:C^1 estimate}
\max_{0 \le k \le \lfloor T/\Delta t\rfloor}  \sup_{x\in K} \left\|D\Phi_{\Delta t}^{k}(x) - DF_{t_k,0}(x)\right\| \le C_{K,1} \Delta t.
\end{equation}

\item In order to estimate the error of higher-order derivatives, we argue by induction on $l \le r$. Let $J^{(l)}_k(x) := D^l\Phi_{\Delta t}^k(x)$ and $J^{(l)}(x,t) := D^lF_{t,0}(x)$, and let

\begin{equation}
\label{eq:Error of the l-th derivative}
E^{(l)}_k(x):=J^{(l)}_k(x)-J^{(l)}(x,t_k),  \qquad \text{for $k=1,2,\ldots$},
\end{equation}

\noindent
with $E^{(l)}_0:=0$. We proceed as in \eqref{eq:J_(k+1) from the chain rule} and \eqref{eq:J(t_(k+1)) from the chain rule}. By repeated differentiation and the Fa\`{a} di Bruno formula (see \cite[Lemma II.2.8]{HWODE1} and \cite{ConstantineSavits1996}), both $D^l \Phi^k_{\Delta t}$ and $D^lF_{t,0}$ satisfy recursion formulas of the form

\begin{align}
\label{eq:Recursion for higher-order derivatives}
J^{(l)}(x,t_{k+1}) &= DF_{t_{k+1},t_k}\big(F_{t_{k},0}(x)\big)\, J^{(l)}(x,t_k) + \mathcal{P}_l, \nonumber \\
J^{(l)}_{k+1}(x) &= D\hat F_{t_{k+1},t_k}(x_k)\, J^{(l)}_k(x) + \mathcal{\hat P}_l,
\end{align}

\noindent
where $\mathcal{P}_l$ and $\mathcal{\hat P}_l$ depend only on derivatives of order~$<l$. Subtracting the two recursions and using the consistency~\ref{ass:First-order consistency} and stability \ref{ass:stability}  assumptions as in \Cref{eq:E_(k+1) recursion}, yields

\begin{equation}
\left\|E^{(l)}_{k+1}(x)\right\| \le (1+ L_{\bar K_{T}}\Delta t) \left\|E^{(l)}_{k}(x)\right\| + B_{K,l} \Delta t^2, \qquad \text{for $k=0,1,\ldots$},
\end{equation}

\noindent
where the induction hypothesis controls the lower-order terms. Applying discrete Gr\"{o}nwall gives

\begin{equation}
\left\|E^{(l)}_{k+1}(x)\right\| \le C_{K,l} \Delta t.
\end{equation}

\noindent
Consequently, we get the $C^l$ error estimate

\begin{equation}
\label{eq:C^l estimate}
\max_{0 \le k \le \lfloor T/\Delta t\rfloor}  \sup_{x\in K} \left\|D^l\Phi_{\Delta t}^{k}(x) - D^lF_{t_k,0}(x)\right\| \le C_{K,l} \Delta t.
\end{equation}

\item Combining the estimates \eqref{eq:C^0 estimate}, \eqref{eq:C^1 estimate}, and \eqref{eq:C^l estimate} for all $l \le r$ yields the claim with $C_K = \max \{C_{K,0},\ldots,C_{K,r}\}$.
\end{enumerate}
\end{proof}

\paragraph{Remark.} Using the implicit function theorem, it is straightforward to verify that under Assumption~\ref{ass:standing} and for a sufficiently small time step $\Delta t$, both the Lagrange-d'Alembert-Euler~\eqref{eq: General Lagrange-d'Alembert Euler scheme} and Lie-Trotter~\eqref{eq:Lie-Trotter splitting method} integrators satisfy the assumptions of Theorem~\ref{thm:Convergence of first-order one-step methods}.

\subsection{Parametric time-dependent forced Hamiltonian systems}
\label{sec:Parametric time-dependent forced Hamiltonian systems}

In many practical applications in physics and engineering, the Hamiltonian $H:T^*Q\times \mathbb{R}\times I\longrightarrow \mathbb{R}$ and external forcing function $f_H:T^*Q\times \mathbb{R}\times I\longrightarrow T^*Q$ may depend on a parameter $\mu \in I$, where $I$ is an open subset of $\mathbb{R}^l$. The evolution of such a system is governed by the equations

\begin{align}
\label{eq: Parametric time-dependent forced Hamiltonian system}
\dot q = \frac{\partial H}{\partial p}(q,p,t,\mu), \qquad\qquad \dot p = -\frac{\partial H}{\partial q}(q,p,t,\mu)+f(q,p,t,\mu),
\end{align}

\noindent
and its flow $F_{t,t_0}:T^*Q\times I \longrightarrow T^*Q$ is parameter-dependent. We will refer to such flows as \emph{parametric Lagrange-d'Alembert flows}. In order to guarantee the existence, uniqueness, and suitable regularity of the parametric flow $F_{t,t_0}$, the conditions in Assumption~\ref{ass:standing} must also account for regularity with respect to the parameter.

\begin{assume}[{\bf Standing assumptions on parametric forced Hamiltonian systems}]
\label{ass:standing parametric}
Let $T>0$, let $r\ge 0$ be an integer, and let $I\subset \mathbb{R}^l$ be open. Throughout this work, the parameter-dependent Hamiltonian
$H: T^*Q \times [0,T]\times I \longrightarrow \mathbb{R}$ and the forcing term
$f: T^*Q \times [0,T]\times I \longrightarrow \mathbb{R}^n$ satisfy:
\begin{enumerate}[label=(\roman*), ref=(\roman*)]
    \item\label{ass:H-reg parametric} $H \in C^{r+2}(T^*Q\times [0,T]\times I)$ and $f \in C^{r+1}(T^*Q\times [0,T]\times I, \mathbb{R}^n)$;
		\item\label{ass:hyperregular parametric} $H$ is hyperregular, that is, for all $t \in [0,T]$ and $\mu \in I$, the fiber derivative
		\[
		\mathbb{F}H(\cdot,\cdot,t,\mu): T^*Q \ni (q,p) \longmapsto \bigg(q,\frac{\partial H}{\partial p}(q,p,t,\mu) \bigg) \in TQ
		\]
		is a diffeomorphism;
		\item\label{ass:global-flow parametric} For every $(q_0,p_0, \mu)\in T^*Q\times I$ and every $t_0\in[0,T]$, the solution of \eqref{eq: Parametric time-dependent forced Hamiltonian system} with the initial condition $(q(t_0),p(t_0))=(q_0,p_0)$ exists on the entire interval $[0,T]$. Equivalently, the associated flow map
\[
F_{t,t_0}:T^*Q\times I\longrightarrow T^*Q
\]
is well defined for all $t,t_0\in[0,T]$.
\end{enumerate}
\end{assume}

\noindent
Assumption~\ref{ass:standing parametric} guarantees the existence and uniqueness of the parametric flow $F_{t,t_0}$, which is jointly $C^{r+1}$ in all variables $(t,t_0,q,p,\mu)$ (see, e.g., \cite{CoddingtonLevinsonBook,HaleODE1969}). Moreover, all results from Sections~\ref{sec:Time-dependent forced Hamiltonian systems}-\ref{sec:Lagrange-d'Alembert maps} extend naturally to the parametric setting. In particular, we adopt the following definition.

\begin{define}[{\bf Parametric Lagrange-d'Alembert map}]
\label{thm: Definition of parametric Lagrange-d'Alembert maps}
Let $r\ge 0$ be an integer. A map $\varphi:T^*Q \times I \longrightarrow T^*Q$ is called a parametric Lagrange-d'Alembert map of class $C^{r+1}$ if there exists a parametric time-dependent forced Hamiltonian system \eqref{eq: Parametric time-dependent forced Hamiltonian system} satisfying Assumption~\ref{ass:standing parametric} for some $T>0$, with the associated parametric Lagrange-d'Alembert flow $F_{t,t_0}$, such that $\varphi = F_{T,0}$. The set of all parametric Lagrange-d'Alembert maps of class $C^{r+1}$ on $T^*Q\times I$ is denoted by $LdA^{r+1}(T^*Q,I)$.
\end{define}

The Lagrange-d'Alembert-Euler integrator \eqref{eq: General Lagrange-d'Alembert Euler scheme} becomes parameter-dependent and can be applied independently for each fixed values of $\mu \in I$. Furthermore, $C^r$ convergence of first-order parameter-dependent one-step methods $x_{k+1} = \hat F_{t_{k+1},t_k}(x_k,\mu)$ for the parametric initial value problem

\begin{equation}
\dot x(t) = g(x(t),t,\mu), \qquad x(0)=x,
\end{equation}

\noindent
on compact subsets $K\subset T^*Q\times I$ can be established by applying Theorem~\ref{thm:Convergence of first-order one-step methods} to the augmented non-parametric system

\begin{equation}
\begin{aligned}
\left\{
\begin{aligned}
\dot x(t) &= g(x(t),t, \mu(t)), \\
\dot \mu(t) &=0,
\end{aligned}
\right.
\qquad
\text{with} \qquad
\left\{
\begin{aligned}
x(0) &= x, \\
\mu(0) &= \mu,
\end{aligned}
\right.
\end{aligned}
\end{equation}

\noindent
and the augmented first-order one-step scheme $(x_{k+1}, \mu_{k+1}) = (\hat F_{t_{k+1},t_k}(x_k,\mu_k), \mu_k)$.


\section{Generalized Forced Hamiltonian Neural Networks}
\label{sec: Generalized Forced Hamiltonian Neural Networks}

In this section, we consider autonomous forced Hamiltonian systems and propose a general structure-preserving neural network architecture designed to learn their flow.

\subsection{Autonomous forced Hamiltonian systems}
\label{sec: Autonomous forced Hamiltonian systems}

The evolution of an autonomous forced Hamiltonian system is governed by the differential equations

\begin{align}
\label{eq: Autonomous forced Hamiltonian system}
\dot q = \frac{\partial H}{\partial p}(q,p), \qquad\qquad \dot p = -\frac{\partial H}{\partial q}(q,p) + f(q,p),
\end{align}

\noindent
where $H: T^*Q \longrightarrow \mathbb{R}$ is the Hamiltonian function and $f_H: T^*Q \longrightarrow T^*Q$, $f_H(q,p) = (q, f(q,p))$ is the external forcing. For autonomous systems, neither $H$ nor $f$ explicitly depend on time. Consequently, the Lagrange-d'Alembert flow $F_{t,t_0}$ depends only on the difference $t-t_0$, and will therefore be denoted by $F_t$. Similarly, the generating function and exact discrete forces \eqref{eq: Generating function and exact discrete forces} depend on the difference $t_b-t_a$ rather than the general times $t_a$ and $t_b$. After fixing a time step $\Delta t$, the map $F_{\Delta t}$ is a Lagrange-d'Alembert map in the sense of Definition~\ref{thm: Definition of Lagrange-d'Alembert maps}. Our goal is to design a structure-preserving neural network that approximates $F_{\Delta t}$ using discrete data sampled from trajectories of the system~\eqref{eq: Autonomous forced Hamiltonian system}.

For separable and time-independent Hamiltonians $H(q,p)=T(p)+U(q)$, the Lagrange-d'Alembert-Euler integrator \eqref{eq: General Lagrange-d'Alembert Euler scheme} becomes explicit. In particular, setting $\Delta t=1$, we obtain mappings of the form

\begin{align}
\label{eq: Explicit Lagrange-d'Alembert Euler scheme}
q_{k+1} &= q_k + \frac{\partial T}{\partial p}(p_k), \nonumber \\
p_{k+1} &= p_k-\frac{\partial U}{\partial q}(q_{k+1})+ f(q_{k+1},p_k),
\end{align}

\noindent
which we refer to as \emph{Lagrange-d'Alembert-Euler maps}, and denote them by the shorthand $LDE_{T,U,f}(q_k,p_k)$. We further denote by

\begin{align}
\label{eq: Definition of LDE^r maps}
LDE^r(T^*Q) = \Big\{ LDE_{T,U,f} \,\Big |\, T, U\in C^{r+1}(\mathbb{R}^n), f \in C^r(\mathbb{R}^{2n},\mathbb{R}^n) \Big\}
\end{align}

\noindent
the set of all such mappings of class $C^r$. The Lagrange-d'Alembert-Euler map \eqref{eq: Explicit Lagrange-d'Alembert Euler scheme} will serve as a building block of our structure-preserving neural network.

\subsection{Structure-preserving neural network architecture}
\label{sec: Structure-preserving neural network architecture}

A structure-preserving neural network for autonomous forced Hamiltonian systems can be constructed by extending the architecture of Generalized Hamiltonian Neural Networks (GHNNs), which were proposed in \cite{HornKorenGHNN} and defined as a concatenation of symplectic integrators. For each stage of these symplectic integrators, the Hamiltonian is constrained to be separable, $H(q,p)=T(p)+U(q)$, and the kinetic and potential components, $T(p)$ and $U(q)$, are modeled by neural networks $\tilde T(p;\theta_1)$ and $\tilde U(q;\theta_2)$, respectively, that is,

\begin{align}
\label{eq:separable_Hamiltonian}
\tilde H(q, p; \theta_1,\theta_2) = \tilde T(p;\theta_1) + \tilde U(q;\theta_2),
\end{align}

\begin{figure}
	\begin{tikzpicture}[
    node distance=5mm and 8mm,
    state_box/.style={rectangle, rounded corners, minimum height=2cm, minimum width=1cm, very thick, text centered, font=\Large},
    red_box/.style={state_box, fill=mred!25, draw=mred!80!black},
    green_box/.style={state_box, fill=mgreen!50, draw=mgreen!60!black},
    blue_node/.style={circle, draw=mblue!50!black, fill=mblue!80, minimum size=5mm},
    special_node/.style={blue_node, fill=mblue!40, minimum size=6mm},
    green_bar/.style={rectangle, rounded corners, draw=mgreen!70!black, fill=mgreen!50, very thick, minimum height=1.8cm, minimum width=4mm},
    si_node/.style={rectangle, rounded corners, draw=mgreen!70!black, fill=mgreen!50, very thick, inner sep=2.5pt, font=\small\bfseries},
    nabla_node/.style={font=\huge}, 
    connector/.style={draw, thick},
    dot_connector/.style={draw, very thick, line cap=round}
]

\tikzset{
    pics/nn_module/.style={
        code={
            \foreach \row in {1,2,4,5} {
                \foreach \col in {1,2} {
                    \node[blue_node] (-node-\row-\col) at (-2+\col*0.8, 2.4-\row*0.8) {};
                }
            }
            \foreach \row in {1,2,4,5} {
                \draw[dot_connector] (-node-\row-1.east) -- (-node-\row-2.west);
                \fill (-node-\row-1.east) circle (1.2pt);
                \fill (-node-\row-2.west) circle (1.2pt);
            }
            \foreach \col in {1,2} {
                 \draw[dot_connector] (-node-2-\col.south) -- (-node-4-\col.north);
                 \draw[dot_connector] (-node-2-\col.north) -- (-node-1-\col.south);
                 \draw[dot_connector] (-node-5-\col.north) -- (-node-4-\col.south);
                 \fill (-node-2-\col.south) circle (1.2pt);
                 \fill (-node-4-\col.north) circle (1.2pt);
                 \fill (-node-2-\col.north) circle (1.2pt);
                 \fill (-node-1-\col.south) circle (1.2pt);
                 \fill (-node-4-\col.south) circle (1.2pt);
                 \fill (-node-5-\col.north) circle (1.2pt);
            }
            
            \node[special_node, fill=mgreen!50] (-H) at (.5, 0) {H};
            \node[special_node, above of=-H, yshift=+1.5em, xshift=-.2em] (-T) {T};
            \node[special_node, below of=-H, yshift=-1.5em, xshift=-.2em] (-U) {U};
            
            \draw[connector] (-node-2-2) -- (-T);
            \draw[connector] (-node-1-2) -- (-T);
            \draw[connector] (-node-4-2) -- (-U);
            \draw[connector] (-node-5-2) -- (-U);
            \draw[connector] (-H) -- (-U);
            \draw[connector] (-T) -- (-H);
        }
    }
}


\node[red_box] (state0) {$\begin{matrix}q_n \\ p_n\end{matrix}$};

\pic[right=1.5cm of state0] (nn1) {nn_module};
\node[green_bar, right=.5cm of nn1-H] (bar1) {$\nabla$};
\node[nabla_node, right=1mm of bar1] (nabla1) {};
\node[si_node, right=1mm of nabla1] (si1) {SI};
\node[green_box, right=of si1] (state1) {$\begin{matrix}\hat{q}_1 \\ \hat{p}_1\end{matrix}$};

\pic[right=1.5cm of state1] (nn2) {nn_module};
\node[green_bar, right=.5cm of nn2-H] (bar2) {$\nabla$};
\node[nabla_node, right=1mm of bar2] (nabla2) {};
\node[si_node, right=1mm of nabla2] (si2) {SI};
\node[green_box, right=of si2] (state2) {$\begin{matrix}\hat{q}_2 \\ \hat{p}_2\end{matrix}$};

\node[right=.3cm of state2, font=\Huge] (dots) {\dots};
\node[si_node, right=of dots, xshift=-2em] (si_final) {SI};
\node[red_box, right=of si_final] (state_final) {$\begin{matrix}q_{n+1} \\ p_{n+1}\end{matrix}$};

\draw[connector] ($(state0.north east)!0.5!(state0.east)$) -- (nn1-node-1-1.west);
\draw[connector] ($(state0.south east)!0.5!(state0.east)$) -- (nn1-node-5-1.west);

\draw[connector] (nn1-H) -- (bar1);
\draw[connector] (bar1) -- (si1); 
\coordinate (fork1) at ($(si1.east)!0.5!(state1.west)$);
\draw[connector] (si1.east) -- (fork1);
\coordinate (upper_coordinate1) at ($(state1.north west)!0.5!(state1.west)$);
\coordinate (lower_coordinate1) at ($(state1.south west)!0.5!(state1.west)$);
\draw[connector] (fork1) -- (upper_coordinate1);
\draw[connector] (fork1) -- (lower_coordinate1);

\draw[connector] ($(state1.north east)!0.5!(state1.east)$) -- (nn2-node-1-1.west);
\draw[connector] ($(state1.south east)!0.5!(state1.east)$) -- (nn2-node-5-1.west);

\draw[connector] (nn2-H) -- (bar2);
\draw[connector] (bar2) -- (si2); 
\coordinate (fork2) at ($(si2.east)!0.5!(state2.west)$);
\draw[connector] (si2.east) -- (fork2);
\coordinate (upper_coordinate2) at ($(state2.north west)!0.5!(state2.west)$);
\coordinate (lower_coordinate2) at ($(state2.south west)!0.5!(state2.west)$);
\draw[connector] (fork2) -- (upper_coordinate2);
\draw[connector] (fork2) -- (lower_coordinate2);

\draw[connector] (dots) -- (si_final);
\coordinate (fork_final) at ($(si_final.east)!0.5!(state_final.west)$);
\draw[connector] (si_final.east) -- (fork_final);
\coordinate (upper_coordinate_final) at ($(state_final.north west)!0.5!(state_final.west)$);
\coordinate (lower_coordinate_final) at ($(state_final.south west)!0.5!(state_final.west)$);
\draw[connector] (fork_final) -- (upper_coordinate_final);
\draw[connector] (fork_final) -- (lower_coordinate_final);

\end{tikzpicture}
	\caption{Schematic representation of a Generalized Hamiltonian Neural Network (GHNN). The architecture consists of a concatenation of symplectic integrator blocks, where each stage is parameterized by separable Hamiltonians with learned kinetic and potential energy components. This figure has been reconstructed to match \cite{HornKorenGHNN}.}
	\label{fig:GHNN}
\end{figure}
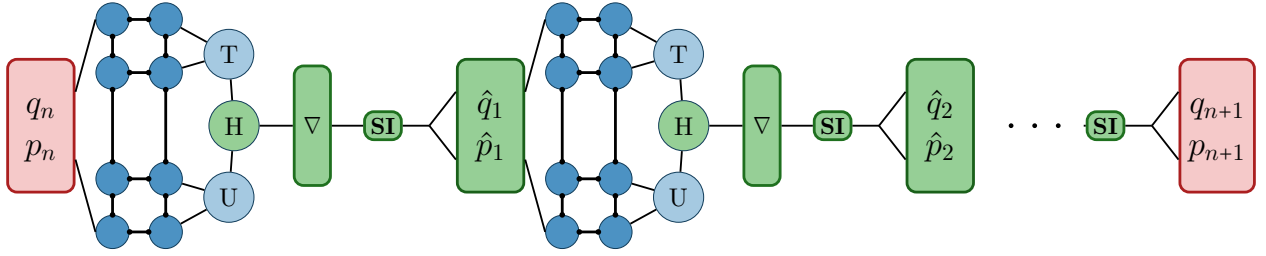

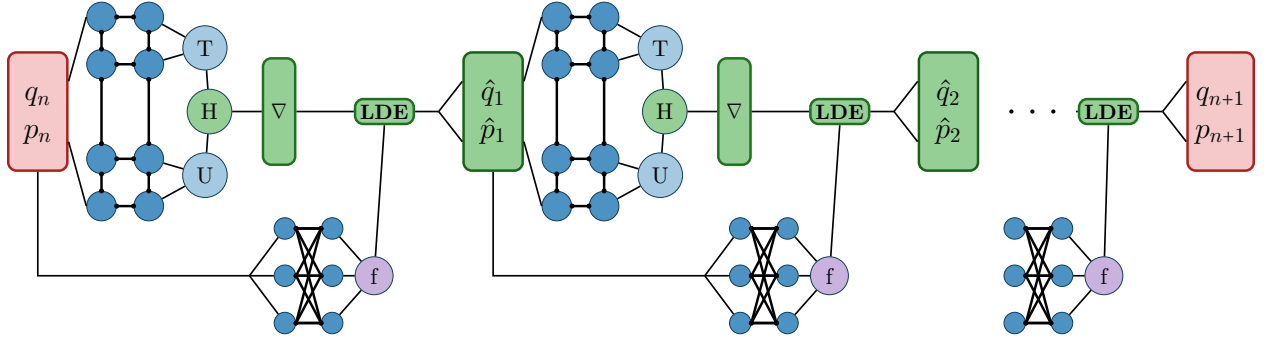
\begin{figure}
	\begin{tikzpicture}[
    node distance=5mm and 8mm,
    state_box/.style={rectangle, rounded corners, minimum height=2cm, minimum width=1cm, very thick, text centered, font=\Large},
    red_box/.style={state_box, fill=mred!25, draw=mred!80!black},
    red_small_box/.style={red_box, minimum height=.4cm},
    green_box/.style={state_box, fill=mgreen!50, draw=mgreen!60!black},
    green_small_box/.style={green_box, minimum height=.4cm},
    blue_node/.style={circle, draw=mblue!50!black, fill=mblue!80, minimum size=5mm},
    blue_small_node/.style={circle, draw=mblue!50!black, fill=mblue!80, minimum size=2mm},
    special_node/.style={blue_node, fill=mblue!40, minimum size=6mm},
    green_bar/.style={rectangle, rounded corners, draw=mgreen!70!black, fill=mgreen!50, very thick, minimum height=1.8cm, minimum width=4mm},
    si_node/.style={rectangle, rounded corners, draw=mgreen!70!black, fill=mgreen!50, very thick, inner sep=2.5pt, font=\small\bfseries},
    nabla_node/.style={font=\huge}, 
    connector/.style={draw, thick},
    arrow/.style={draw, thick, -stealth},
    dot_connector/.style={draw, very thick, line cap=round}
]

\tikzset{
    pics/nn_module/.style={
        code={
            \foreach \row in {1,2,4,5} {
                \foreach \col in {1,2} {
                    \node[blue_node] (-node-\row-\col) at (-2+\col*0.8, 2.4-\row*0.8) {};
                }
            }
            \foreach \row in {1,2,4,5} {
                \draw[dot_connector] (-node-\row-1.east) -- (-node-\row-2.west);
                \fill (-node-\row-1.east) circle (1.2pt);
                \fill (-node-\row-2.west) circle (1.2pt);
            }
            \foreach \col in {1,2} {
                 \draw[dot_connector] (-node-2-\col.south) -- (-node-4-\col.north);
                 \draw[dot_connector] (-node-2-\col.north) -- (-node-1-\col.south);
                 \draw[dot_connector] (-node-5-\col.north) -- (-node-4-\col.south);
                 \fill (-node-2-\col.south) circle (1.2pt);
                 \fill (-node-4-\col.north) circle (1.2pt);
                 \fill (-node-2-\col.north) circle (1.2pt);
                 \fill (-node-1-\col.south) circle (1.2pt);
                 \fill (-node-4-\col.south) circle (1.2pt);
                 \fill (-node-5-\col.north) circle (1.2pt);
            }

            \node[special_node, fill=mgreen!50] (-H) at (.5, 0) {H};
            \node[special_node, above of=-H, yshift=+1.5em, xshift=-.2em] (-T) {T};
            \node[special_node, below of=-H, yshift=-1.5em, xshift=-.2em] (-U) {U};

            \draw[connector] (-node-2-2) -- (-T);
            \draw[connector] (-node-1-2) -- (-T);
            \draw[connector] (-node-4-2) -- (-U);
            \draw[connector] (-node-5-2) -- (-U);
            \draw[connector] (-H) -- (-U);
            \draw[connector] (-T) -- (-H);
        }
    },
    pics/fnn_module/.style={
        code={
            \foreach \row in {1,2,3} {
                \foreach \col in {1,2} {
                    \node[blue_small_node] (-node-\row-\col) at (-2+\col*0.8, 2.4-\row*0.8) {};
                }
            }
            \foreach \row in {1,2,3} {
                \draw[dot_connector] (-node-\row-1.east) -- (-node-\row-2.west);
                \fill (-node-\row-1.east) circle (1.2pt);
                \fill (-node-\row-2.west) circle (1.2pt);
                \foreach \rowt in {1,2,3} {
                \draw[dot_connector] (-node-\row-1.east) -- (-node-\rowt-2.west);
                }
            }
            \fill (-node-1-1.east) circle (1.2pt);
            \fill (-node-2-1.east) circle (1.2pt);
            \fill (-node-3-1.east) circle (1.2pt);
            \fill (-node-1-2.west) circle (1.2pt);
            \fill (-node-2-2.west) circle (1.2pt);
            \fill (-node-3-2.west) circle (1.2pt);

            \node[special_node, fill=mpurple!50, right=.5em of -node-2-2] (-f) {f};

            \draw[connector] (-node-1-2) -- (-f);
            \draw[connector] (-node-2-2) -- (-f);
            \draw[connector] (-node-3-2) -- (-f);
        }
    }
}


\node[red_box] (state0) {$\begin{matrix}q_n \\ p_n\end{matrix}$};

\pic[right=1.5cm of state0] (nn1) {nn_module};
\node[green_bar, right=.5cm of nn1-H] (bar1) {$\nabla$};
\node[nabla_node, right=1mm of bar1] (nabla1) {};
\node[si_node, right=6mm of nabla1] (si1) {LDE};
\node[green_box, right=of si1] (state1) {$\begin{matrix}\hat{q}_1 \\ \hat{p}_1\end{matrix}$};
\pic[below of=nabla1, yshift=-8em, xshift=2em] (fnn1) {fnn_module};
\coordinate (fnn_fork1) at ($(fnn1-node-2-1)-(.6,0)$);

\pic[right=1.5cm of state1] (nn2) {nn_module};
\node[green_bar, right=.5cm of nn2-H] (bar2) {$\nabla$};
\node[nabla_node, right=1mm of bar2] (nabla2) {};
\node[si_node, right=6mm of nabla2] (si2) {LDE};
\node[green_box, right=of si2] (state2) {$\begin{matrix}\hat{q}_2 \\ \hat{p}_2\end{matrix}$};
\pic[below of=nabla2, yshift=-8em, xshift=2em] (fnn2) {fnn_module};
\coordinate (fnn_fork2) at ($(fnn2-node-2-1)-(.6,0)$);

\node[right=.3cm of state2, font=\Huge] (dots) {\dots};
\node[si_node, right=of dots, xshift=-2em] (si_final) {LDE};
\node[red_box, right=of si_final] (state_final) {$\begin{matrix}q_{n+1} \\ p_{n+1}\end{matrix}$};
\pic[below of=si_final, yshift=-8em, xshift=-1em] (fnn_final) {fnn_module};
\coordinate (fnn_fork_final) at ($(fnn_final-node-2-1)-(.6,0)$);

\draw[connector] ($(state0.north east)!0.5!(state0.east)$) -- (nn1-node-1-1.west);
\draw[connector] ($(state0.south east)!0.5!(state0.east)$) -- (nn1-node-5-1.west);

\draw[connector] (nn1-H) -- (bar1);
\draw[connector] (bar1) -- (si1);
\coordinate (fork1) at ($(si1.east)!0.5!(state1.west)$);
\draw[connector] (si1.east) -- (fork1);
\coordinate (upper_coordinate1) at ($(state1.north west)!0.5!(state1.west)$);
\coordinate (lower_coordinate1) at ($(state1.south west)!0.5!(state1.west)$);
\draw[connector] (fork1) -- (upper_coordinate1);
\draw[connector] (fork1) -- (lower_coordinate1);
\draw[connector] (fnn1-f) -- (si1.south);
\draw[connector] (state0.south) |- (fnn_fork1);
\draw[connector] (fnn_fork1) -- (fnn1-node-1-1);
\draw[connector] (fnn_fork1) -- (fnn1-node-2-1);
\draw[connector] (fnn_fork1) -- (fnn1-node-3-1);

\draw[connector] ($(state1.north east)!0.5!(state1.east)$) -- (nn2-node-1-1.west);
\draw[connector] ($(state1.south east)!0.5!(state1.east)$) -- (nn2-node-5-1.west);

\draw[connector] (nn2-H) -- (bar2);
\draw[connector] (bar2) -- (si2);
\coordinate (fork2) at ($(si2.east)!0.5!(state2.west)$);
\draw[connector] (si2.east) -- (fork2);
\coordinate (upper_coordinate2) at ($(state2.north west)!0.5!(state2.west)$);
\coordinate (lower_coordinate2) at ($(state2.south west)!0.5!(state2.west)$);
\draw[connector] (fork2) -- (upper_coordinate2);
\draw[connector] (fork2) -- (lower_coordinate2);
\draw[connector] (fnn2-f) -- (si2.south);
\draw[connector] (state1.south) |- (fnn_fork2);
\draw[connector] (fnn_fork2) -- (fnn2-node-1-1);
\draw[connector] (fnn_fork2) -- (fnn2-node-2-1);
\draw[connector] (fnn_fork2) -- (fnn2-node-3-1);

\draw[connector] (dots) -- (si_final);
\coordinate (fork_final) at ($(si_final.east)!0.5!(state_final.west)$);
\draw[connector] (si_final.east) -- (fork_final);
\coordinate (upper_coordinate_final) at ($(state_final.north west)!0.5!(state_final.west)$);
\coordinate (lower_coordinate_final) at ($(state_final.south west)!0.5!(state_final.west)$);
\draw[connector] (fork_final) -- (upper_coordinate_final);
\draw[connector] (fork_final) -- (lower_coordinate_final);
\draw[connector] (fnn_final-f) -- (si_final.south);

\end{tikzpicture}
	\caption{Schematic representation of a Generalized Forced Hamiltonian Neural Network (GFHNN), extending the GHNN architecture (see \Cref{fig:GHNN}) to incorporate external forcing. The model combines learned Hamiltonian dynamics with additional terms representing non-conservative forces.}
	\label{fig:GFHNN}
\end{figure}

\noindent
where $\theta_1$ and $\theta_2$ denote the trainable parameters of $\tilde T$ and $\tilde U$, respectively. Throughout this work, neural network approximations are denoted by a tilde, and the dependence on trainable parameters is omitted whenever it is not relevant. In the simplest case, $\tilde T$ and $\tilde U$ are modeled using multilayer perceptrons \cite{goodfellow2016deep}. GHNNs are visualized in \Cref{fig:GHNN}. In this sense, GHNNs provide a unifying perspective encompassing other structure-preserving architectures, including SympNets \cite{JinZhangZhuTangKarniadakis2020} and H\'enonNets \cite{BurbyHenonNets2020}.

We propose a Generalized Forced Hamiltonian Neural Network (GFHNN) architecture, defined as a concatenation of Lagrange-d'Alembert-Euler maps \eqref{eq: Explicit Lagrange-d'Alembert Euler scheme}. Each map includes a separable Hamiltonian and a forcing term, where the kinetic and potential energies, as well as the forcing term, are parametrized by multilayer perceptrons. Using notation similar to that in \cite{HornKorenGHNN}, a GFHNN can be expressed as

\begin{align}
\label{eq: GFHNN as a concatenation}
GFHNN(q,p;\theta) &= LDE_{\tilde T_m(\cdot;\vartheta_m), \tilde U_m(\cdot;\psi_m), \tilde f_m(\cdot,\cdot;\chi_m)}\circ \ldots \circ LDE_{\tilde T_1(\cdot;\vartheta_1), \tilde U_1(\cdot;\psi_1), \tilde f_1(\cdot,\cdot;\chi_1)} (q,p),
\end{align}

\noindent
where $\theta=(\vartheta_1,\psi_1,\chi_1,\vartheta_2,\psi_2,\chi_2,\ldots)$ collects all the trainable parameters $\vartheta_i$, $\psi_i$, $\chi_i$ of the neural networks $\tilde T_i$, $\tilde U_i$, $\tilde f_i$, respectively, for $i=1,\ldots,m$. A schematic representation of a GFHNN is shown in \Cref{fig:GFHNN}. Given a set of training data $((q^a_i,p^a_i),(q^b_i,p^b_i))$ for $i=1,\ldots,d$ such that $(q^b_i,p^b_i)=F_{\Delta t}(q^a_i,p^a_i)$, the neural network \eqref{eq: GFHNN as a concatenation} can be trained by minimizing the mean squared loss,

\begin{align}
\label{eq:Loss function for GFHNN}
\text{Loss}(\theta) = \frac{1}{2nd}\sum_{i=1}^d \big\| (q^b_i,p^b_i) - GFHNN(q^a_i,p^a_i;\theta) \big\|^2,
\end{align}

\noindent
using a suitable optimization algorithm.

\subsection{Universal approximation theorem for GFHNNs}
\label{sec: Universal approximation theorem for GFHNNs}

We now establish a universal approximation theorem for the GFHNN architecture, demonstrating that it can approximate the flow map $F_{\Delta t}$ arbitrarily well in the strong $C^r$ topology on compact sets. We begin by showing the following two propositions.

\begin{prop}
\label{thm:Proposition on approximating phi by Lie-Trotter}
Given a Lagrange-d'Alembert map $\varphi \in LdA^{r+1}(T^*Q)$ for some integer $r\ge 0$, there exists a time-dependent forced Hamiltonian system satisfying Assumption~\ref{ass:standing} for some $T>0$ such that for every compact set $K \subset T^*Q$ and every $\epsilon > 0$, there exists $N\in\mathbb{N}$ such that

\begin{equation}
\label{eq:Approximation of phi by Lie-Trotter}
\left\| \varphi - \Phi^{LT}_{t_N, t_{N-1}} \circ \cdots \circ \Phi^{LT}_{t_{1}, t_0} \right\|_{C^r(K)} < \epsilon,
\end{equation}

\noindent
where $\Phi^{LT}_{t_{k+1}, t_k}$ denotes the Lie-Trotter integrator \eqref{eq:Lie-Trotter splitting method} with $t_k=k\Delta t$ and $\Delta t = T/N$.
\end{prop}

\begin{proof}
By Definition~\ref{thm: Definition of Lagrange-d'Alembert maps}, we have $\varphi = F_{T,0}$, where $F_{t,t_0}$ is the flow of a time-dependent forced Hamiltonian system satisfying Assumption~\ref{ass:standing}. The Lie-Trotter integrator \eqref{eq:Lie-Trotter splitting method} is first-order accurate (see, e.g., \cite{HLWGeometric}), and as noted in Section~\ref{sec: Lagrange-d'Alembert integrators}, it satisfies the assumptions of Theorem~\ref{thm:Convergence of first-order one-step methods}. Let $\epsilon > 0$ and let $K \subset T^*Q$ be compact. Then the estimate \eqref{eq:Global error estimate of a one-step method} holds. Define

\begin{equation}
N\coloneq\bigg\lceil \frac{C_K T}{\epsilon} \bigg\rceil +1 \qquad \text{and} \qquad \Delta t \coloneq \frac{T}{N}.
\end{equation}

\noindent
Then the estimate \eqref{eq:Global error estimate of a one-step method} implies \eqref{eq:Approximation of phi by Lie-Trotter}.\\
\end{proof}

\begin{prop}
\label{thm:Proposition on approximating phi by LDE}
Given a Lagrange-d'Alembert map $\varphi \in LdA^{r+1}(T^*Q)$ for some integer $r\ge 0$, for every compact set $K \subset T^*Q$ and every $\epsilon > 0$, there exists a finite sequence of maps $LDE_{T_i,U_i,f_i} \in LDE^{r+1}(T^*Q)$, $i = 1,\ldots,m$, such that

\begin{equation}
\label{eq:Approximating phi by LDE}
\left\| \varphi - LDE_{T_m,U_m,f_m} \circ \cdots \circ LDE_{T_1,U_1,f_1} \right\|_{C^r(K)} < \epsilon.
\end{equation}
\end{prop}

\begin{proof}
Let $\varphi \in LdA^{r+1}(T^*Q)$. By Proposition~\ref{thm:Proposition on approximating phi by Lie-Trotter}, $\varphi$ can be approximated on compact sets with arbitrary accuracy by finite compositions of Lie-Trotter integrators $\Phi^{LT}_{t_{k+1}, t_k}= \Phi^{f}_{t_{k+1}, t_k} \circ\Phi^{H}_{t_{k+1}, t_k}$, as in \eqref{eq:Approximation of phi by Lie-Trotter}. Note that

\begin{equation}
\Phi^{f}_{t_{k+1}, t_k} = LDE_{0,0,\Delta t f(\cdot,\cdot,t_k)}\in LDE^{r+1}(T^*Q),
\end{equation}

\noindent
therefore it suffices to show that the symplectic Euler integrators $\Phi^{H}_{t_{k+1}, t_k}$ can be approximated on compact sets with arbitrary accuracy by finite compositions of maps in $LDE^{r+1}(T^*Q)$. Since $\Phi^{H}_{t_{k+1}, t_k}$ is symplectic and of class $C^{r+1}$, \cite[Theorem~2]{Turaev2002} (see also \cite[Lemma~4]{JinZhangZhuTangKarniadakis2020}) implies that, for every compact set $K \subset T^*Q$ and every $\epsilon > 0$ there exist functions $V_1,\ldots,V_{m'} \in C^{r+2}(\mathbb{R}^n)$ such that

\begin{equation}
 \big\|\Phi^{H}_{t_{k+1}, t_k}-\mathcal{H}_{V_{m'}}\circ \ldots \circ \mathcal{H}_{V_1}\big\|_{C^r(K)} < \epsilon,
\end{equation}

\noindent
where

\begin{equation}
\label{eq:Henon-like map}
\mathcal{H}_{V}(q,p) = \bigg(-p+\frac{\partial V}{\partial q}(q), q\bigg)
\end{equation}

\noindent
is a H\'{e}non-like map. Finally, we observe that a H\'{e}non-like map can be represented exactly as a composition of Lagrange-d'Alembert-Euler maps \cite{JinZhangZhuTangKarniadakis2020},

\begin{equation}
\mathcal{H}_V = LDE_{V,0,0}\circ LDE_{0,-\frac{1}{2}\|q\|^2,0}\circ LDE_{-\frac{1}{2}\|p\|^2,0,0}\circ LDE_{0,-\frac{1}{2}\|q\|^2,0},
\end{equation}

\noindent
which completes the proof.\\
\end{proof}

Proposition~\ref{thm:Proposition on approximating phi by LDE} justifies why using Lagrange-d'Alembert-Euler maps in the architecture \eqref{eq: GFHNN as a concatenation} is a natural choice. It remains to show that parametrizing the kinetic $T$, potential $U$, and forcing $f$ terms with neural networks $\tilde T$, $\tilde U$, and $\tilde f$, respectively, allows one to approximate any map in $LDE^r(T^*Q)$. We denote by $\mathcal{N}^r_\sigma(\mathbb{R}^n,\mathbb{R}^m)$ the set of all multilayer perceptrons of class $C^r$, that is,

\begin{equation}
\label{eq:Multilayer perceptrons}
\mathcal{N}^r_\sigma(\mathbb{R}^n,\mathbb{R}^m):=\Big\{ \tilde f : \mathbb{R}^n \longrightarrow \mathbb{R}^m \;\Big|\; \tilde f \text{ is a multilayer perceptron with activation } \sigma \in \mathcal{A}^r \Big\},
\end{equation}

\noindent
where $\mathcal{A}^r$ denotes the set of all activation functions $\sigma : \mathbb{R} \longrightarrow \mathbb{R}$ of class $C^r$, which are bounded and non-constant. As shown in \cite{Hornik1991} (see also \cite{HornikStinchcombeWhite1989,HornikStinchcombeWhite1990}), multilayer perceptrons are universal approximators of $C^r$ functions. In fact, the subset of $\mathcal{N}^r_\sigma(\mathbb{R}^n,\mathbb{R}^m)$ consisting of neural networks with a single hidden layer is already uniformly $r$-dense on compacta in $C^r(\mathbb{R}^n,\mathbb{R}^m)$. We further define

\begin{equation}
\label{eq:Definition of the space of LDE^r neural networks}
LDE^r_\mathcal{N_\sigma}(T^*Q)= \Big\{ LDE_{\tilde T, \tilde U, \tilde f} \; \Big| \; \tilde T, \tilde U \in \mathcal{N}^{r+1}_\sigma(\mathbb{R}^n,\mathbb{R}) \text{ and } \tilde f \in \mathcal{N}^{r}_\sigma(\mathbb{R}^{2n},\mathbb{R}^n) \Big\},
\end{equation}

\noindent
where, for simplicity, we assume that the same activation function $\sigma \in \mathcal{A}^{r+1}$ is used for all three neural networks $\tilde T$, $\tilde U$, and $\tilde f$. We now prove the following proposition.

\begin{prop}
\label{thm:Proposition on approximating LDE by neural networks}
The set $LDE^r_\mathcal{N_\sigma}(T^*Q)$ is uniformly $r$-dense on compacta in $LDE^r(T^*Q)$ for all integer $r\ge 0$. That is, for every Lagrange-d'Alembert-Euler map $LDE_{T,U,f} \in LDE^{r}(T^*Q)$, for every compact set $K \subset T^*Q$, and every $\epsilon > 0$, there exists a map $LDE_{\tilde T, \tilde U, \tilde f} \in LDE^r_\mathcal{N_\sigma}(T^*Q)$ such that

\begin{equation}
\label{eq:Approximating LDE by neural networks}
\left\| LDE_{T,U,f} - LDE_{\tilde T, \tilde U, \tilde f} \right\|_{C^r(K)} < \epsilon.
\end{equation}
\end{prop}

\begin{proof}
We note that

\begin{equation}
LDE_{T,U,f} = LDE_{T,0,0} \circ LDE_{0,U,f},
\end{equation}

\noindent
therefore it suffices to prove the proposition separately for Lagrange-d'Alembert-Euler maps of the forms $LDE_{T,0,0}$ and $LDE_{0,U,f}$. First, consider $LDE_{T,0,0}$ for $T\in C^{r+1}(\mathbb{R}^n)$. Let $\epsilon > 0$ and let $K \subset T^*Q$ be compact. The set

\begin{equation}
K_p = \{ p \in \mathbb{R}^n \,|\, \exists q \in \mathbb{R}^n: (q,p) \in K \}
\end{equation}

\noindent
is the projection of $K$ onto the $p$-coordinates and is therefore compact in $\mathbb{R}^n$. By Theorem~3 in \cite{Hornik1991}, there exists $\tilde T \in \mathcal{N}^{r+1}_\sigma(\mathbb{R}^n,\mathbb{R})$ such that

\begin{equation}
\left\| T - \tilde T \right\|_{C^{r+1}(K_p)} < \epsilon.
\end{equation}

\noindent
We will show that $LDE_{\tilde T, 0, 0}\in LDE^r_\mathcal{N_\sigma}(T^*Q) $ approximates $LDE_{T, 0, 0}$ with the required accuracy. Indeed, by \Cref{eq: Explicit Lagrange-d'Alembert Euler scheme} and the properties of the $C^r$ norm,

\begin{equation}
\left\| LDE_{T,0,0} - LDE_{\tilde T, 0,0} \right\|_{C^r(K)} = \left\| \frac{\partial T}{\partial p} - \frac{\partial \tilde T}{\partial p} \right\|_{C^r(K_p)} \le \left\| T - \tilde T \right\|_{C^{r+1}(K_p)} < \epsilon.
\end{equation}

\noindent
Next, consider now $LDE_{0,U,f}$ for $U\in C^{r+1}(\mathbb{R}^n)$ and $f \in C^r(\mathbb{R}^{2n},\mathbb{R}^n)$. Let $\epsilon > 0$ and let $K \subset T^*Q$ be compact. The set

\begin{equation}
K_q = \{ q \in \mathbb{R}^n \,|\, \exists p \in \mathbb{R}^n: (q,p) \in K \}
\end{equation}

\noindent
is the projection of $K$ onto the $q$-coordinates and is therefore compact in $\mathbb{R}^n$. By Theorem~3 in \cite{Hornik1991}, there exist $\tilde U \in \mathcal{N}^{r+1}_\sigma(\mathbb{R}^n,\mathbb{R})$ and $\tilde f \in \mathcal{N}^{r}_\sigma(\mathbb{R}^{2n},\mathbb{R}^n)$ such that

\begin{equation}
\left\| U - \tilde U \right\|_{C^{r+1}(K_q)} < \frac{\epsilon}{2} \qquad \text{and} \qquad \left\| f - \tilde f \right\|_{C^{r}(K)} < \frac{\epsilon}{2}.
\end{equation}

\noindent
We will show that $LDE_{0,\tilde U, \tilde f}\in LDE^r_\mathcal{N_\sigma}(T^*Q) $ approximates $LDE_{0, U, f}$ with the required accuracy. We have

\begin{align}
\left\| LDE_{0, U, f} - LDE_{0,\tilde U, \tilde f} \right\|_{C^r(K)} &= \left\| - \bigg(\frac{\partial U}{\partial q} - \frac{\partial \tilde U}{\partial q} \bigg) + f - \tilde f \right\|_{C^r(K)} \nonumber \\
&\le \left\| \frac{\partial U}{\partial q} - \frac{\partial \tilde U}{\partial q} \right\|_{C^r(K_q)} + \left\| f - \tilde f \right\|_{C^r(K)} \\
&\le \left\| U - \tilde U \right\|_{C^{r+1}(K_q)} + \left\| f - \tilde f \right\|_{C^r(K)} \nonumber \\
&< \frac{\epsilon}{2} + \frac{\epsilon}{2} = \epsilon, \nonumber
\end{align}

\noindent
which completes the proof.\\
\end{proof}

The following universal approximation theorem follows directly from Proposition~\ref{thm:Proposition on approximating phi by LDE} and Proposition~\ref{thm:Proposition on approximating LDE by neural networks}.

\begin{thm}[{\bf Universal approximation theorem for GFHNNs}]
\label{thm:Universal approximation theorem for GFHNNs}
Let $r\ge 0$ be an integer and let $\varphi \in LdA^{r+1}(T^*Q)$. Then, for every compact set $K \subset T^*Q$ and every $\epsilon > 0$, there exists a finite sequence of maps $LDE_{\tilde T_i, \tilde U_i, \tilde f_i} \in LDE^{r+1}_{\mathcal{N}_\sigma}(T^*Q)$, $i = 1,\ldots,m$, such that

\begin{equation}
\label{eq:Approximating phi by LDE neural networks}
\left\| \varphi - LDE_{\tilde T_m, \tilde U_m, \tilde f_m}\circ \ldots \circ LDE_{\tilde T_1, \tilde U_1, \tilde f_1} \right\|_{C^r(K)} < \epsilon.
\end{equation}
\end{thm}


\section{Parametric Generalized Forced Hamiltonian Neural Networks}
\label{sec: Parametric Generalized Forced Hamiltonian Neural Networks}

In this section we consider parametric autonomous forced Hamiltonian systems and propose a general structure-preserving neural network architecture designed to learn their parameter-dependent flow.

\subsection{Parametric autonomous forced Hamiltonian systems}
\label{eq: Parametric autonomous forced Hamiltonian systems}

The evolution of a parametric autonomous forced Hamiltonian system is governed by the differential equations \eqref{eq: Parametric forced Hamiltonian system}, where the parameter-dependent Hamiltonian $H: T^*Q \times I \longrightarrow \mathbb{R}$ and external forcing $f_H: T^*Q \times I \longrightarrow T^*Q$, given by $f_H(q,p,\mu) = (q, f(q,p,\mu))$, do not depend explicitly on time. Consequently, the parametric Lagrange-d'Alembert flow $F_{t,t_0}$ depends only on the difference $t-t_0$, and will therefore be denoted by $F_t$. For a fixed value of $\mu$, numerical integration of the system \eqref{eq: Parametric forced Hamiltonian system} can be performed by using the scheme \eqref{eq: General Lagrange-d'Alembert Euler scheme}. The Lagrange-d'Alembert-Euler map \eqref{eq: Explicit Lagrange-d'Alembert Euler scheme} becomes parameter-dependent through the parameter-dependent functions $T=T(q,p,\mu)$, $U=U(q,p,\mu)$, and $f=f(q,p,\mu)$. Let us introduce the augmented Lagrange-d'Alembert-Euler map,

\begin{equation}
\label{eq: Augmented LDE map}
LDE^\mathrm{aug}_{T,U,f}(q,p,\mu) = \Big(LDE_{T(\cdot,\mu),U(\cdot,\mu),f(\cdot,\cdot,\mu)}(q,p),\mu\Big),
\end{equation}

\noindent
which applies the Lagrange-d'Alembert-Euler map corresponding to the parameter value $\mu$ to the $(q,p)$ variables and leaves the parameter $\mu$ unchanged. Analogously to \eqref{eq: Definition of LDE^r maps}, we define the set of all augmented Lagrange-d'Alembert-Euler maps of class $C^r$ as

\begin{align}
\label{eq: Definition of parametric LDE^r maps}
LDE^r(T^*Q, I) = \Big\{ LDE^\mathrm{aug}_{T,U,f} \,\Big |\, T, U\in C^{r+1}(\mathbb{R}^n\times I), f \in C^r(\mathbb{R}^{2n}\times I,\mathbb{R}^n) \Big\}.
\end{align}

\subsection{Parametric structure-preserving neural network architecture}
\label{sec: Parametric structure-preserving neural network architecture}

A class of neural networks called Parametric Generalized Hamiltonian Neural Networks (PGHNNs), suitable for learning parametric canonical Hamiltonian systems, was introduced in \cite{Horn2026PHD, HornKorenPGHNN} as an extension of GHNNs (see Figure~\ref{fig:PGHNN}). In the same spirit, we propose Parametric Generalized Forced Hamiltonian Neural Networks (PGFHNNs), which allow parameter dependence in the architecture of GFHNNs introduced in Section~\ref{sec: Structure-preserving neural network architecture}. We define a PGFHNN as a concatenation of augmented Lagrange-d'Alembert-Euler maps \eqref{eq: Augmented LDE map} in which the kinetic and potential energies, together with the forcing term, are represented by multilayer perceptrons that take as input both the dynamic variables $q$ and $p$, as well as the parameter $\mu$. Using notation similar to that in \cite{Horn2026PHD, HornKorenPGHNN}, a PGFHNN can be expressed as

\begin{align}
\label{eq: PGFHNN as a concatenation}
PGFHNN(q,p,\mu;\theta) &= LDE^\mathrm{aug}_{\tilde T_m(\cdot,\cdot;\vartheta_m), \tilde U_m(\cdot,\cdot;\psi_m), \tilde f_m(\cdot,\cdot,\cdot;\chi_m)}\circ \ldots \circ LDE^\mathrm{aug}_{\tilde T_1(\cdot,\cdot;\vartheta_1), \tilde U_1(\cdot,\cdot;\psi_1), \tilde f_1(\cdot,\cdot,\cdot;\chi_1)} (q,p,\mu),
\end{align}

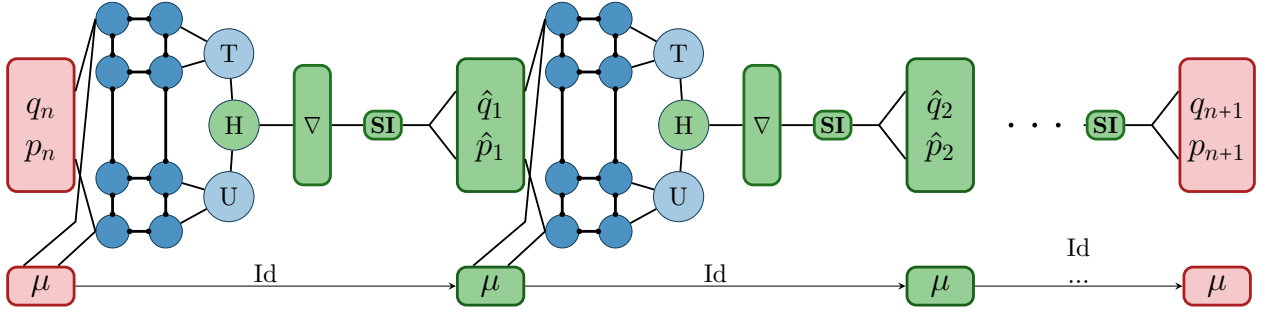
\begin{figure}
	\begin{tikzpicture}[
    node distance=5mm and 8mm,
    state_box/.style={rectangle, rounded corners, minimum height=2cm, minimum width=1cm, very thick, text centered, font=\Large},
    red_box/.style={state_box, fill=mred!25, draw=mred!80!black},
    red_small_box/.style={red_box, minimum height=.4cm},
    green_box/.style={state_box, fill=mgreen!50, draw=mgreen!60!black},
    green_small_box/.style={green_box, minimum height=.4cm},
    blue_node/.style={circle, draw=mblue!50!black, fill=mblue!80, minimum size=5mm},
    special_node/.style={blue_node, fill=mblue!40, minimum size=6mm},
    green_bar/.style={rectangle, rounded corners, draw=mgreen!70!black, fill=mgreen!50, very thick, minimum height=1.8cm, minimum width=4mm},
    si_node/.style={rectangle, rounded corners, draw=mgreen!70!black, fill=mgreen!50, very thick, inner sep=2.5pt, font=\small\bfseries},
    nabla_node/.style={font=\huge}, 
    connector/.style={draw, thick},
    dot_connector/.style={draw, very thick, line cap=round}
]

\tikzset{
    pics/nn_module/.style={
        code={
            \foreach \row in {1,2,4,5} {
                \foreach \col in {1,2} {
                    \node[blue_node] (-node-\row-\col) at (-2+\col*0.8, 2.4-\row*0.8) {};
                }
            }
            \foreach \row in {1,2,4,5} {
                \draw[dot_connector] (-node-\row-1.east) -- (-node-\row-2.west);
                \fill (-node-\row-1.east) circle (1.2pt);
                \fill (-node-\row-2.west) circle (1.2pt);
            }
            \foreach \col in {1,2} {
                 \draw[dot_connector] (-node-2-\col.south) -- (-node-4-\col.north);
                 \draw[dot_connector] (-node-2-\col.north) -- (-node-1-\col.south);
                 \draw[dot_connector] (-node-5-\col.north) -- (-node-4-\col.south);
                 \fill (-node-2-\col.south) circle (1.2pt);
                 \fill (-node-4-\col.north) circle (1.2pt);
                 \fill (-node-2-\col.north) circle (1.2pt);
                 \fill (-node-1-\col.south) circle (1.2pt);
                 \fill (-node-4-\col.south) circle (1.2pt);
                 \fill (-node-5-\col.north) circle (1.2pt);
            }
            
            \node[special_node, fill=mgreen!50] (-H) at (.5, 0) {H};
            \node[special_node, above of=-H, yshift=+1.5em, xshift=-.2em] (-T) {T};
            \node[special_node, below of=-H, yshift=-1.5em, xshift=-.2em] (-U) {U};
            
            \draw[connector] (-node-2-2) -- (-T);
            \draw[connector] (-node-1-2) -- (-T);
            \draw[connector] (-node-4-2) -- (-U);
            \draw[connector] (-node-5-2) -- (-U);
            \draw[connector] (-H) -- (-U);
            \draw[connector] (-T) -- (-H);
        }
    }
}


\node[red_box] (state0) {$\begin{matrix}q_n \\ p_n\end{matrix}$};
\node[red_small_box, below of=state0, yshift=-5em] (mu0) {$\mu$};

\pic[right=1.5cm of state0] (nn1) {nn_module};
\node[green_bar, right=.5cm of nn1-H] (bar1) {$\nabla$};
\node[nabla_node, right=1mm of bar1] (nabla1) {};
\node[si_node, right=1mm of nabla1] (si1) {SI};
\node[green_box, right=of si1] (state1) {$\begin{matrix}\hat{q}_1 \\ \hat{p}_1\end{matrix}$};
\node[green_small_box, below of=state1, yshift=-5em] (mu1) {$\mu$};

\pic[right=1.5cm of state1] (nn2) {nn_module};
\node[green_bar, right=.5cm of nn2-H] (bar2) {$\nabla$};
\node[nabla_node, right=1mm of bar2] (nabla2) {};
\node[si_node, right=1mm of nabla2] (si2) {SI};
\node[green_box, right=of si2] (state2) {$\begin{matrix}\hat{q}_2 \\ \hat{p}_2\end{matrix}$};
\node[green_small_box, below of=state2, yshift=-5em] (mu2) {$\mu$};

\node[right=.3cm of state2, font=\Huge] (dots) {\dots};
\node[si_node, right=of dots, xshift=-2em] (si_final) {SI};
\node[red_box, right=of si_final] (state_final) {$\begin{matrix}q_{n+1} \\ p_{n+1}\end{matrix}$};
\node[red_small_box, below of=state_final, yshift=-5em] (mu_final) {$\mu$};

\draw[connector] ($(state0.north east)!0.5!(state0.east)$) -- (nn1-node-1-1.west);
\draw[connector] ($(state0.south east)!0.5!(state0.east)$) -- (nn1-node-5-1.west);
\draw[connector] ($(mu0.north west)!0.5!(mu0.north)$) -- ($(mu0.east)!0.4!(state0.east)$) -- (nn1-node-1-1.west);
\draw[connector] ($(mu0.north east)!0.5!(mu0.north)$) -- (nn1-node-5-1.west);

\draw[connector] (nn1-H) -- (bar1);
\draw[connector] (bar1) -- (si1); 
\coordinate (fork1) at ($(si1.east)!0.5!(state1.west)$);
\draw[connector] (si1.east) -- (fork1);
\coordinate (upper_coordinate1) at ($(state1.north west)!0.5!(state1.west)$);
\coordinate (lower_coordinate1) at ($(state1.south west)!0.5!(state1.west)$);
\draw[connector] (fork1) -- (upper_coordinate1);
\draw[connector] (fork1) -- (lower_coordinate1);

\draw[connector] ($(state1.north east)!0.5!(state1.east)$) -- (nn2-node-1-1.west);
\draw[connector] ($(state1.south east)!0.5!(state1.east)$) -- (nn2-node-5-1.west);
\draw[connector] ($(mu1.north west)!0.5!(mu1.north)$) -- ($(mu1.east)!0.4!(state1.east)$) -- (nn2-node-1-1.west);
\draw[connector] ($(mu1.north east)!0.5!(mu1.north)$) -- (nn2-node-5-1.west);

\draw[connector] (nn2-H) -- (bar2);
\draw[connector] (bar2) -- (si2); 
\coordinate (fork2) at ($(si2.east)!0.5!(state2.west)$);
\draw[connector] (si2.east) -- (fork2);
\coordinate (upper_coordinate2) at ($(state2.north west)!0.5!(state2.west)$);
\coordinate (lower_coordinate2) at ($(state2.south west)!0.5!(state2.west)$);
\draw[connector] (fork2) -- (upper_coordinate2);
\draw[connector] (fork2) -- (lower_coordinate2);

\draw[connector] (dots) -- (si_final);
\coordinate (fork_final) at ($(si_final.east)!0.5!(state_final.west)$);
\draw[connector] (si_final.east) -- (fork_final);
\coordinate (upper_coordinate_final) at ($(state_final.north west)!0.5!(state_final.west)$);
\coordinate (lower_coordinate_final) at ($(state_final.south west)!0.5!(state_final.west)$);
\draw[connector] (fork_final) -- (upper_coordinate_final);
\draw[connector] (fork_final) -- (lower_coordinate_final);

\draw[-stealth] (mu0.east) -- (mu1.west) node[midway, yshift=.5em] {Id};
\draw[-stealth] (mu1.east) -- (mu2.west) node[midway, yshift=.5em] {Id};
\draw[-stealth] (mu2.east) -- (mu_final.west) node[midway, yshift=.8em] {$\begin{matrix}\mathrm{Id} \\ \Huge\cdots\end{matrix}$};

\end{tikzpicture}
	\caption{Schematic representation of a Parametric Generalized Hamiltonian Neural Network (PGHNN). The architecture extends the GHNN framework to parameter-dependent Hamiltonian systems through an additional parameter input. Each stage consists of a symplectic integrator block parameterized by separable Hamiltonians with learned kinetic and potential energy components. This figure has been reconstructed to match \cite{Horn2026PHD, HornKorenPGHNN}.}
	\label{fig:PGHNN}
\end{figure}

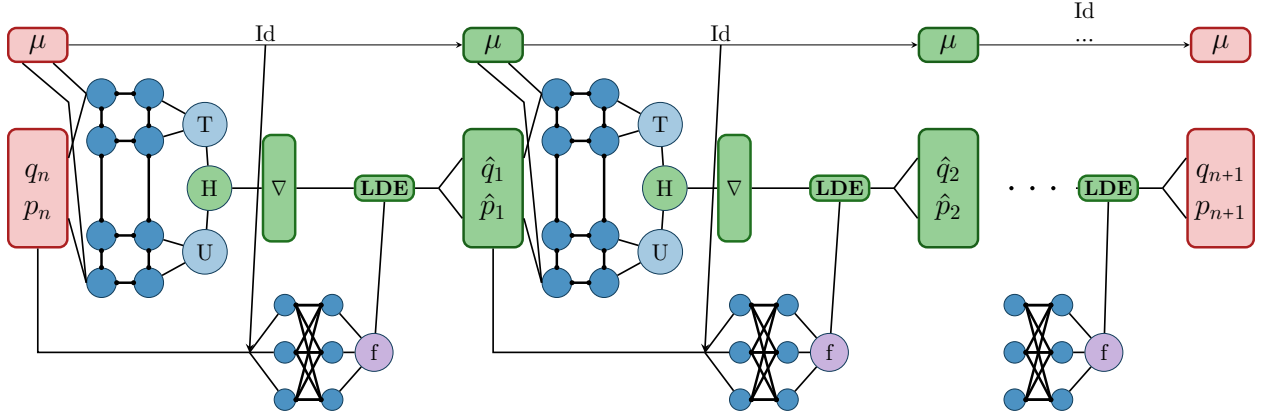
\begin{figure}
	\begin{tikzpicture}[
    node distance=5mm and 8mm,
    state_box/.style={rectangle, rounded corners, minimum height=2cm, minimum width=1cm, very thick, text centered, font=\Large},
    red_box/.style={state_box, fill=mred!25, draw=mred!80!black},
    red_small_box/.style={red_box, minimum height=.4cm},
    green_box/.style={state_box, fill=mgreen!50, draw=mgreen!60!black},
    green_small_box/.style={green_box, minimum height=.4cm},
    blue_node/.style={circle, draw=mblue!50!black, fill=mblue!80, minimum size=5mm},
    blue_small_node/.style={circle, draw=mblue!50!black, fill=mblue!80, minimum size=2mm},
    special_node/.style={blue_node, fill=mblue!40, minimum size=6mm},
    green_bar/.style={rectangle, rounded corners, draw=mgreen!70!black, fill=mgreen!50, very thick, minimum height=1.8cm, minimum width=4mm},
    si_node/.style={rectangle, rounded corners, draw=mgreen!70!black, fill=mgreen!50, very thick, inner sep=2.5pt, font=\small\bfseries},
    nabla_node/.style={font=\huge}, 
    connector/.style={draw, thick},
    arrow/.style={draw, thick, -stealth},
    dot_connector/.style={draw, very thick, line cap=round}
]

\tikzset{
    pics/nn_module/.style={
        code={
            \foreach \row in {1,2,4,5} {
                \foreach \col in {1,2} {
                    \node[blue_node] (-node-\row-\col) at (-2+\col*0.8, 2.4-\row*0.8) {};
                }
            }
            \foreach \row in {1,2,4,5} {
                \draw[dot_connector] (-node-\row-1.east) -- (-node-\row-2.west);
                \fill (-node-\row-1.east) circle (1.2pt);
                \fill (-node-\row-2.west) circle (1.2pt);
            }
            \foreach \col in {1,2} {
                 \draw[dot_connector] (-node-2-\col.south) -- (-node-4-\col.north);
                 \draw[dot_connector] (-node-2-\col.north) -- (-node-1-\col.south);
                 \draw[dot_connector] (-node-5-\col.north) -- (-node-4-\col.south);
                 \fill (-node-2-\col.south) circle (1.2pt);
                 \fill (-node-4-\col.north) circle (1.2pt);
                 \fill (-node-2-\col.north) circle (1.2pt);
                 \fill (-node-1-\col.south) circle (1.2pt);
                 \fill (-node-4-\col.south) circle (1.2pt);
                 \fill (-node-5-\col.north) circle (1.2pt);
            }

            \node[special_node, fill=mgreen!50] (-H) at (.5, 0) {H};
            \node[special_node, above of=-H, yshift=+1.5em, xshift=-.2em] (-T) {T};
            \node[special_node, below of=-H, yshift=-1.5em, xshift=-.2em] (-U) {U};

            \draw[connector] (-node-2-2) -- (-T);
            \draw[connector] (-node-1-2) -- (-T);
            \draw[connector] (-node-4-2) -- (-U);
            \draw[connector] (-node-5-2) -- (-U);
            \draw[connector] (-H) -- (-U);
            \draw[connector] (-T) -- (-H);
        }
    },
    pics/fnn_module/.style={
        code={
            \foreach \row in {1,2,3} {
                \foreach \col in {1,2} {
                    \node[blue_small_node] (-node-\row-\col) at (-2+\col*0.8, 2.4-\row*0.8) {};
                }
            }
            \foreach \row in {1,2,3} {
                \draw[dot_connector] (-node-\row-1.east) -- (-node-\row-2.west);
                \fill (-node-\row-1.east) circle (1.2pt);
                \fill (-node-\row-2.west) circle (1.2pt);
                \foreach \rowt in {1,2,3} {
                \draw[dot_connector] (-node-\row-1.east) -- (-node-\rowt-2.west);
                }
            }
            \fill (-node-1-1.east) circle (1.2pt);
            \fill (-node-2-1.east) circle (1.2pt);
            \fill (-node-3-1.east) circle (1.2pt);
            \fill (-node-1-2.west) circle (1.2pt);
            \fill (-node-2-2.west) circle (1.2pt);
            \fill (-node-3-2.west) circle (1.2pt);

            \node[special_node, fill=mpurple!50, right=.5em of -node-2-2] (-f) {f};

            \draw[connector] (-node-1-2) -- (-f);
            \draw[connector] (-node-2-2) -- (-f);
            \draw[connector] (-node-3-2) -- (-f);
        }
    }
}


\node[red_box] (state0) {$\begin{matrix}q_n \\ p_n\end{matrix}$};
\node[red_small_box, above of=state0, yshift=+5em] (mu0) {$\mu$};

\pic[right=1.5cm of state0] (nn1) {nn_module};
\node[green_bar, right=.5cm of nn1-H] (bar1) {$\nabla$};
\node[nabla_node, right=1mm of bar1] (nabla1) {};
\node[si_node, right=6mm of nabla1] (si1) {LDE};
\node[green_box, right=of si1] (state1) {$\begin{matrix}\hat{q}_1 \\ \hat{p}_1\end{matrix}$};
\node[green_small_box, above of=state1, yshift=+5em] (mu1) {$\mu$};
\pic[below of=nabla1, yshift=-8em, xshift=2em] (fnn1) {fnn_module};
\coordinate (fnn_fork1) at ($(fnn1-node-2-1)-(.6,0)$);

\pic[right=1.5cm of state1] (nn2) {nn_module};
\node[green_bar, right=.5cm of nn2-H] (bar2) {$\nabla$};
\node[nabla_node, right=1mm of bar2] (nabla2) {};
\node[si_node, right=6mm of nabla2] (si2) {LDE};
\node[green_box, right=of si2] (state2) {$\begin{matrix}\hat{q}_2 \\ \hat{p}_2\end{matrix}$};
\node[green_small_box, above of=state2, yshift=+5em] (mu2) {$\mu$};
\pic[below of=nabla2, yshift=-8em, xshift=2em] (fnn2) {fnn_module};
\coordinate (fnn_fork2) at ($(fnn2-node-2-1)-(.6,0)$);

\node[right=.3cm of state2, font=\Huge] (dots) {\dots};
\node[si_node, right=of dots, xshift=-2em] (si_final) {LDE};
\node[red_box, right=of si_final] (state_final) {$\begin{matrix}q_{n+1} \\ p_{n+1}\end{matrix}$};
\node[red_small_box, above of=state_final, yshift=+5em] (mu_final) {$\mu$};
\pic[below of=si_final, yshift=-8em, xshift=-1em] (fnn_final) {fnn_module};
\coordinate (fnn_fork_final) at ($(fnn_final-node-2-1)-(.6,0)$);

\draw[connector] ($(state0.north east)!0.5!(state0.east)$) -- (nn1-node-1-1.west);
\draw[connector] ($(state0.south east)!0.5!(state0.east)$) -- (nn1-node-5-1.west);
\draw[connector] ($(mu0.south east)!0.5!(mu0.south)$) -- (nn1-node-1-1.west);
\draw[connector] ($(mu0.south west)!0.5!(mu0.south)$) -- ($(mu0.east)!0.4!(state0.east)$) -- (nn1-node-5-1.west);

\draw[connector] (nn1-H) -- (bar1);
\draw[connector] (bar1) -- (si1);
\coordinate (fork1) at ($(si1.east)!0.5!(state1.west)$);
\draw[connector] (si1.east) -- (fork1);
\coordinate (upper_coordinate1) at ($(state1.north west)!0.5!(state1.west)$);
\coordinate (lower_coordinate1) at ($(state1.south west)!0.5!(state1.west)$);
\draw[connector] (fork1) -- (upper_coordinate1);
\draw[connector] (fork1) -- (lower_coordinate1);
\draw[connector] (fnn1-f) -- (si1.south);
\draw[arrow] ($(mu0.east)!0.5!(mu1.west)$) -- (fnn_fork1);
\draw[connector] (state0.south) |- (fnn_fork1);
\draw[connector] (fnn_fork1) -- (fnn1-node-1-1);
\draw[connector] (fnn_fork1) -- (fnn1-node-2-1);
\draw[connector] (fnn_fork1) -- (fnn1-node-3-1);

\draw[connector] ($(state1.north east)!0.5!(state1.east)$) -- (nn2-node-1-1.west);
\draw[connector] ($(state1.south east)!0.5!(state1.east)$) -- (nn2-node-5-1.west);
\draw[connector] ($(mu1.south east)!0.5!(mu1.south)$) -- (nn2-node-1-1.west);
\draw[connector] ($(mu1.south west)!0.5!(mu1.south)$) -- ($(mu1.east)!0.4!(state1.east)$) -- (nn2-node-5-1.west);

\draw[connector] (nn2-H) -- (bar2);
\draw[connector] (bar2) -- (si2);
\coordinate (fork2) at ($(si2.east)!0.5!(state2.west)$);
\draw[connector] (si2.east) -- (fork2);
\coordinate (upper_coordinate2) at ($(state2.north west)!0.5!(state2.west)$);
\coordinate (lower_coordinate2) at ($(state2.south west)!0.5!(state2.west)$);
\draw[connector] (fork2) -- (upper_coordinate2);
\draw[connector] (fork2) -- (lower_coordinate2);
\draw[connector] (fnn2-f) -- (si2.south);
\draw[arrow] ($(mu1.east)!0.5!(mu2.west)$) -- (fnn_fork2);
\draw[connector] (state1.south) |- (fnn_fork2);
\draw[connector] (fnn_fork2) -- (fnn2-node-1-1);
\draw[connector] (fnn_fork2) -- (fnn2-node-2-1);
\draw[connector] (fnn_fork2) -- (fnn2-node-3-1);

\draw[connector] (dots) -- (si_final);
\coordinate (fork_final) at ($(si_final.east)!0.5!(state_final.west)$);
\draw[connector] (si_final.east) -- (fork_final);
\coordinate (upper_coordinate_final) at ($(state_final.north west)!0.5!(state_final.west)$);
\coordinate (lower_coordinate_final) at ($(state_final.south west)!0.5!(state_final.west)$);
\draw[connector] (fork_final) -- (upper_coordinate_final);
\draw[connector] (fork_final) -- (lower_coordinate_final);
\draw[connector] (fnn_final-f) -- (si_final.south);

\draw[-stealth] (mu0.east) -- (mu1.west) node[midway, yshift=.5em] {Id};
\draw[-stealth] (mu1.east) -- (mu2.west) node[midway, yshift=.5em] {Id};
\draw[-stealth] (mu2.east) -- (mu_final.west) node[midway, yshift=.8em] {$\begin{matrix}\mathrm{Id} \\ \Huge\cdots\end{matrix}$};

\end{tikzpicture}
	\caption{Schematic representation of a Parametric Generalized Forced Hamiltonian Neural Network (PGFHNN), extending the PGHNN architecture (see \Cref{fig:PGHNN}) to incorporate external forcing. The model combines parameter-dependent Hamiltonian dynamics with additional terms representing non-conservative forces.}
	\label{fig:PGFHNN}
\end{figure}

\noindent
where $\theta=(\vartheta_1,\psi_1,\chi_1,\vartheta_2,\psi_2,\chi_2,\ldots)$ collects all the trainable parameters $\vartheta_i$, $\psi_i$, $\chi_i$ of the neural networks $\tilde T_i$, $\tilde U_i$, $\tilde f_i$, respectively, for $i=1,\ldots,m$. A schematic representation of a PGFHNN is shown in \Cref{fig:PGFHNN}. Given a set of training data $((q^a_i,p^a_i),(q^b_i,p^b_i),\mu_i)$ for $i=1,\ldots,d$ such that $(q^b_i,p^b_i)=F_{\Delta t}(q^a_i,p^a_i,\mu_i)$, the neural network \eqref{eq: PGFHNN as a concatenation} can be trained by minimizing the mean squared loss,

\begin{align}
\label{eq:Loss function for PGFHNN}
\text{Loss}(\theta) = \frac{1}{2nd}\sum_{i=1}^d \big\| (q^b_i,p^b_i,\mu_i) - PGFHNN(q^a_i,p^a_i,\mu_i;\theta) \big\|^2,
\end{align}

\noindent
using a suitable optimization algorithm.

\subsection{Universal approximation theorem for PGFHNNs}
\label{sec: Universal approximation theorem for PGFHNNs}

A universal approximation theorem for PGFHNNs can be established by following the same steps as in Section~\ref{sec: Universal approximation theorem for GFHNNs}. First, in view of the discussion in Section~\ref{sec:Parametric time-dependent forced Hamiltonian systems}, Proposition~\ref{thm:Proposition on approximating phi by Lie-Trotter} extends naturally to the parametric setting and allows us to approximate parametric Lagrange-d'Alembert maps by parameter-dependent Lie-Trotter integrators in the $C^r$ topology on compact sets $K\subset T^*Q\times I$. Moreover, Proposition~\ref{thm:Proposition on approximating phi by LDE} extends naturally to approximations by augmented Lagrange-d'Alembert-Euler maps. Here we use the fact that \cite[Theorem~2]{Turaev2002} was proved for parameter-dependent symplectic diffeomorphisms. Analogously to \eqref{eq:Definition of the space of LDE^r neural networks}, we define

\begin{equation}
\label{eq:Definition of the space of augmented LDE^r neural networks}
LDE^r_\mathcal{N_\sigma}(T^*Q,I)= \Big\{ LDE^\mathrm{aug}_{\tilde T, \tilde U, \tilde f} \; \Big| \; \tilde T, \tilde U \in \mathcal{N}^{r+1}_\sigma(\mathbb{R}^n\times I,\mathbb{R}) \text{ and } \tilde f \in \mathcal{N}^{r}_\sigma(\mathbb{R}^{2n}\times I,\mathbb{R}^n) \Big\}.
\end{equation}

\noindent
It is then straightforward to extend Proposition~\ref{thm:Proposition on approximating LDE by neural networks} and show that $LDE^r_\mathcal{N_\sigma}(T^*Q,I)$ is uniformly $r$-dense on compacta in $LDE^r(T^*Q,I)$. Combining all these results, we obtain the following theorem.

\begin{thm}[{\bf Universal approximation theorem for PGFHNNs}]
\label{thm:Universal approximation theorem for PGFHNNs}
Let $r\ge 0$ be an integer and let $\varphi \in LdA^{r+1}(T^*Q,I)$. Then, for every compact set $K \subset T^*Q\times I$ and every $\epsilon > 0$, there exists a finite sequence of maps $LDE^\mathrm{aug}_{\tilde T_i, \tilde U_i, \tilde f_i} \in LDE^{r+1}_{\mathcal{N}_\sigma}(T^*Q,I)$, $i = 1,\ldots,m$, such that

\begin{equation}
\label{eq:Approximating phi by augmented LDE neural networks}
\left\| \varphi - \pi\circ LDE^\mathrm{aug}_{\tilde T_m, \tilde U_m, \tilde f_m}\circ \ldots \circ LDE^\mathrm{aug}_{\tilde T_1, \tilde U_1, \tilde f_1} \right\|_{C^r(K)} < \epsilon,
\end{equation}

\noindent
where $\pi : T^*Q \times I \longrightarrow T^*Q$ denotes the projection onto $T^*Q$.
\end{thm}

\section{Learning time-dependent and stochastic systems}
\label{sec: Learning time-dependent and stochastic systems}

In this section, we show how time-dependent and stochastic systems can be reformulated so that the PGFHNNs developed in Section~\ref{sec: Parametric Generalized Forced Hamiltonian Neural Networks} can be used to learn their flows.

\subsection{Time-dependent systems}
\label{sec: Time-dependent systems}

Consider the time-dependent forced Hamiltonian system \eqref{eq: Time-dependent forced Hamiltonian system}. Fix a time step $\Delta t >0$. Then the flow $F_{t+\Delta t,t}$ can be viewed as a parametric Lagrange-d'Alembert map in the sense of Definition~\ref{thm: Definition of parametric Lagrange-d'Alembert maps}, with time $t$ playing the role of a parameter. Indeed, define a parameter-dependent Hamiltonian $\tilde H: T^*Q \times [0,\Delta t]\times\mathbb{R} \longrightarrow \mathbb{R}$ and a parameter-dependent force $\tilde f: T^*Q \times [0,\Delta t]\times\mathbb{R} \longrightarrow \mathbb{R}$ by

\begin{equation}
\tilde H(q,p,\tau,\mu) = H(q,p,\tau+\mu), \qquad\quad \tilde f(q,p,\tau,\mu) = f(q,p,\tau+\mu), \qquad\quad \text{for $\tau \in [0,\Delta t]$,}
\end{equation}

\noindent
and let $\tilde F_{\tau,\tau_0}:T^*Q\times\mathbb{R}\longrightarrow T^*Q$ denote the corresponding parametric Lagrange-d'Alembert flow. We then have $F_{t+\Delta t,t}(q,p) = \tilde F_{\Delta t,0}(q,p,t)$. Therefore, the technique proposed in Section~\ref{sec: Parametric Generalized Forced Hamiltonian Neural Networks} can be used to learn $F_{t+\Delta t,t}$ from data.

\subsection{Stochastic systems}
\label{sec: Stochastic systems}

Stochastic differential equations (SDEs) play an important role in modeling dynamical systems subject to internal or external random fluctuations \cite{ArnoldSDE,IkedaWatanabe1989,KloedenPlatenSDE,Kunita1997,MilsteinBook}. Within this class of problems, we are interested in stochastic forced Hamiltonian systems (see \cite{KrausTyranowski2019} and the references therein), which take the form

\begin{align}
\label{eq: Stochastic dissipative Hamiltonian system}
d_t q &= \frac{\partial H_0}{\partial p}dt + \sum_{i=1}^m\frac{\partial H_i}{\partial p}\circ dW^i(t), \nonumber \\
d_t p &= \bigg[-\frac{\partial H_0}{\partial q} + f_0(q,p) \bigg] dt + \sum_{i=1}^m \bigg[-\frac{\partial H_i}{\partial q}+f_i(q,p)\bigg]\circ dW^i(t),
\end{align}

\noindent
where $H_i=H_i(q,p)$ for $i=0,\ldots,m$ are the Hamiltonian functions, $f_i=f_i(q,p)$ are the forcing terms, $W(t)=(W^1(t),\ldots,W^m(t))$ is the standard $m$-dimensional Wiener process, and $\circ$ denotes Stratonovich integration. We use $d_t$ to denote the stochastic differential of stochastic processes (other than the Wiener process $W(t)$) to avoid confusion with the exterior derivative $d$ of differential forms. The system \eqref{eq: Stochastic dissipative Hamiltonian system} can be formally regarded as a time-dependent forced Hamiltonian system \eqref{eq: Time-dependent forced Hamiltonian system} with the randomized Hamiltonian given by $H(q,p,t) = H_0(q,p) + \sum_{i=1}^m H_i(q,p)\circ \dot W^i(t)$, and the randomized forcing given by $f(q,p,t) = f_0(q,p) + \sum_{i=1}^m f_i(q,p)\circ \dot W^i(t)$, where $H_0(q,p)$ and $f_0(q,p)$ are the deterministic Hamiltonian and forcing, respectively, and  $H_i(q,p)$, $f_i(q,p)$ represent the intensity of the noise. Such systems can serve to model, for instance, mechanical systems affected by uncertainty or error, which are presumed to result from random forcing, limited precision of experimental measurements, or unresolved physical processes on which the Hamiltonian of the underlying deterministic system might otherwise depend. Applications can be found in a wide range of models in physics, chemistry, and biology. Particular examples include molecular dynamics \cite{Skeel1999}, dissipative particle dynamics \cite{PengArai2022,Ripoll2001}, and collisional kinetic plasmas \cite{KrausTyranowski2019,LuMengTyranowski2025,Sonnendrucker2015,TyranowskiVlasovMaxwell}.

The stochastic flow $F_{t,t_0}:\Omega_s\times T^*Q \longrightarrow T^*Q$ for \Cref{eq: Stochastic dissipative Hamiltonian system} is time-dependent due to the fact that it is driven by the Wiener process $W(t)$, where $\Omega_s$ denotes the sample space of the underlying probability space. As shown in \cite{KrausTyranowski2019}, the stochastic system \eqref{eq: Stochastic dissipative Hamiltonian system} has an underlying stochastic variational principle which generalizes the Lagrange-d'Alembert principle \eqref{eq: Lagrange-d'Alembert principle}. This means that the map $F_{t,t_0}$ is a stochastic Lagrange-d'Alembert flow, and the stochastic forced Hamiltonian system \eqref{eq: Stochastic dissipative Hamiltonian system} can be numerically approximated in a structure-preserving way by using stochastic Lagrange-d'Alembert integrators which are defined by \Cref{eq: Lagrange-d'Alembert integrator} with a stochastic discrete Lagrangian $L_d: \Omega_s \times Q\times Q \longrightarrow \mathbb{R}^n$ and stochastic discrete forces $f_d^+,f_d^-:\Omega_s \times Q\times Q \longrightarrow \mathbb{R}^n$. One example of such an integrator is the stochastic implicit midpoint method,

\begin{align}
	\label{eq:Stochastic midpoint method}
	q_{k+1} &= q_k + \frac{\partial H_0}{\partial p} \bigg(\frac{q_k+q_{k+1}}{2},\frac{p_k+p_{k+1}}{2} \bigg)\Delta t
	               + \sum_{i=1}^m \frac{\partial H_i}{\partial p} \bigg(\frac{q_k+q_{k+1}}{2},\frac{p_k+p_{k+1}}{2} \bigg)\Delta W^i, \nonumber \\
	p_{k+1} &= p_k + \bigg[ -\frac{\partial H_0}{\partial q} \bigg(\frac{q_k+q_{k+1}}{2},\frac{p_k+p_{k+1}}{2} \bigg) + f_0\bigg(\frac{q_k+q_{k+1}}{2},\frac{p_k+p_{k+1}}{2} \bigg) \bigg] \Delta t \nonumber \\
	        &\phantom{= p_k}+ \sum_{i=1}^m \bigg[ -\frac{\partial H_i}{\partial q} \bigg(\frac{q_k+q_{k+1}}{2},\frac{p_k+p_{k+1}}{2} \bigg) + f_i\bigg(\frac{q_k+q_{k+1}}{2},\frac{p_k+p_{k+1}}{2} \bigg) \bigg]\Delta W^i,
	\end{align}

\noindent
which defines a stochastic discrete flow $\widehat F_{t_{k+1},t_k}: \Omega_s \times T^*Q \ni(q_k,p_k) \longmapsto (q_{k+1},p_{k+1})\in T^*Q$, where $\Delta t = t_{k+1}-t_k$ is the time step, and $\Delta W = W(t_{k+1})-W(t_k)$ is the increment of the Wiener process. As demonstrated in \cite{KrausTyranowski2019}, the scheme \eqref{eq:Stochastic midpoint method} is a stochastic Lagrange-d'Alembert integrator with the stochastic discrete Lagrangian and discrete forces given by

\begin{align}
\label{eq: Discrete Lagrangian and forces for the stochastic midpoint method}
L_d(q_k,q_{k+1}) &= \Delta t \bigg[ p_c\frac{\partial H_0}{\partial p}(q_c,p_c) -H_0(q_c,p_c) \bigg] +\sum_{i=1}^m \Delta W^i \bigg[ p_c\frac{\partial H_i}{\partial p}(q_c,p_c) -H_i(q_c,p_c) \bigg], \nonumber \\
f^-_d(q_k,q_{k+1}) &= \frac{1}{2}\Delta t  f_0(q_c,p_c) + \frac{1}{2}\sum_{i=1}^m \Delta W^i f_i(q_c,p_c), \nonumber \\
f^+_d(q_k,q_{k+1}) &= \frac{1}{2} \Delta t f_0(q_c,p_c) + \frac{1}{2}\sum_{i=1}^m \Delta W^i f_i(q_c,p_c),
\end{align}

\noindent
with $q_c=(q_k+q_{k+1})/2$, $p_c=(p_k+p_{k+1})/2$, where $p_k$ and $p_{k+1}$ are understood as functions of $q_k$ and $q_{k+1}$, implicitly defined by \Cref{eq:Stochastic midpoint method}. Note that stochasticity and time-dependence enter the definition of $\widehat F_{t_{k+1},t_k}$ via the Wiener process increments $\Delta W^i \sim N(0,\Delta t)$, which are normally distributed random variables. Let us instead consider the deterministic parameter-dependent Lagrange-d'Alembert map $\varphi:T^*Q\times \mathbb{R}^m \ni(q_k,p_k,\mu) \longmapsto (q_{k+1},p_{k+1}) \in T^*Q$ defined implicitly by the equations

\begin{align}
	\label{eq:Parameter-dependent midpoint method}
	q_{k+1} &= q_k + \frac{\partial H_0}{\partial p} \bigg(\frac{q_k+q_{k+1}}{2},\frac{p_k+p_{k+1}}{2} \bigg)\Delta t
	               + \sum_{i=1}^m \frac{\partial H_i}{\partial p} \bigg(\frac{q_k+q_{k+1}}{2},\frac{p_k+p_{k+1}}{2} \bigg)\mu^i, \nonumber \\
	p_{k+1} &= p_k + \bigg[ -\frac{\partial H_0}{\partial q} \bigg(\frac{q_k+q_{k+1}}{2},\frac{p_k+p_{k+1}}{2} \bigg) + f_0\bigg(\frac{q_k+q_{k+1}}{2},\frac{p_k+p_{k+1}}{2} \bigg) \bigg] \Delta t \nonumber \\
	        &\phantom{= p_k}+ \sum_{i=1}^m \bigg[ -\frac{\partial H_i}{\partial q} \bigg(\frac{q_k+q_{k+1}}{2},\frac{p_k+p_{k+1}}{2} \bigg) + f_i\bigg(\frac{q_k+q_{k+1}}{2},\frac{p_k+p_{k+1}}{2} \bigg) \bigg]\mu^i,
\end{align}

\noindent
which is generated by the deterministic parameter-dependent discrete Lagrangian and forces

\begin{align}
\label{eq: Discrete Lagrangian and forces for the parameter-dependent midpoint method}
\bar L_d(q_k,q_{k+1}; \mu) &= \Delta t \bigg[ p_c\frac{\partial H_0}{\partial p}(q_c,p_c) -H_0(q_c,p_c) \bigg] +\sum_{i=1}^m \mu^i \bigg[ p_c\frac{\partial H_i}{\partial p}(q_c,p_c) -H_i(q_c,p_c) \bigg], \nonumber \\
\bar f^-_d(q_k,q_{k+1}; \mu) &= \frac{1}{2}\Delta t  f_0(q_c,p_c) + \frac{1}{2}\sum_{i=1}^m \mu^i f_i(q_c,p_c), \nonumber \\
\bar f^+_d(q_k,q_{k+1}; \mu) &= \frac{1}{2} \Delta t f_0(q_c,p_c) + \frac{1}{2}\sum_{i=1}^m \mu^i f_i(q_c,p_c).
\end{align}

\noindent
It is straightforward to see that then we have

\begin{equation}
\label{eq: Stochastic flow as a parameter-dependent flow}
\widehat F_{t_{k+1},t_k}(\omega, q,p) = \varphi\big(q,p,W(\omega,t_{k+1})-W(\omega,t_k)\big),
\end{equation}

\noindent
where we explicitly stated the sample space argument $\omega \in \Omega_s$. The map $\varphi_\mu$ can be learnt from data using the technique developed in Section~\ref{sec: Parametric Generalized Forced Hamiltonian Neural Networks} if information about the discrete Wiener process paths $W(t_0), W(t_1), \ldots$ is also available. Such situations arise naturally in applications where large-scale Monte Carlo simulations of SDEs are performed repeatedly for different parameter values, for example in uncertainty quantification. In these settings, the realizations of the Wiener process are available and can be used directly as inputs to the learning procedure.

The stochastic midpoint method \eqref{eq:Stochastic midpoint method} uses only time increments $J_0\equiv \Delta t$ and Wiener process increments $J_i\equiv\Delta W^i$ for $i=1,\ldots,m$, and it is therefore strongly convergent of order $1/2$ in general, and of order 1 in the case of commutative noise \cite{MilsteinRepin,KrausTyranowski2019}. In order to achieve a higher order of convergence, a numerical scheme must involve higher-order multiple Stratonovich integrals \cite{KloedenPlatenSDE}

\begin{equation}
\label{eq: Multiple Stratonovich integrals}
J_{i_1,\ldots,i_l} = \int_{t_k}^{t_{k+1}} \!\!\cdots\!\! \int_{t_k}^{s_2} \circ dZ^{i_1}(s_1)\ldots\circ dZ^{i_l}(s_l),
\end{equation}

\noindent
for $0\leq i_1,\ldots,i_l \leq m$, where $Z=(Z^0,Z^1,\ldots,Z^m)$ with $Z^0(t)=t$ and $Z^i(t)=W^i(t)$ for $i=1,\ldots,m$. For instance, a scheme of order 3/2 must involve $J_i$ and $J_{i,0}$ \cite{HolmTyranowskiGalerkin,MilsteinRepin}, therefore the corresponding parameter-dependent Lagrange-d'Alembert map $\varphi_\mu$ will use a higher-dimensional parameter $\mu=(J_i, J_{j,0})$. Formally speaking, the exact stochastic flow $F_{t_{k+1},t_k}$ can be viewed as parametrized by the infinite sequence $\mu=(J_i, J_{i,j}, J_{i,j,l},\ldots)$ that appears in the Stratonovich-Taylor expansion of the solution $(q(t),p(t))$ of \Cref{eq: Stochastic dissipative Hamiltonian system} \cite{KloedenPlatenSDE}.


\section{Numerical experiments}
\label{sec: Numerical experiments}

In this section, we present the results of numerical experiments in which we tested the neural network architectures introduced in Sections~\ref{sec: Generalized Forced Hamiltonian Neural Networks} and~\ref{sec: Parametric Generalized Forced Hamiltonian Neural Networks}, and compared their performance with that of non-geometric residual neural networks (ResNets, \cite{ChenRubanova2018, HeResNets2016}). The experiments were implemented using the \texttt{GeometricMachineLearning.jl} package \cite{GeometricMachineLearning}. In all experiments, the hyperbolic tangent activation function ($\tanh$) was used, and all networks were trained using the Adam optimizer.

\subsection{Linearly damped harmonic oscillator}
\label{sec: Linearly damped harmonic oscillator}

As the first example we consider the linearly damped harmonic oscillator, which is a system of the form \eqref{eq: Autonomous forced Hamiltonian system} with

\begin{equation}
\label{eq: Linearly damped harmonic oscillator: Hamiltonian and forcing}
H(q,p) = \frac{1}{2}p^2 + \frac{1}{2}q^2, \qquad\qquad f(q,p)=-\nu p,
\end{equation}

\noindent
where $\nu$ is the friction coefficient. For this system an analytic solution is available and is given by

\begin{align}
\label{eq:Linearly damped oscillator---exact solution}
\bar q(t)&= \bar q_0 e^{-\frac{\nu}{2}t} \cos \omega t + \frac{1}{\omega}\Big(\bar p_0+\frac{\nu}{2} \bar q_0\Big) e^{-\frac{\nu}{2}t} \sin \omega t, \nonumber \\
\bar p(t)&= \bar p_0 e^{-\frac{\nu}{2}t} \cos \omega t - \frac{1}{\omega}\Big(\bar q_0+\frac{\nu}{2} \bar p_0\Big) e^{-\frac{\nu}{2}t} \sin \omega t,
\end{align}

\noindent
where $\bar q_0$ and $\bar p_0$ denote the initial conditions, the angular frequency is $\omega=\frac{1}{2}\sqrt{4-\nu^2}$, and we assume the underdamped case $0\leq \nu < 2$. The training data for our numerical experiments were created by sampling the exact solution \eqref{eq:Linearly damped oscillator---exact solution} for $0\leq t \leq T$ with the final time $T=13$, time step $\Delta t = 0.13$, and $N_\mathrm{train}=400$ initial conditions $(\bar q_0,\bar p_0)$ uniformly distributed in the square $-1\leq q,p\leq 1$ in the phase space, that is,

\begin{align}
\label{eq: Initial conditions for the training data}
\bar q_0^{i,j}=-1+\frac{2(i-1)}{\sqrt{N_\mathrm{train}}-1}, \qquad\quad \bar p_0^{i,j}=-1+\frac{2(j-1)}{\sqrt{N_\mathrm{train}}-1}, \qquad\quad \text{for $i,j=1,\ldots,\sqrt{N_\mathrm{train}}.$}
\end{align}

\noindent
A total of 6 training data sets were created for different values of the friction coefficient $\nu$, namely,

\begin{equation}
\label{eq: Friction coefficients for linearly damped harmonic oscillator}
\nu = 0.5,\quad 0.1,\quad 0.05,\quad 0.01,\quad 0.005,\quad 0.001.
\end{equation}

\noindent
Using each data set, a GFHNN and a ResNet network were trained for 10,000 epochs with the Adam optimizer in order to learn the flow map $F_{\Delta t}$ of the autonomous system \eqref{eq: Autonomous forced Hamiltonian system}. To make the comparison fair, the networks were chosen to have comparable numbers of trainable parameters, namely, 74 for the GFHNN (consisting of 3 LDE blocks with $\tilde T_i$, $\tilde U_i$, and $\tilde f_i$ each having one hidden layer of width 4) and 72 for the ResNet. Figure~\ref{fig:Damped_Oscillator_Training_loss_for_nu01} reports the loss evolution for a representative training run using data corresponding to $\nu=0.1$. All other runs exhibited comparable convergence behavior.

\begin{figure}
	\centering
		\includegraphics[width=1.00\textwidth]{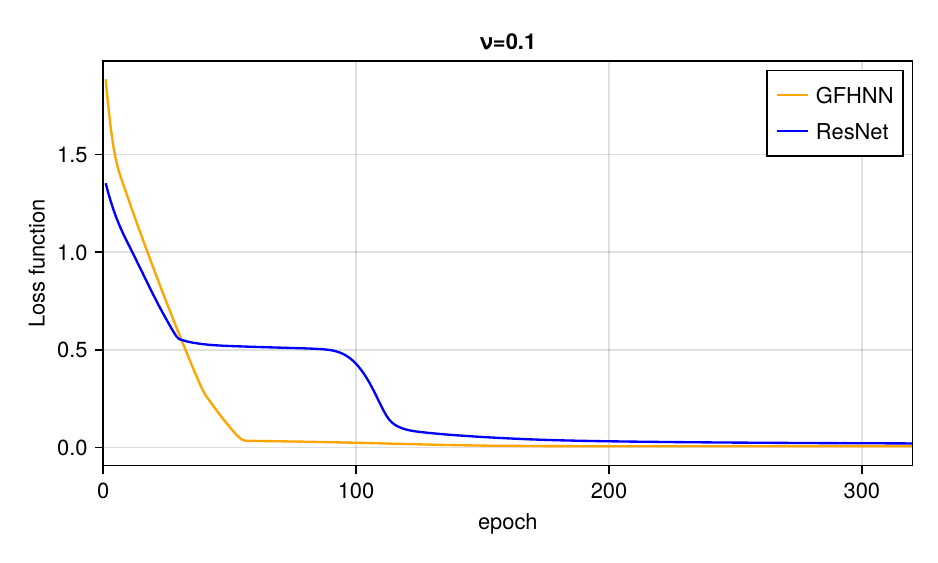}
	\caption{Convergence of the training loss for a representative training run using data generated from the linearly damped harmonic oscillator with damping parameter $\nu = 0.1$.}
\label{fig:Damped_Oscillator_Training_loss_for_nu01}
\end{figure}

\begin{figure}
	\centering
		\includegraphics[width=1.00\textwidth]{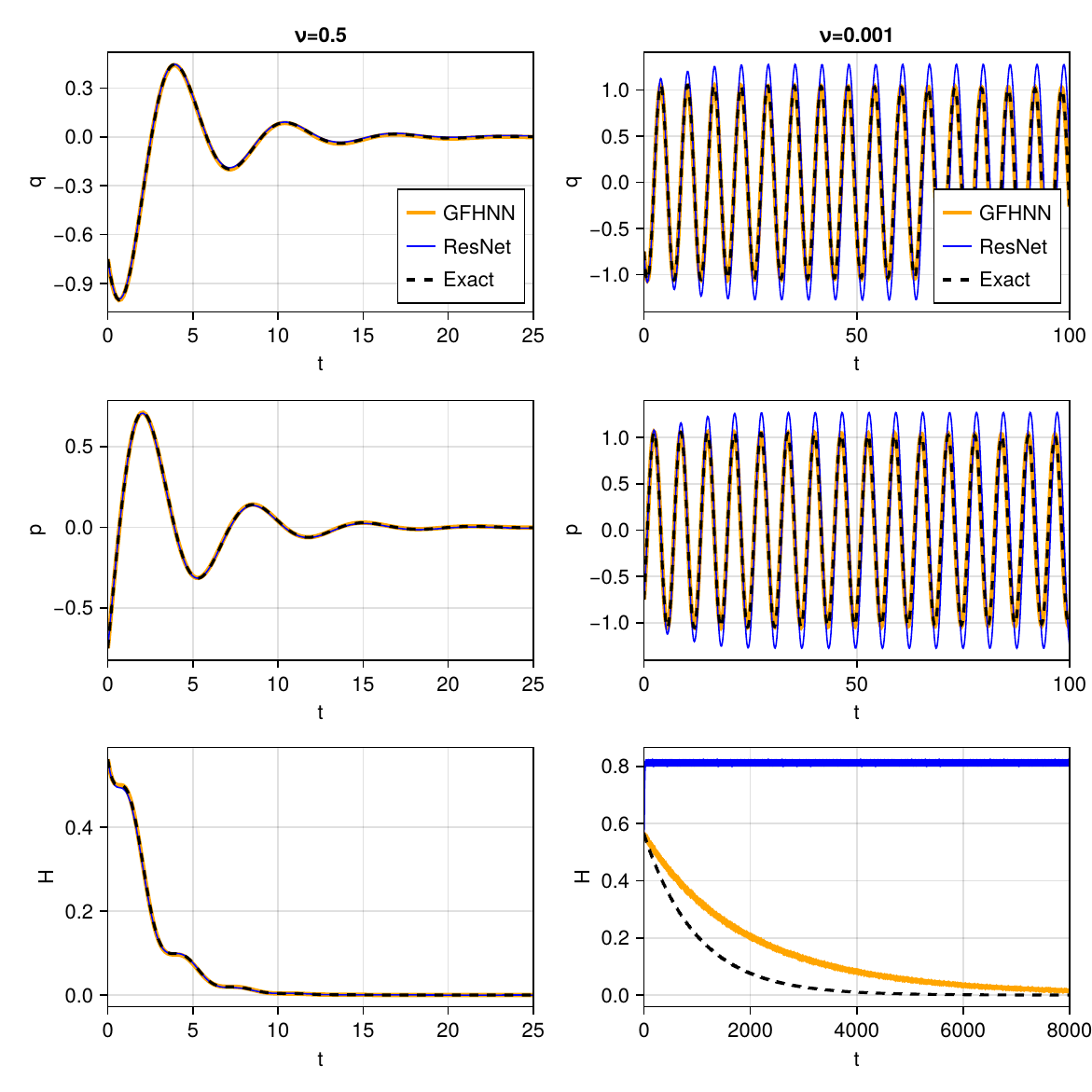}
	\caption{Comparison of trajectories generated by the GFHNN and ResNet flows for the linearly damped harmonic oscillator. The left column corresponds to the case $\nu = 0.5$, while the right column shows the case $\nu = 0.001$. In both cases, trajectories are generated from the same initial condition $(\bar q_0, \bar p_0) = (-0.75, -0.75)$. The first row displays the position $q(t)$, the second row the momentum $p(t)$, and the third row the Hamiltonian $H(t)$ evaluated along the corresponding trajectories.}
	\label{fig:Damped_Oscillator_Trajectories_for_nu05_and_nu0001}
\end{figure}

\begin{figure}
	\centering
		\includegraphics[width=1.00\textwidth]{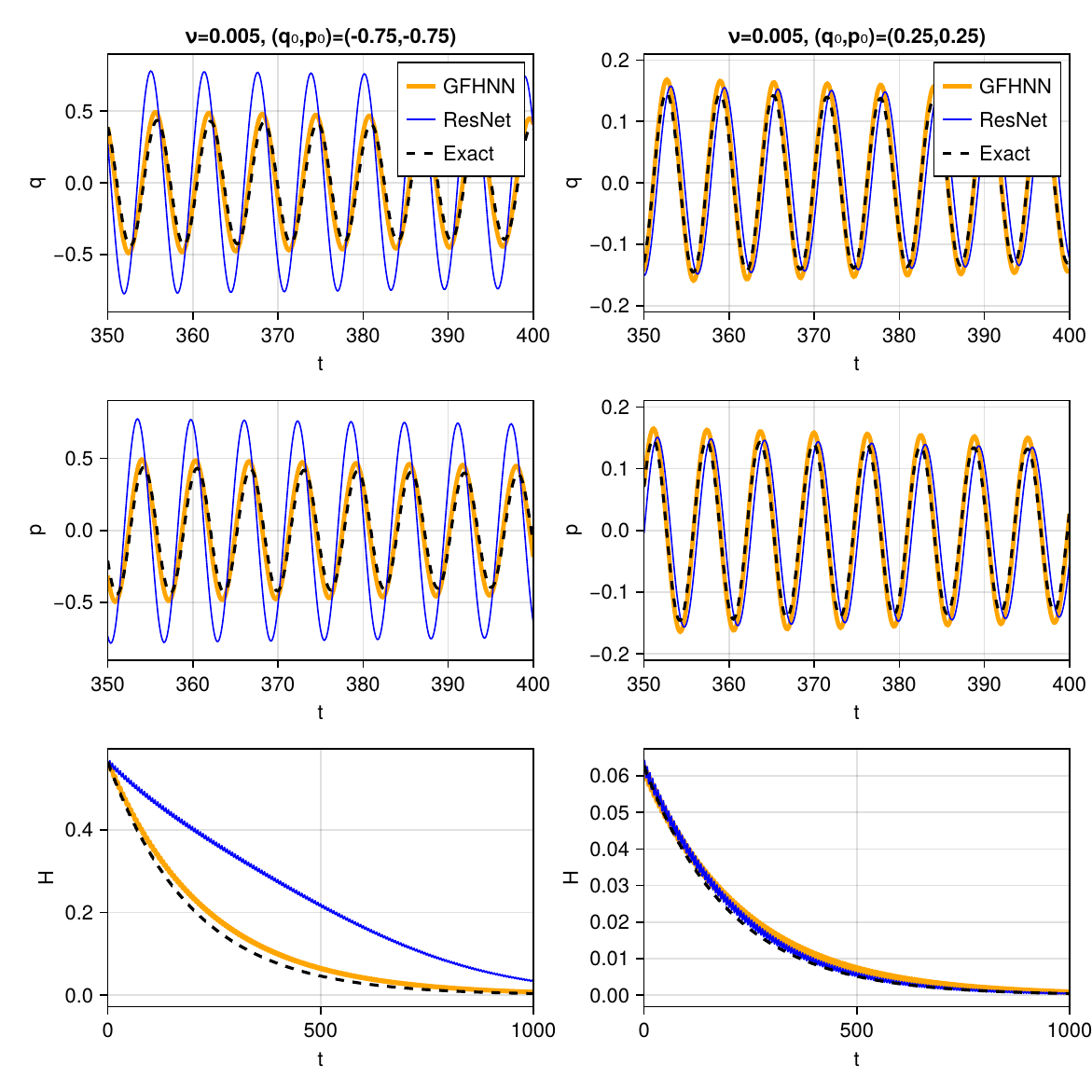}
	\caption{Comparison of trajectories generated by the GFHNN and ResNet flows for the linearly damped harmonic oscillator with $\nu = 0.005$. The left column corresponds to the initial condition $(\bar q_0, \bar p_0) = (-0.75, -0.75)$, while the right column corresponds to $(\bar q_0, \bar p_0) = (0.25, 0.25)$. In both cases, the first row displays the position $q(t)$, the second row the momentum $p(t)$, and the third row the Hamiltonian $H(t)$ evaluated along the corresponding trajectories.}
	\label{fig:Damped_Oscillator_Trajectories_for_nu0005}
\end{figure}

\begin{figure}
	\centering
		\includegraphics[width=1.00\textwidth]{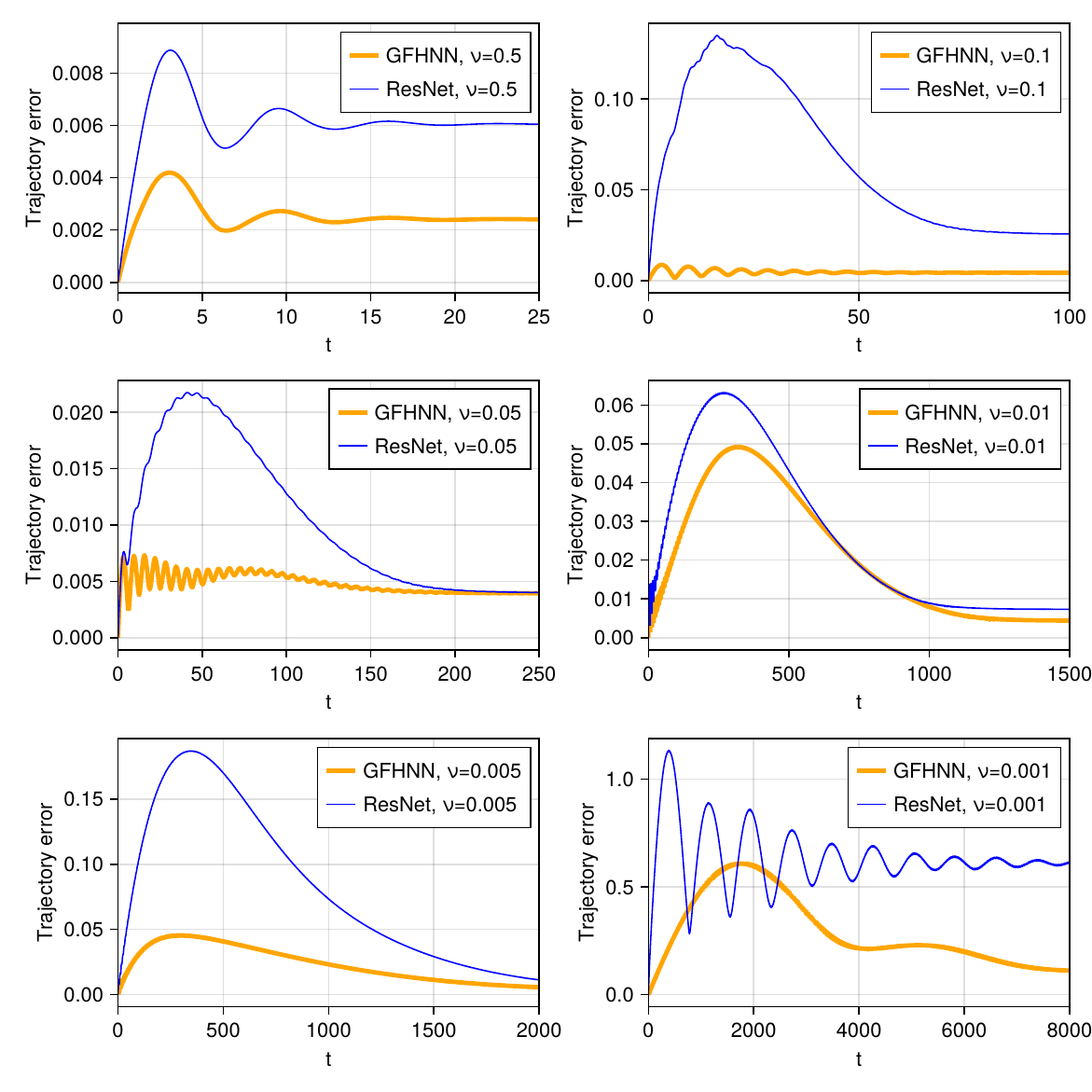}
	\caption{Averaged trajectory error $\epsilon_{\mathrm{traj}}(t)$ for trajectories generated by the GFHNN and ResNet flows for the linearly damped harmonic oscillator. For each value of the damping coefficient $\nu$, the error is computed by averaging over 49 test initial conditions distributed in the square $-1 \leq q,p \leq 1$ in phase space, none of which were used during training. Each subplot corresponds to a distinct value of $\nu$, as indicated in the legend.}
	\label{fig:Damped_Oscillator_Trajectory_Error}
\end{figure}

\begin{figure}
	\centering
		\includegraphics[width=1.00\textwidth]{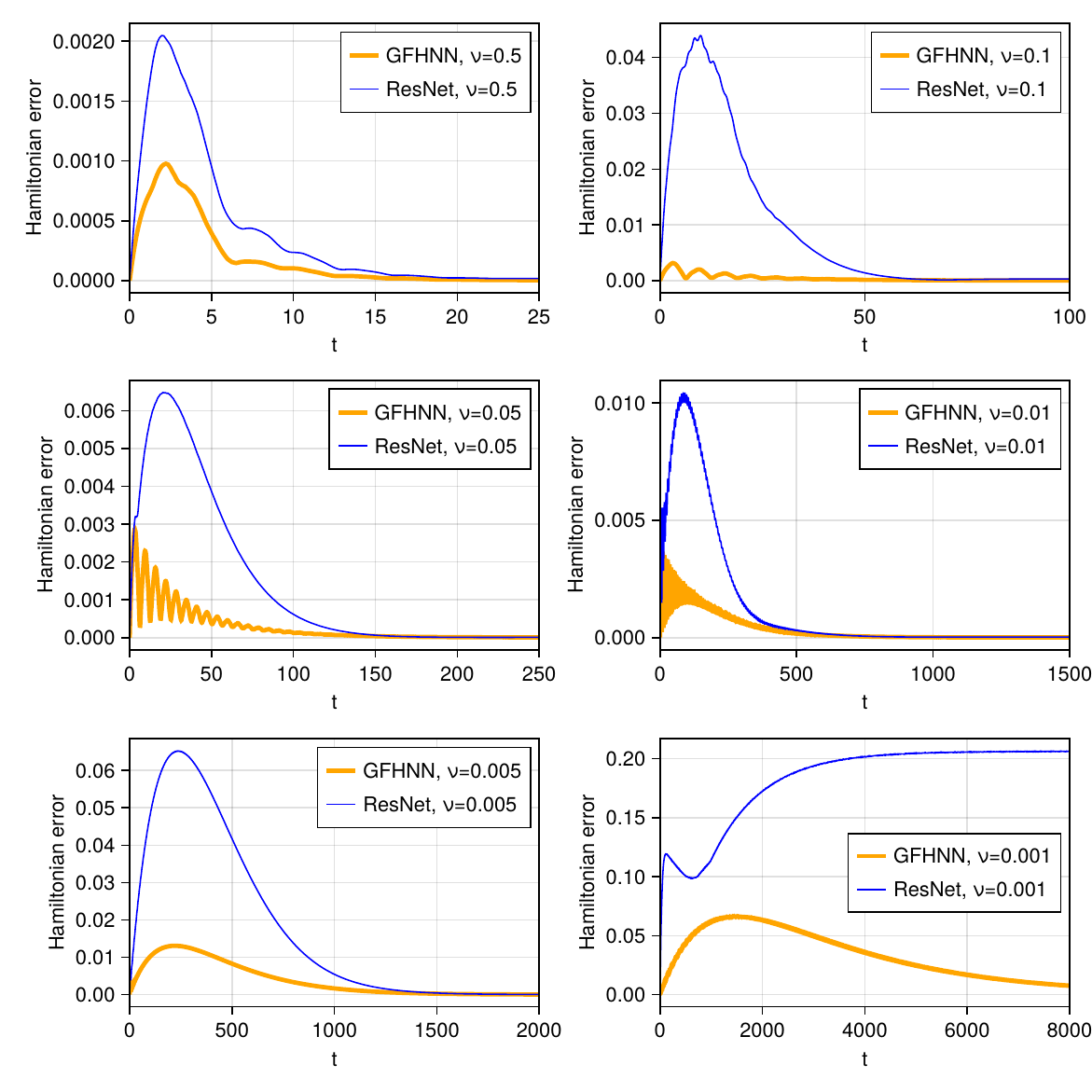}
	\caption{Averaged Hamiltonian error $\epsilon_H(t)$ for trajectories generated by the GFHNN and ResNet flows for the linearly damped harmonic oscillator. The error is computed by averaging over the same set of 49 test initial conditions used in Figure~\ref{fig:Damped_Oscillator_Trajectory_Error}. Each subplot corresponds to a distinct value of the damping coefficient $\nu$, as indicated in the legend.}
	\label{fig:Damped_Oscillator_Hamiltonian_Error}
\end{figure}

The learned flow was then used to generate trajectories from arbitrary initial conditions. We observed that the GFHNN flow outperformed the ResNet flow in terms of accuracy and stability, especially when generating trajectories over time intervals significantly longer than the characteristic time scale $t_{\mathrm{scale}} = 2\pi/\omega$ of the damped system~\eqref{eq: Linearly damped harmonic oscillator: Hamiltonian and forcing}. An illustrative example is shown in Figure~\ref{fig:Damped_Oscillator_Trajectories_for_nu05_and_nu0001}, where GFHNN and ResNet trajectories starting from the same initial condition $(\bar q_0, \bar p_0) = (-0.75, -0.75)$, together with the evolution of the Hamiltonian along these trajectories, are compared for the cases $\nu = 0.5$ and $\nu = 0.001$. In the first case, the solution is damped very rapidly compared to $t_{\mathrm{scale}} = 6.49$, and both the GFHNN and ResNet flows reproduce the exact behavior accurately. In contrast, in the second case the decay of the solution occurs over time intervals several orders of magnitude longer than the corresponding $t_{\mathrm{scale}} = 6.28$, and the geometric GFHNN flow captures the evolution of the exact solution significantly better than the non--structure-preserving ResNet flow.

However, the behavior of the neural-network--generated trajectories was also observed to depend on the choice of the initial condition. An illustrative example is shown in Figure~\ref{fig:Damped_Oscillator_Trajectories_for_nu0005}, where GFHNN and ResNet trajectories for the case $\nu = 0.005$ are compared for two different initial conditions, namely $(\bar q_0, \bar p_0) = (-0.75, -0.75)$ and $(\bar q_0, \bar p_0) = (0.25, 0.25)$. In the former case, the GFHNN trajectory outperforms the ResNet trajectory, whereas in the latter case the ResNet exhibits slightly better agreement with the reference solution. To account for this dependence, we considered a set of $N_\mathrm{test}=49$ test initial conditions distributed over the square $-1 \leq q,p \leq 1$ in phase space,

\begin{align}
\label{eq: Test initial conditions}
q_0^{i,j}=-\frac{3}{4}+\frac{1}{4}(i-1), \qquad\quad p_0^{i,j}=-\frac{3}{4}+\frac{1}{4}(j-1), \qquad\quad \text{for $i,j=1,\ldots,7,$}
\end{align}

\noindent
none of which belonged to the set of initial conditions \eqref{eq: Initial conditions for the training data} used to generate the training data. For each initial condition, we generated a trajectory $(q^{i,j}(t), p^{i,j}(t))$ and computed the averaged trajectory and Hamiltonian errors, defined respectively as

\begin{align}
\label{eq: Computation of Errors}
\epsilon_{\mathrm{traj}}(t)&=\frac{1}{49} \sum_{i,j=1}^7 \left\| \begin{pmatrix} q^{i,j}(t) \\ p^{i,j}(t) \end{pmatrix} -\begin{pmatrix} \bar q^{i,j}(t) \\ \bar p^{i,j}(t) \end{pmatrix} \right\|_2, \nonumber \\
\epsilon_{H}(t)&=\frac{1}{49} \sum_{i,j=1}^7 \Big| H\big(q^{i,j}(t), p^{i,j}(t)\big) - H\big(\bar q^{i,j}(t), \bar p^{i,j}(t)\big)\Big|,
\end{align}

\noindent
where $(\bar q^{i,j}(t), \bar p^{i,j}(t))$ denotes the exact solution \eqref{eq:Linearly damped oscillator---exact solution} corresponding to the initial condition \eqref{eq: Test initial conditions}. The resulting errors are shown in Figure~\ref{fig:Damped_Oscillator_Trajectory_Error} and Figure~\ref{fig:Damped_Oscillator_Hamiltonian_Error}, respectively, for all values of the damping coefficient listed in \eqref{eq: Friction coefficients for linearly damped harmonic oscillator}. As is evident from all plots, the geometric GFHNN flow demonstrates superior performance in all cases.

\subsection{Quadratically damped pendulum}
\label{sec: Quadratically damped pendulum}

As the second example, we consider the quadratically damped pendulum, which is again a system of the form \eqref{eq: Autonomous forced Hamiltonian system} with

\begin{equation}
\label{eq: Quadratically damped pendulum: Hamiltonian and forcing}
H(q,p) = \frac{1}{2}p^2 + \cos(q), \qquad\qquad f(q,p)=-\nu \cdot |p| p,
\end{equation}

\noindent
where $\nu$ is the friction coefficient. Unlike the linearly damped harmonic oscillator considered in \Cref{sec: Linearly damped harmonic oscillator}, this system does not admit a closed-form analytical solution. Therefore, we generate reference trajectories using the \emph{implicit midpoint method}. The training data for our numerical experiment were created by integrating the system $0\leq t \leq T$ with the final time $T=13$, the time step $\Delta t = 0.13$, and for 400 initial conditions $(\bar q_0,\bar p_0)$ uniformly distributed in the square $-1\leq q,p\leq 1$ in the phase space (see \eqref{eq: Initial conditions for the training data}). A total of 6 training data sets were created for different values of the friction coefficient $\nu$, namely,

\begin{equation}
\label{eq: Friction coefficients for quadratically damped pendulum}
\nu = 0.5,\quad 0.1,\quad 0.05,\quad 0.01,\quad 0.005,\quad 0.001.
\end{equation}

\noindent
Similarly to the experiments for the linearly damped harmonic oscillator, a GFHNN and a ResNet were trained for $10\,000$ epochs using the Adam optimizer in order to learn the flow $F_{\Delta t}$. All training runs exhibited convergence behavior comparable to that shown in \Cref{fig:Damped_Oscillator_Training_loss_for_nu01}.

\begin{figure}
	\centering
		\includegraphics[width=1.00\textwidth]{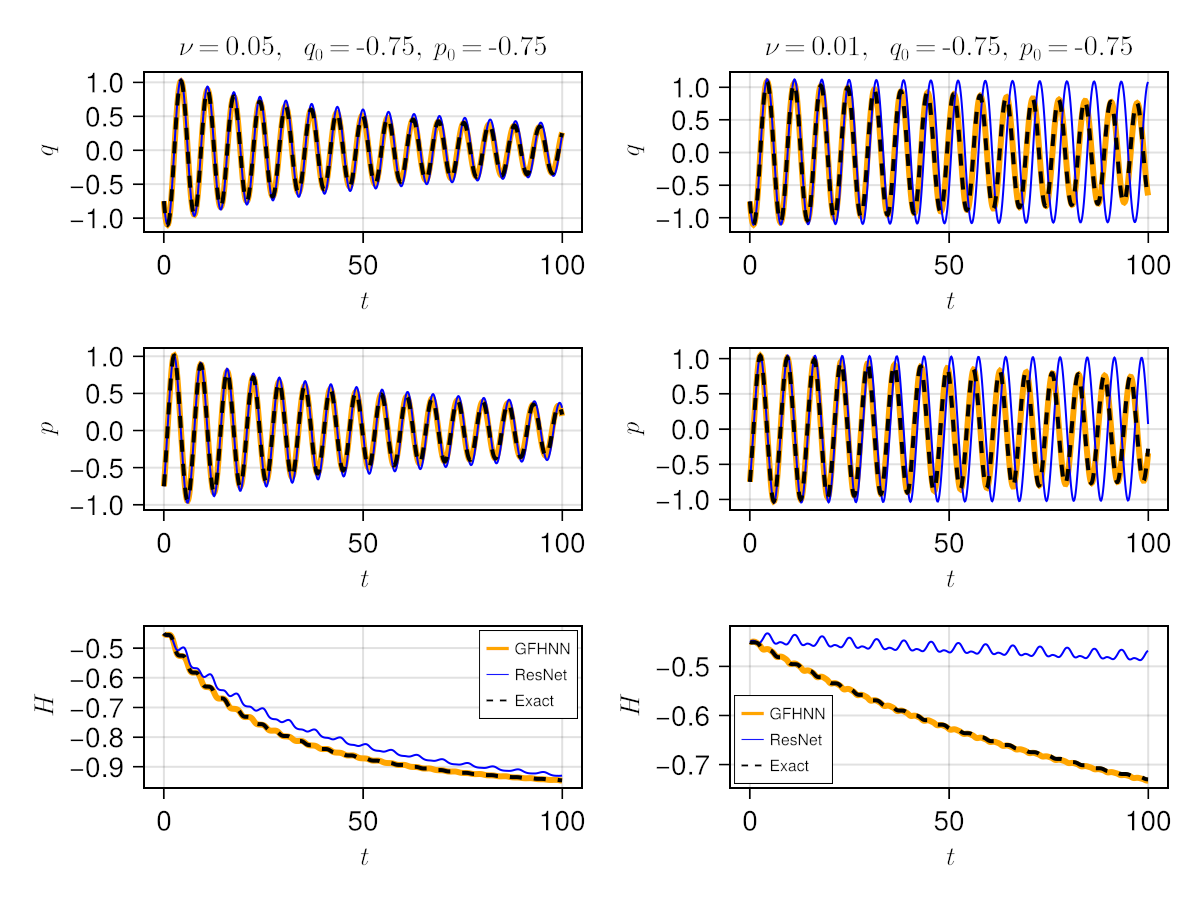}
	\caption{Comparison of trajectories generated by the GFHNN and ResNet flows for the quadratically damped pendulum. The left column corresponds to the case $\nu = 0.05$, while the right column shows the case $\nu = 0.01$. In both cases, trajectories are generated from the same initial condition $(\bar q_0, \bar p_0) = (-0.75, -0.75)$. The first row displays the position $q(t)$, the second row the momentum $p(t)$, and the third row the Hamiltonian $H(t)$ evaluated along the corresponding trajectories.}
	\label{fig:Damped_Oscillator_Trajectories_for_mu05_and_mu01}
\end{figure}

\begin{figure}
	\centering
		\includegraphics[width=1.00\textwidth]{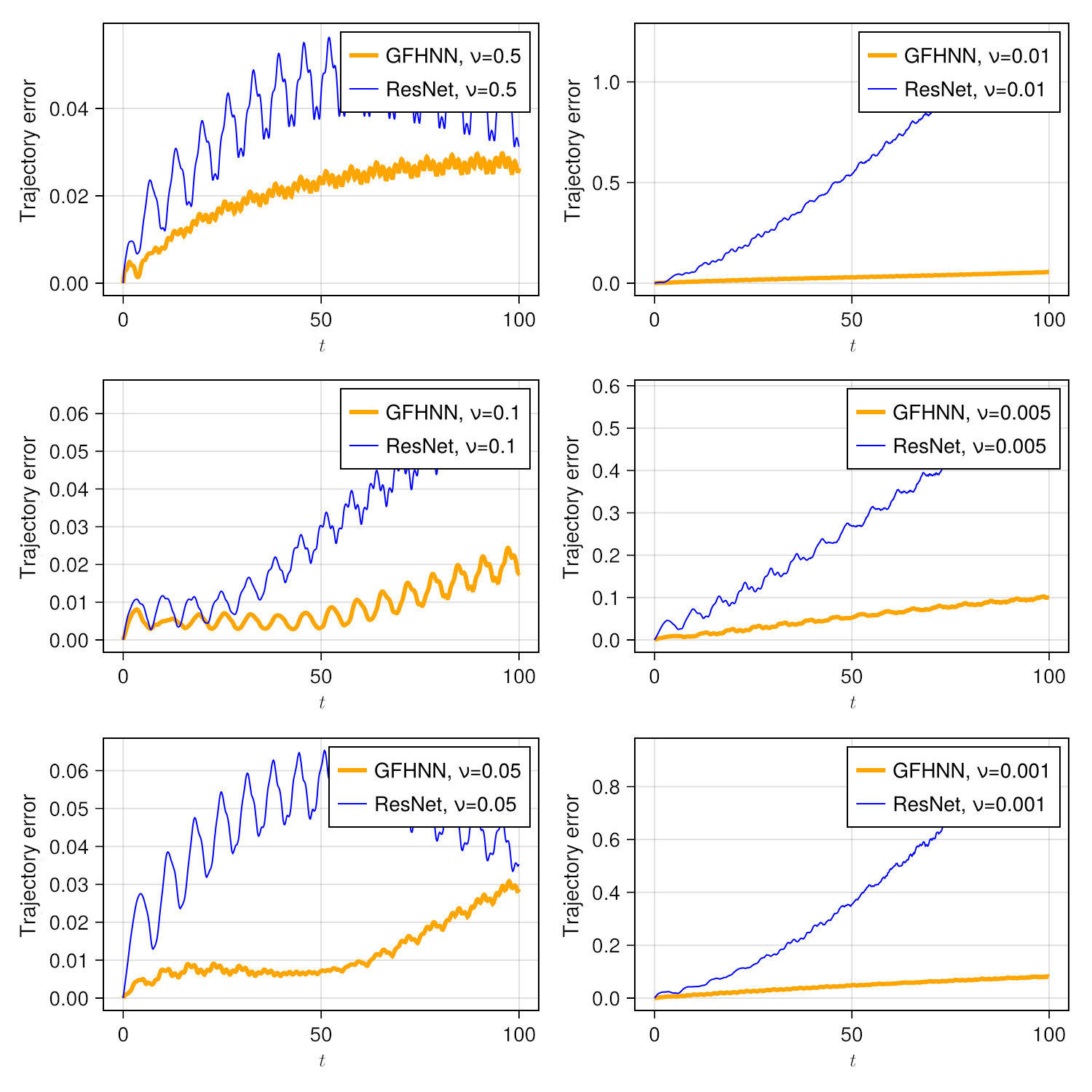}
	\caption{Averaged trajectory error $\epsilon_{\mathrm{traj}}(t)$ for trajectories generated by the GFHNN and ResNet flows for the quadratically damped pendulum. For each value of the damping coefficient $\nu$, the error is computed by averaging over 49 test initial conditions distributed in the square $-0.75 \leq q,p \leq 0.75$ in phase space, none of which were used during training. Each subplot corresponds to a distinct value of $\nu$, as indicated in the legend.}
	\label{fig:Damped_Pendulum_Trajectory_Error}
\end{figure}

\begin{figure}
	\centering
		\includegraphics[width=1.00\textwidth]{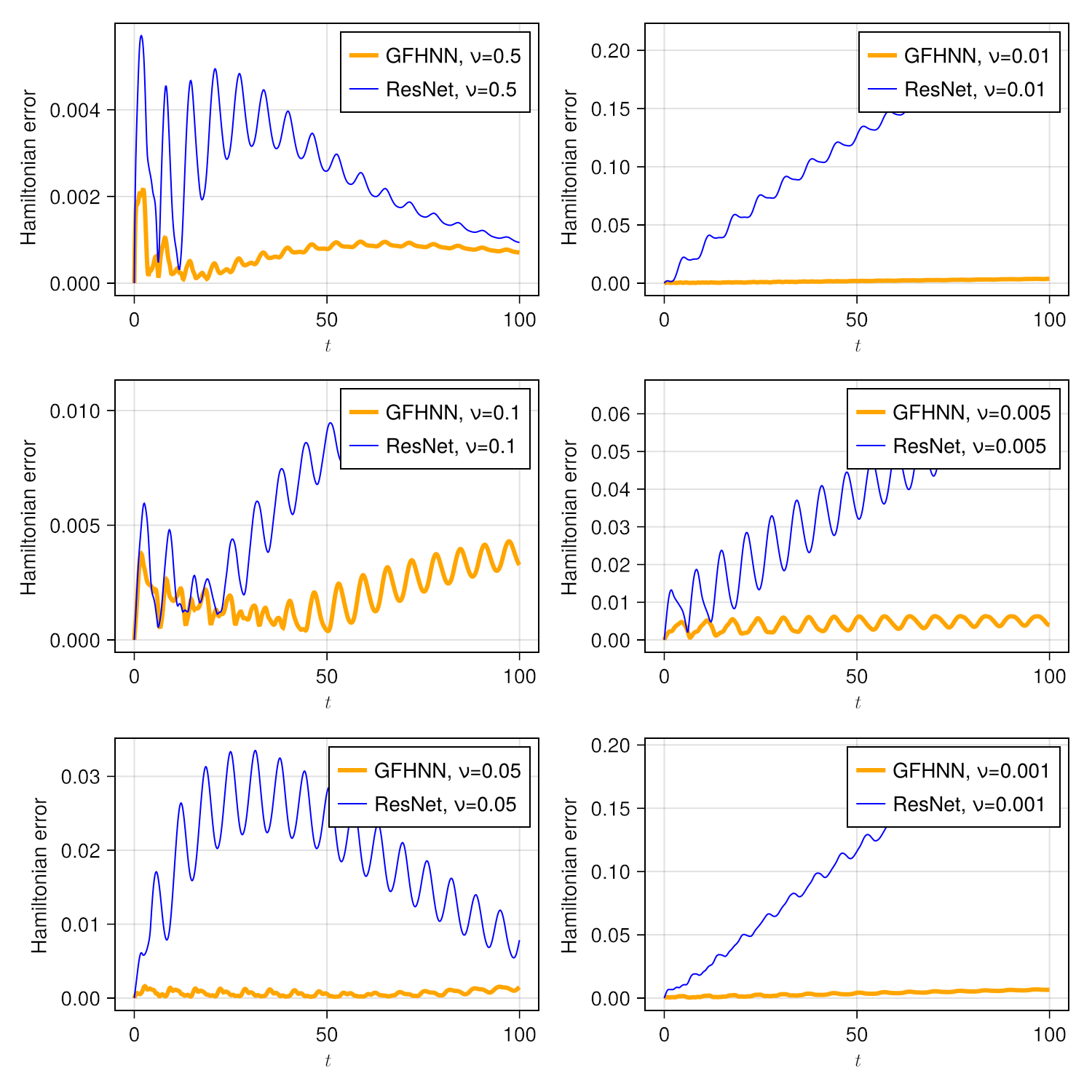}
	\caption{Averaged Hamiltonian error $\epsilon_H(t)$ for trajectories generated by the GFHNN and ResNet flows for the quadratically damped pendulum. The error is computed by averaging over the same set of 49 test initial conditions used in \Cref{fig:Damped_Pendulum_Trajectory_Error}. Each subplot corresponds to a distinct value of the damping coefficient $\nu$, as indicated in the legend. Also compare this to the case of the linearly damped harmonic oscillator in \Cref{fig:Damped_Oscillator_Hamiltonian_Error}.}
	\label{fig:Damped_Pendulum_Hamiltonian_Error}
\end{figure}

In \Cref{fig:Damped_Oscillator_Trajectories_for_mu05_and_mu01}, we compare representative trajectories generated by the trained GFHNN and ResNet models. More precisely, we compare the evolution of the position $q(t)$, the momentum $p(t)$, and the Hamiltonian $H(t)$ for two different values of the damping coefficient, namely $\nu=0.05$ and $\nu=0.01$, using the same initial condition $(\bar q_0, \bar p_0)=(-0.75,-0.75)$. In both cases, the GFHNN produces more accurate trajectories than the ResNet. This difference is particularly pronounced for $\nu=0.01$, where the ResNet appears unable to accurately reproduce the dissipative behavior of the system.

To account for the dependence on the initial condition, we again consider a set of $49$ test initial conditions distributed over the square $-0.75\leq q,p\leq0.75$ in phase space (see \eqref{eq: Test initial conditions}). None of these initial conditions belonged to the set of initial conditions \eqref{eq: Initial conditions for the training data} used to generate the training data. As before, we generated a trajectory $(q^{i,j}(t), p^{i,j}(t))$ for every initial condition and computed the averaged trajectory and Hamiltonian errors defined in \Cref{eq: Computation of Errors}. The resulting errors are shown in \Cref{fig:Damped_Pendulum_Trajectory_Error,fig:Damped_Pendulum_Hamiltonian_Error}, respectively, for all values of the damping coefficient listed in \eqref{eq: Friction coefficients for quadratically damped pendulum}. The plots show that the geometric GFHNN flow consistently achieves lower errors than the ResNet in all cases.

\subsection{Time-dependent damped harmonic oscillator}
\label{sec: Time-dependent damped harmonic oscillator}

In order to test the applicability of the proposed PGFHNN architecture to learning time-dependent flows, we consider the time-dependent forced Hamiltonian system \eqref{eq: Time-dependent forced Hamiltonian system} with

\begin{equation}
\label{eq: Time-dependent damped oscillator: Hamiltonian and forcing}
H(q,p,t) = \frac{1}{2}p^2 + \frac{1}{2}q^2 - V_0 q \sin \Omega t, \qquad\qquad f(q,p)=-\nu p,
\end{equation}

\noindent
where the Hamiltonian $H$ contains a time-periodic potential force with amplitude $V_0$ and angular frequency $\Omega$, and the system is damped by a non-conservative friction force with friction coefficient~$\nu$. For this system an analytic solution is available and is given by

\begin{align}
\label{eq:Time-dependent damped oscillator---exact solution}
\bar q(t)&= A e^{-\frac{\nu}{2}t} \sin \omega t + B e^{-\frac{\nu}{2}t} \cos \omega t + \gamma V_0 (1-\Omega^2) \sin \Omega t - \gamma V_0 \nu \Omega \cos \Omega t, \nonumber \\
\bar p(t)&= \Big(-\frac{\nu}{2}A-\omega B\Big) e^{-\frac{\nu}{2}t} \sin \omega t + \Big(\omega A - \frac{\nu}{2} B\Big) e^{-\frac{\nu}{2}t} \cos \omega t + \gamma V_0 \nu \Omega^2 \sin \Omega t + \gamma V_0 \Omega (1-\Omega^2) \cos \Omega t,
\end{align}

\noindent
with

\begin{align}
\label{eq: Parameters in the exact solution for the time-dependent damped oscillator}
\gamma = \frac{1}{(1-\Omega^2)^2+\nu^2 \Omega^2}, \qquad A = \frac{1}{\omega}\bigg(\bar p_0+\frac{\nu}{2} \bar q_0 + \gamma V_0 \Omega \Big(\Omega^2+\frac{\nu^2}{2}-1 \Big) \bigg), \qquad B = \bar q_0 + \gamma V_0\nu \Omega,
\end{align}

\noindent
where $\bar q_0$ and $\bar p_0$ denote the initial conditions, the damped natural angular frequency is $\omega=\frac{1}{2}\sqrt{4-\nu^2}$, and we assume the underdamped case $0\leq \nu < 2$. The flow map $F_{t,t_0}$ associated with this system is time-dependent. Since the Hamiltonian is periodic in time with period $T_\mathrm{p}=2\pi/\Omega$, one can easily verify that $F_{t+T_\mathrm{p},t_0+T_\mathrm{p}}=F_{t,t_0}$.

The training data for our numerical experiments were created by choosing $V_0=1$, $\Omega=3.5$, and $\nu=0.25$, and sampling the exact solution \eqref{eq:Linearly damped oscillator---exact solution} for $0\leq t \leq T$ with the final time $T=5T_\mathrm{p}\approx 8.98$, time step $\Delta t = T/100\approx 0.0898$, and $N_\mathrm{train}=625$ initial conditions $(\bar q_0,\bar p_0)$ defined in \eqref{eq: Initial conditions for the training data}. Using this data set, a PGFHNN and a parametric ResNet, the latter obtained by augmenting the network input with the parameter, were trained for 300 epochs with the Adam optimizer to learn the time-dependent flow map $F_{t+\Delta t, t}$ of the system \eqref{eq: Time-dependent forced Hamiltonian system}, with time $t$ treated as a parameter. To make the comparison fair, the networks were chosen to have comparable numbers of trainable parameters, namely, 186 for the PGFHNN (consisting of 3 LDE blocks with $\tilde T_i$, $\tilde U_i$, and $\tilde f_i$ each having 2 hidden layers of width 2) and 180 for the parametric ResNet. Since the ResNet turned out to yield significantly worse predictions than the PGFHNN, an additional ResNet was trained using a substantially larger training data set with $N_\mathrm{train}=4761$.

\begin{figure}
	\centering
		\includegraphics[width=1.00\textwidth]{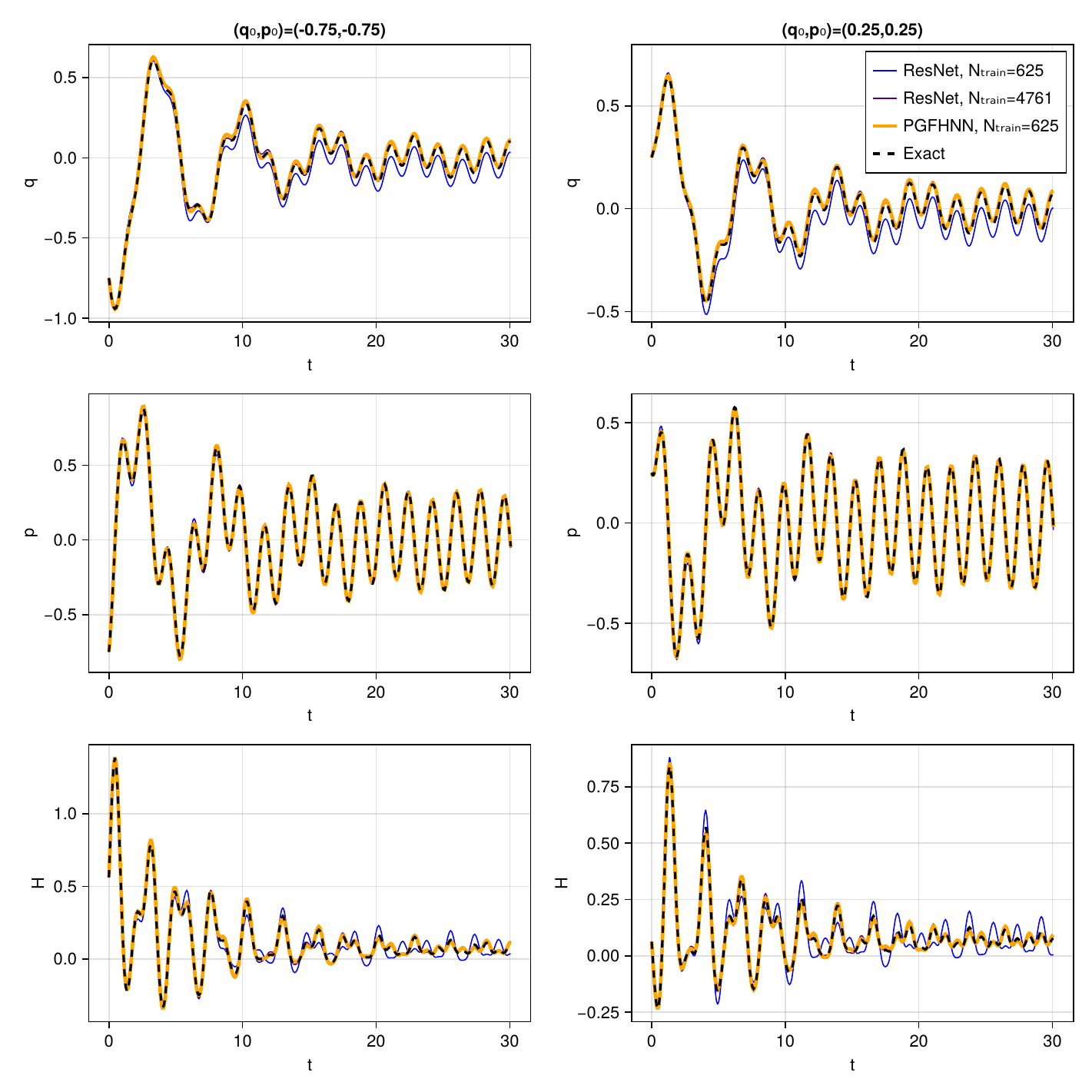}
	\caption{Comparison of trajectories generated by the PGFHNN and parametric ResNet flows for the time-dependent damped harmonic oscillator with $V_0=1$, $\Omega=3.5$, and $\nu = 0.25$. The left column corresponds to the initial condition $(\bar q_0, \bar p_0) = (-0.75, -0.75)$, while the right column corresponds to $(\bar q_0, \bar p_0) = (0.25, 0.25)$. In both cases, the first row displays the position $q(t)$, the second row the momentum $p(t)$, and the third row the Hamiltonian $H(t)$ evaluated along the corresponding trajectories. Note that the plots corresponding to the ResNet with $N_\mathrm{train}=4761$ and the PGFHNN overlap very closely and are therefore nearly indistinguishable in the figure. }
	\label{fig:Time_Dependent_Damped_Oscillator_Trajectories}
\end{figure}

\begin{figure}
	\centering
		\includegraphics[width=1.00\textwidth]{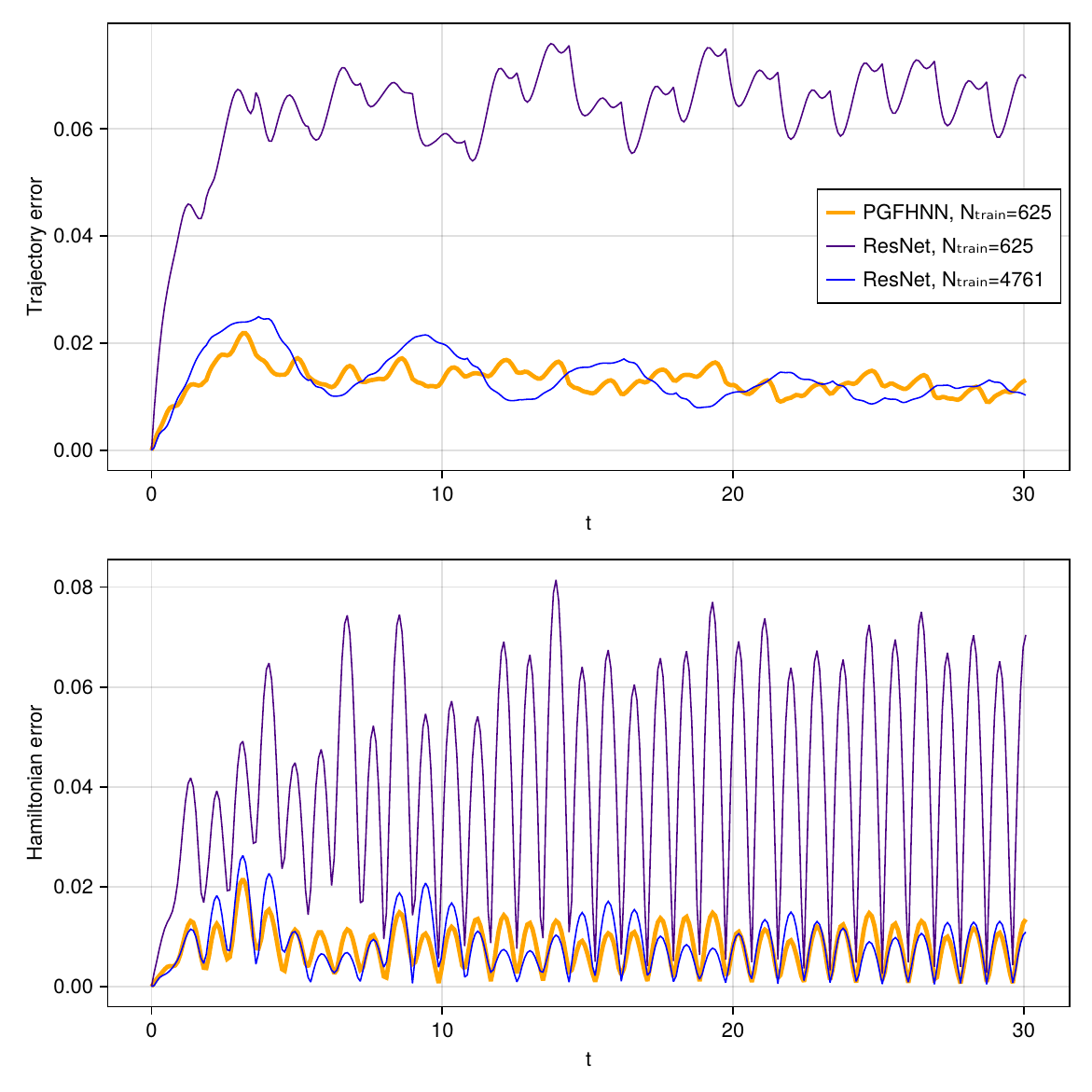}
	\caption{Averaged trajectory $\epsilon_{\mathrm{traj}}(t)$ (\emph{Top}) and Hamiltonian $\epsilon_H(t)$ (\emph{Bottom}) errors for trajectories generated by the PGFHNN and parametric ResNet flows for the time-dependent damped harmonic oscillator with $V_0=1$, $\Omega=3.5$, and $\nu = 0.25$. The errors are computed by averaging over 49 test initial conditions distributed in the square $-1 \leq q,p \leq 1$ in phase space, none of which were used during training. In order to reach a comparable accuracy, the non-structure-preserving ResNet needed a 7.6-times larger training data set than the geometric PGFHNN.}
	\label{fig:Time_Dependent_Damped_Oscillator_TrajectoryHamiltonian_Error}
\end{figure}

The learned flow was then used to generate trajectories from arbitrary initial conditions. Because the neural networks were trained only on the interval $t \in [0,T]$, direct evaluation for $t > T$ results in substantial errors. To mitigate this issue, we periodically extended the learned models beyond $[0,T]$, consistent with the time periodicity of the exact flow map. An illustrative example is shown in Figure~\ref{fig:Time_Dependent_Damped_Oscillator_Trajectories}, where PGFHNN and ResNet trajectories are compared for two different initial conditions, namely $(\bar q_0, \bar p_0) = (-0.75, -0.75)$ and $(\bar q_0, \bar p_0) = (0.25, 0.25)$. One may observe that, when trained on the same data set with $N_\mathrm{train}=625$, the ResNet yields a less accurate evolution than the PGFHNN. In fact, in order to reach the level of accuracy of the PGFHNN, the ResNet required a larger training data set with \(N_\mathrm{train}=4761\). To further investigate this behavior, we calculated the trajectory and Hamiltonian errors \eqref{eq: Computation of Errors}, averaged over 49 trajectories with the test initial conditions \eqref{eq: Test initial conditions}. The results are depicted in Figure~\ref{fig:Time_Dependent_Damped_Oscillator_TrajectoryHamiltonian_Error}, and they clearly show that, due to the lack of structure preservation, the ResNet needs more training data to achieve precision comparable to that of the PGFHNN.

\subsection{Forced Kubo oscillator}
\label{sec: Forced Kubo oscillator}

As a final example, we present the results of our numerical experiments for stochastic systems. We consider the forced Kubo oscillator, which is a stochastic forced Hamiltonian system of the form \eqref{eq: Stochastic dissipative Hamiltonian system} with the Hamiltonians and forcing terms, respectively,

\begin{align}
\label{eq: Forced Kubo: Hamiltonians and forces}
H_0(q,p)&=p^2/2+q^2/2, &  f_0(q,p)&=-\nu p, \nonumber\\
H_1(q,p)&=\beta(p^2/2+q^2/2), & f_1(q,p)&=-\beta \nu p,
\end{align}

\noindent
where $\nu$ is the damping coefficient, and $\beta$ is the noise intensity. It is an example of an oscillator with a fluctuating frequency and it was first introduced in the context of line-shape theory \cite{Anderson1954,Kubo1954}. It has since found numerous applications in mechanical systems, turbulence, laser theory, and wave propagation \cite{VanKampen1976}, magnetic resonance spectroscopy and nonlinear spectroscopy \cite{MukamelBook1995}, single molecule spectroscopy \cite{Jung2003}, and stochastic resonance \cite{ChaudhuriMicroscopic2009,Chaudhuri2010,ChaudhuriNonequilibrium2009,Gitterman2004}. The Kubo oscillator also serves as a prototype for multiplicative stochastic processes, and since its solutions can be calculated analytically, it is frequently used to validate numerical algorithms \cite{Fox1987,MaDing2015,MilsteinRepin,SunWang2016}. It is straightforward to verify that the exact solution is given by

\begin{align}
\label{eq:Forced Kubo oscillator---exact solution}
\bar q(t)&= \bar q_0 e^{-\frac{\nu}{2}(t+\beta W(t))} \cos \omega \big(t+\beta W(t)\big) + \frac{1}{\omega}\Big(\bar p_0+\frac{\nu}{2} \bar q_0\Big) e^{-\frac{\nu}{2}(t+\beta W(t))} \sin \omega \big(t+\beta W(t)\big), \nonumber \\
\bar p(t)&= \bar p_0 e^{-\frac{\nu}{2}(t+\beta W(t))} \cos \omega \big(t+\beta W(t)\big) - \frac{1}{\omega}\Big(\bar q_0+\frac{\nu}{2} \bar p_0\Big) e^{-\frac{\nu}{2}(t+\beta W(t))} \sin \omega \big(t+\beta W(t)\big),
\end{align}

\noindent
where $\bar q_0$ and $\bar p_0$ denote the initial conditions, the angular frequency is $\omega=\frac{1}{2}\sqrt{4-\nu^2}$, and we assume the underdamped regime $0\leq \nu < 2$; see \cite{KrausTyranowski2019}. Note that \eqref{eq:Forced Kubo oscillator---exact solution} coincides with the solution of the deterministic damped harmonic oscillator \eqref{eq:Linearly damped oscillator---exact solution}, with the time argument shifted by $\beta W(t)$.

The training data for our experiment were generated by choosing $\beta=0.5$ and sampling the exact solution \eqref{eq:Forced Kubo oscillator---exact solution} for $0\leq t \leq T$, with final time $T=7$ and time step $\Delta t = 0.07$. We used $N_\mathrm{train}=100$ initial conditions $(\bar q_0,\bar p_0)$ defined in \eqref{eq: Initial conditions for the training data}, and $M_\mathrm{train}=10$ independent sample paths of the Wiener process. Two training data sets were created for different values of the friction coefficient $\nu$, namely,

\begin{equation}
\label{eq: Friction coefficients for the forced Kubo oscillator}
\nu = 0.25,\quad\quad \text{and} \quad\quad \nu=0.01.
\end{equation}

\noindent
Using these data sets, a PGFHNN and a parametric ResNet were trained for 300 epochs with the Adam optimizer to learn the parameter-dependent Lagrange-d'Alembert map $\varphi_\mu$ underlying the stochastic flow $F_{t+\Delta t, t}$ of the system \eqref{eq: Stochastic dissipative Hamiltonian system}, as in \eqref{eq: Stochastic flow as a parameter-dependent flow}, with the increment of the Wiener process $\Delta W$ treated as a parameter. To make the comparison fair, the networks were chosen to have comparable numbers of trainable parameters, namely, 186 for the PGFHNN (consisting of 3 LDE blocks with $\tilde T_i$, $\tilde U_i$, and $\tilde f_i$ each having 2 hidden layers of width 2) and 180 for the parametric ResNet.

\begin{figure}
	\centering
		\includegraphics[width=1.00\textwidth]{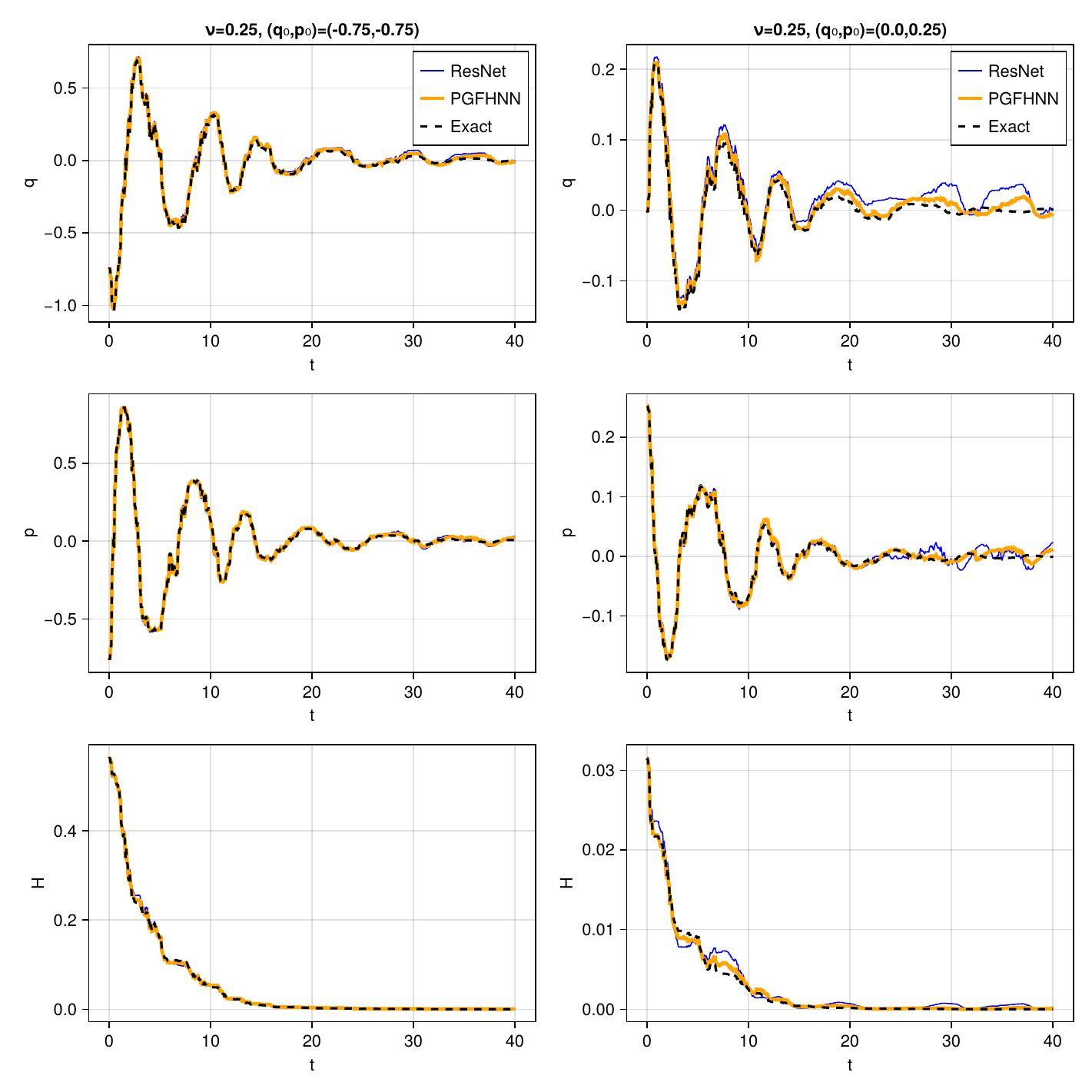}
	\caption{Comparison of trajectories generated by the PGFHNN and parametric ResNet flows for the forced Kubo oscillator with $\beta=0.5$ and $\nu = 0.25$. The left column corresponds to the initial condition $(\bar q_0, \bar p_0) = (-0.75, -0.75)$, while the right column corresponds to $(\bar q_0, \bar p_0) = (0, 0.25)$. In both cases, the first row displays the position $q(t)$, the second row the momentum $p(t)$, and the third row the Hamiltonian $H(t)$ evaluated along the corresponding trajectories.}
	\label{fig:Forced_Kubo_Trajectories_nu025}
\end{figure}

\begin{figure}
	\centering
		\includegraphics[width=1.00\textwidth]{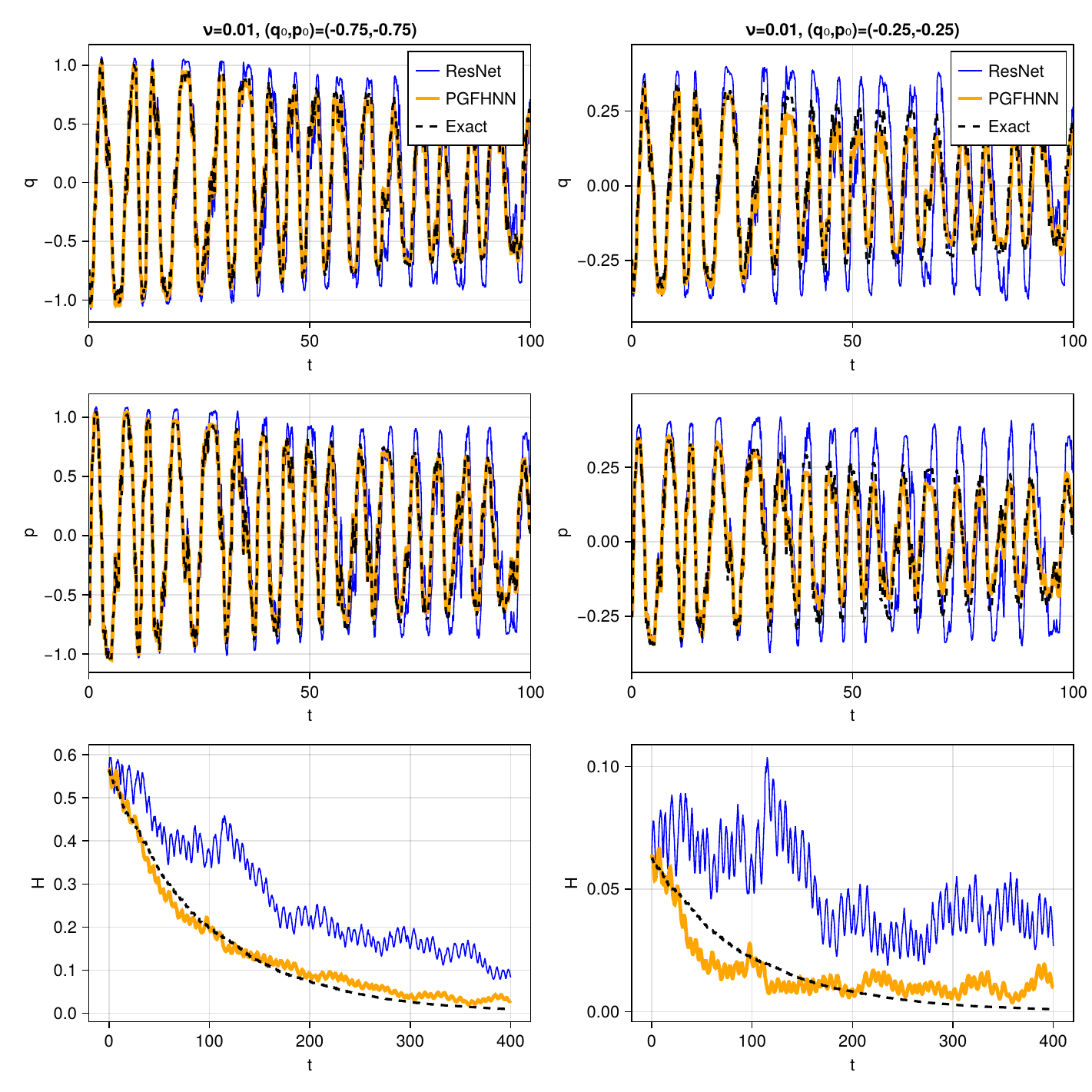}
	\caption{Comparison of trajectories generated by the PGFHNN and parametric ResNet flows for the forced Kubo oscillator with $\beta=0.5$ and $\nu = 0.01$. The left column corresponds to the initial condition $(\bar q_0, \bar p_0) = (-0.75, -0.75)$, while the right column corresponds to $(\bar q_0, \bar p_0) = (-0.25, -0.25)$. In both cases, the first row displays the position $q(t)$, the second row the momentum $p(t)$, and the third row the Hamiltonian $H(t)$ evaluated along the corresponding trajectories.}
	\label{fig:Forced_Kubo_Trajectories_nu001}
\end{figure}

\begin{figure}
	\centering
		\includegraphics[width=1.00\textwidth]{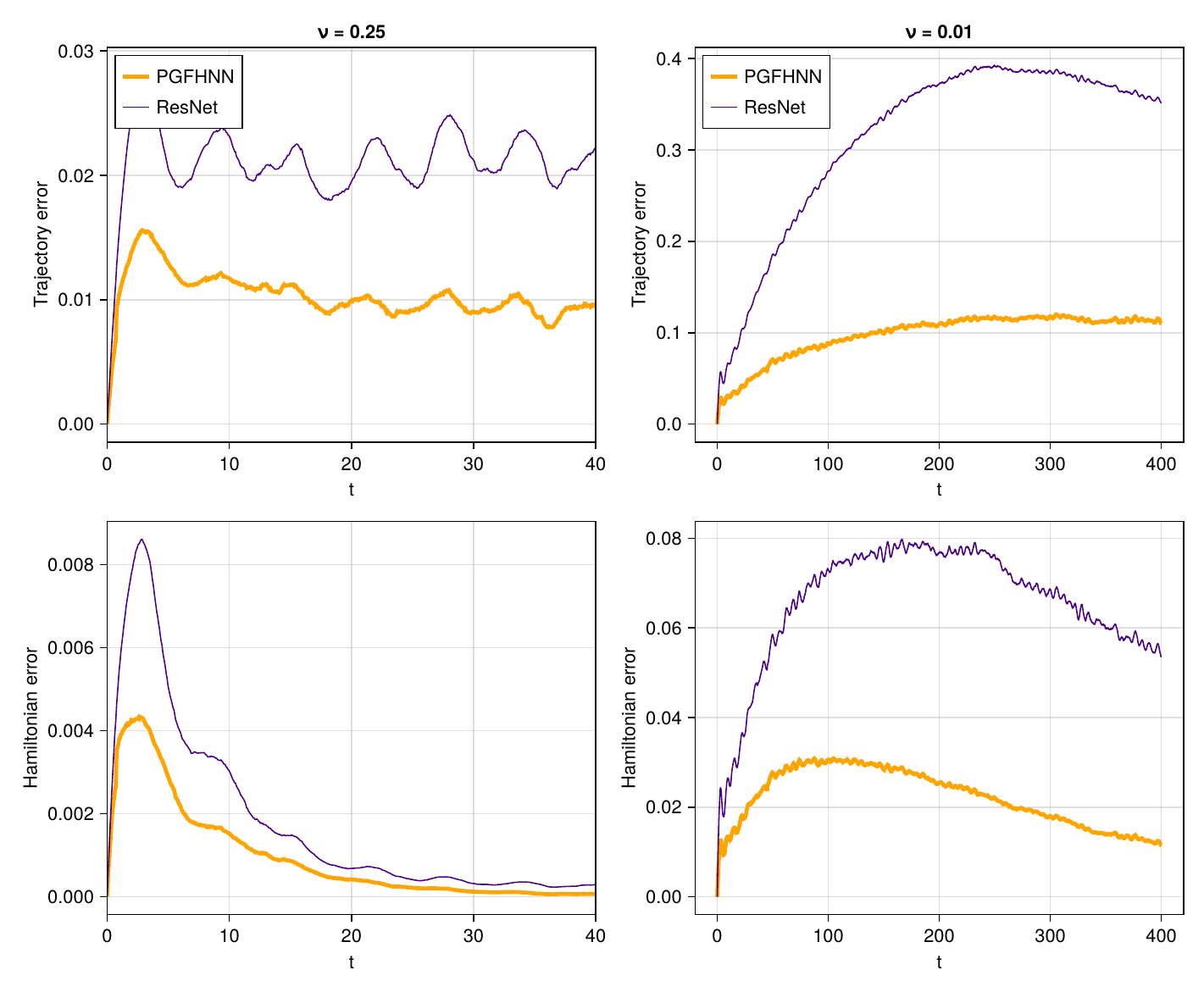}
	\caption{Averaged trajectory $\epsilon_{\mathrm{traj}}(t)$ (\emph{Top}) and Hamiltonian $\epsilon_H(t)$ (\emph{Bottom}) errors for trajectories generated by the PGFHNN and parametric ResNet flows for the forced Kubo oscillator with $\beta=0.5$. The left column corresponds to the friction coefficient $\nu=0.25$, while the right column corresponds to $\nu=0.01$. In both cases, the errors are computed by averaging over 49 test initial conditions distributed in the square $-1 \leq q,p \leq 1$ in phase space, and 49 independent sample paths of the Wiener process, none of which were used during training.}
	\label{fig:Forced_Kubo_TrajectoriesHamiltonian_Error}
\end{figure}

The learned flow was then used to generate trajectories from arbitrary initial conditions, and for arbitrary sample paths of the Wiener process. Illustrative examples are shown in Figure~\ref{fig:Forced_Kubo_Trajectories_nu025} and Figure~\ref{fig:Forced_Kubo_Trajectories_nu001}, where PGFHNN and ResNet trajectories are compared for the $\nu=0.25$ and $\nu=0.01$ data sets, respectively, in each case for two different initial conditions. One may observe that the quality of the ResNet trajectories deteriorates faster than that of the PGFHNN trajectories, especially in the $\nu=0.01$ case, where the damping of the solution occurs over a longer time interval. To further investigate this behavior, we calculated the trajectory and Hamiltonian errors~\eqref{eq: Computation of Errors}, averaged over $N_\mathrm{test}\cdot M_\mathrm{test}=2401$ trajectories, corresponding to $N_\mathrm{test}=49$ initial conditions \eqref{eq: Test initial conditions} and $M_\mathrm{test}=49$ independent sample paths of the Wiener process. The results are depicted in Figure~\ref{fig:Forced_Kubo_TrajectoriesHamiltonian_Error}, and they clearly show that, due to the lack of structure preservation, the ResNet generates less accurate solutions than the PGFHNN.

\section{Summary}
\label{sec: Summary}

In this work, we developed a geometric framework for learning deterministic and stochastic forced Hamiltonian systems with neural networks. Our construction is motivated by the Lagrange-d'Alembert principle and the theory of variational integrators, which provide a natural geometric description of mechanical systems subject to external forcing. We reviewed forced Hamiltonian systems and Lagrange-d'Alembert integrators, introduced the notion of a Lagrange-d'Alembert map, and established a $C^r$ convergence theorem for first-order one-step methods. These results provide a mathematical foundation for approximating forced Hamiltonian flows by compositions of geometric numerical integrators.

Using this foundation, we proposed Generalized Forced Hamiltonian Neural Networks (GFHNNs), a class of structure-preserving neural networks obtained by concatenating Lagrange-d'Alembert-Euler maps whose kinetic and potential energies, together with the forcing term, are represented by neural networks. We proved a universal approximation theorem showing that GFHNNs are dense in the space of Lagrange-d'Alembert maps in the $C^r$ topology on compact sets. We further extended the theory to parameter-dependent forced Hamiltonian systems and introduced Parametric Generalized Forced Hamiltonian Neural Networks (PGFHNNs), together with a corresponding universal approximation theorem.

We also considered stochastic forced Hamiltonian systems. By treating the multiple Stratonovich integrals appearing in the Stratonovich--Taylor expansion as parameters, stochastic Lagrange-d'Alembert flows can be viewed as parameter-dependent Lagrange-d'Alembert maps. This allows PGFHNNs to be applied directly to stochastic forced Hamiltonian systems whenever information about the underlying Wiener process is available, thereby extending the approximation framework developed for deterministic systems to this stochastic setting.

The proposed architectures were evaluated on several benchmark problems involving deterministic and stochastic forced Hamiltonian systems. The numerical results demonstrate that incorporating geometric structure through variational principles and Lagrange-d'Alembert integrators yields models with significantly improved long-time stability and accuracy. In particular, the proposed geometric architectures consistently outperform non-geometric residual neural networks in long-time simulations. We also observed that residual neural networks require substantially larger training datasets to achieve a comparable level of accuracy.

Several directions for future research remain open. A natural extension of the present work is the development of analogous architectures on manifolds, in particular on Lie groups, with potential applications in robotics \cite{DuruisseauxLeok2023,Saemundsson2020}. Another promising area of application is computational plasma physics. Particle discretizations of collisional Vlasov equations possess the structure of stochastic forced Hamiltonian systems \cite{KrausTyranowski2019,LuMengTyranowski2025,TyranowskiVlasovMaxwell}, and many practical applications require repeated simulations using expensive high-fidelity stochastic numerical methods with millions of particles and varying input parameters. In such settings, the underlying Wiener process is known, and the proposed neural network architecture could significantly reduce computational costs by providing a cheaper surrogate for the stochastic flow. This approach could be particularly effective when combined with structure-preserving model reduction techniques \cite{BrantnerKraus2023,TyranowskiStochasticModelReduction,TyranowskiKraus2021}. Finally, our approach to learning stochastic flows could be combined with denoising techniques, such as the autoencoder-based methods developed in \cite{ChenWang2025,Xu2024}, to infer latent random variables. This would make it possible to apply the proposed framework to engineering problems in which the underlying Wiener process is not directly observable.

\section*{Acknowledgments}

We would like to thank Philipp Horn for useful comments and references. We further thank Michael Kraus for useful comments and a thorough review of the code that was used to produce the results presented here. This publication is part of the project \emph{Stochastic Geometric Integrators for Dynamical Systems} with file number OCENW.M.24.105 of the research programme Open Competition Domain Science-M Package 24-2, which is (partly) financed by the Dutch Research Council (NWO) under grant DOI \href{https://doi.org/10.61686/PIDBU22657} {https://doi.org/10.61686/PIDBU22657}.

\FloatBarrier

\bibliographystyle{abbrv}
\bibliography{bibliography}

@article{Anderson1954,
author = {Anderson, P.W.},
title = {A Mathematical Model for the Narrowing of Spectral Lines by Exchange or Motion},
journal = {Journal of the Physical Society of Japan},
volume = {9},
number = {3},
pages = {316-339},
year = {1954},
doi = {10.1143/JPSJ.9.316}
}

@book{ArnoldSDE,
  title={Stochastic {D}ifferential {E}quations: {T}heory and {A}pplications},
  author={Arnold, L.},
  isbn={9780486482361},
  lccn={2010044194},
  series={Dover Books on Mathematics},
  year={2013},
  publisher={Dover Publications}
}

@unpublished{BrantnerKraus2023,
title={Symplectic Autoencoders for Model Reduction of {H}amiltonian Systems},
author={Benedikt Brantner and Michael Kraus},
year={2023},
note={Preprint ar{X}iv:2312.10004}
}

@article{BurbyHenonNets2020,
doi = {10.1088/1361-6587/abcbaa},
url = {https://doi.org/10.1088/1361-6587/abcbaa},
year = {2020},
month = {dec},
publisher = {IOP Publishing},
volume = {63},
number = {2},
pages = {024001},
author = {Burby, J W and Tang, Q and Maulik, R},
title = {Fast neural {P}oincar\'{e} maps for toroidal magnetic fields},
journal = {Plasma Physics and Controlled Fusion}
}

@InProceedings{Chaudhuri2010,
author="Chaudhuri, Jyotipratim Ray
and Chattopadhyay, Sudip",
editor="Chaudhuri, Rajat K.
and Mekkaden, M.V.
and Raveendran, A. V.
and Satya Narayanan, A.",
title="Kubo Oscillator and its Application to Stochastic Resonance: A Microscopic Realization",
booktitle="Recent Advances in Spectroscopy",
year="2010",
publisher="Springer Berlin Heidelberg",
address="Berlin, Heidelberg",
pages="75--83",
isbn="978-3-642-10322-3"
}

@article{ChaudhuriMicroscopic2009,
title = {Microscopic realization of {K}ubo oscillator},
journal = "Chemical Physics Letters",
volume = "480",
number = "1",
pages = "140 - 143",
year = "2009",
doi = "https://doi.org/10.1016/j.cplett.2009.08.057",
author = "Jyotipratim Ray Chaudhuri and Sudip Chattopadhyay"
}

@article{ChaudhuriNonequilibrium2009,
author = {Chaudhuri,Jyotipratim Ray  and Chaudhury,Pinaki  and Chattopadhyay,Sudip },
title = {Harmonic oscillator in presence of nonequilibrium environment},
journal = {The Journal of Chemical Physics},
volume = {130},
number = {23},
pages = {234109},
year = {2009},
doi = {10.1063/1.3155698}
}

@article{ChenChen1995,
  title={Universal approximation to nonlinear operators by neural networks with arbitrary activation functions and its application to dynamical systems},
  author={Chen, Tianping and Chen, Hong},
  journal={IEEE transactions on neural networks},
  volume={6},
  number={4},
  pages={911--917},
  year={1995},
  publisher={IEEE}
}

@inproceedings{ChenRubanova2018,
 author = {Chen, Ricky T. Q. and Rubanova, Yulia and Bettencourt, Jesse and Duvenaud, David K},
 booktitle = {Advances in Neural Information Processing Systems},
 editor = {S. Bengio and H. Wallach and H. Larochelle and K. Grauman and N. Cesa-Bianchi and R. Garnett},
 pages = {},
 publisher = {Curran Associates, Inc.},
 title = {Neural Ordinary Differential Equations},
 url = {https://proceedings.neurips.cc/paper_files/paper/2018/file/69386f6bb1dfed68692a24c8686939b9-Paper.pdf},
 volume = {31},
 year = {2018}
}

@article{ChenMatsubara2021,
  title={Neural symplectic form: Learning {H}amiltonian equations on general coordinate systems},
  author={Chen, Yuhan and Matsubara, Takashi and Yaguchi, Takaharu},
  journal={Advances in Neural Information Processing Systems},
  volume={34},
  pages={16659--16670},
  year={2021}
}

@article{ChenWang2025,
  title        = {Learning Stochastic {H}amiltonian Systems via Stochastic Generating Function Neural Network},
  author       = {Chen, Chen and Wang, Lijin and Cao, Yanzhao and Cheng, Xupeng},
  journal      = {arXiv preprint},
  volume       = {arXiv:2507.14467},
  year         = {2025},
  url          = {https://arxiv.org/abs/2507.14467},
  eprint       = {2507.14467},
  archivePrefix= {arXiv},
  primaryClass = {math.DS}
}

@article{ChenXiu2024,
title = {Learning stochastic dynamical system via flow map operator},
journal = {Journal of Computational Physics},
volume = {508},
pages = {112984},
year = {2024},
issn = {0021-9991},
doi = {https://doi.org/10.1016/j.jcp.2024.112984},
url = {https://www.sciencedirect.com/science/article/pii/S002199912400233X},
author = {Yuan Chen and Dongbin Xiu}
}

@Article{ChengWang2024,
AUTHOR = {Cheng, Xupeng and Wang, Lijin and Cao, Yanzhao},
TITLE = {Quadrature Based Neural Network Learning of Stochastic {H}amiltonian Systems},
JOURNAL = {Mathematics},
VOLUME = {12},
YEAR = {2024},
NUMBER = {16},
ARTICLE-NUMBER = {2438},
URL = {https://www.mdpi.com/2227-7390/12/16/2438},
ISSN = {2227-7390},
DOI = {10.3390/math12162438}
}

@book{CoddingtonLevinsonBook,
  author    = {Coddington, Earl A. and Levinson, Norman},
  title     = {Theory of Ordinary Differential Equations},
  publisher = {McGraw-Hill},
  year      = {1955},
  address   = {New York},
  series    = {International Series in Pure and Applied Mathematics}
}

@article{ConstantineSavits1996,
author = {Gregory M. Constantine and Thomas H. Savits},
title = {A Multivariate {F}a{\`a} di {B}runo Formula with Applications},
journal = {Transactions of the American Mathematical Society},
volume = {348},
number = {2},
pages = {503--520},
year = {1996}
}

@article{CourtesFranckKrausNavoretTremant2025,
  title={Neural non-canonical {H}amiltonian dynamics for long-time simulations},
  author={Court{\`e}s, Cl{\'e}mentine and Franck, Emmanuel and Kraus, Michael and Navoret, Laurent and Tr{\'e}mant, L{\'e}opold},
  journal={arXiv preprint arXiv:2510.01788},
  year={2025}
}

@article{CranmerGreydanusHoyerBattagliaSpergelHo2020,
  title={Lagrangian neural networks},
  author={Cranmer, Miles and Greydanus, Sam and Hoyer, Stephan and Battaglia, Peter and Spergel, David and Ho, Shirley},
  journal={arXiv preprint arXiv:2003.04630},
  year={2020}
}

@article{DecoBrauer1995,
  title={Nonlinear higher-order statistical decorrelation by volume-conserving neural architectures},
  author={Deco, Gustavo and Brauer, Wilfried},
  journal={Neural Networks},
  volume={8},
  number={4},
  pages={525--535},
  year={1995},
  publisher={Elsevier}
}

@article{DesaiMattheakisSondakProtopapasRoberts2021,
  title={Port-{H}amiltonian neural networks for learning explicit time-dependent dynamical systems},
  author={Desai, Shaan A and Mattheakis, Marios and Sondak, David and Protopapas, Pavlos and Roberts, Stephen J},
  journal={Physical Review E},
  volume={104},
  number={3},
  pages={034312},
  year={2021},
  publisher={APS}
}

@article{Dietrich2023,
    author = {Dietrich, Felix and Makeev, Alexei and Kevrekidis, George and Evangelou, Nikolaos and Bertalan, Tom and Reich, Sebastian and Kevrekidis, Ioannis G.},
    title = {Learning effective stochastic differential equations from microscopic simulations: Linking stochastic numerics to deep learning},
    journal = {Chaos: An Interdisciplinary Journal of Nonlinear Science},
    volume = {33},
    number = {2},
    pages = {023121},
    year = {2023},
    month = {02},
    issn = {1054-1500},
    doi = {10.1063/5.0113632},
    url = {https://doi.org/10.1063/5.0113632},
    eprint = {https://pubs.aip.org/aip/cha/article-pdf/doi/10.1063/5.0113632/16744177/023121_1_online.pdf},
}

@INPROCEEDINGS{DridiDrumetzFablet2021,
  author={Dridi, Naura and Drumetz, Lucas and Fablet, Ronan},
  booktitle={2021 29th European Signal Processing Conference (EUSIPCO)},
  title={Learning stochastic dynamical systems with neural networks mimicking the {E}uler-{M}aruyama scheme},
  year={2021},
  volume={},
  number={},
  pages={1990-1994},
  doi={10.23919/EUSIPCO54536.2021.9616068}
}

@InProceedings{DuruisseauxLeok2023,
  title = 	 {Lie Group Forced Variational Integrator Networks for Learning and Control of Robot Systems},
  author =       {Duruisseaux, Valentin and Duong, Thai P. and Leok, Melvin and Atanasov, Nikolay},
  booktitle = 	 {Proceedings of The 5th Annual Learning for Dynamics and Control Conference},
  pages = 	 {731--744},
  year = 	 {2023},
  editor = 	 {Matni, Nikolai and Morari, Manfred and Pappas, George J.},
  volume = 	 {211},
  series = 	 {Proceedings of Machine Learning Research},
  month = 	 {15--16 Jun},
  publisher =    {PMLR},
  url = 	 {https://proceedings.mlr.press/v211/duruisseaux23a.html}
}

@ARTICLE{Fox1987,
  author={Fox, R.F. and Roy, R. and Yu, A.W.},
  title={Tests of numerical simulation algorithms for the {K}ubo oscillator},
  journal={Journal of Statistical Physics},
  year={1987},
	pages={477-487},
  volume={47}
}

@article{Gitterman2004,
  title = {Harmonic oscillator with fluctuating damping parameter},
  author = {Gitterman, M.},
  journal = {Phys. Rev. E},
  volume = {69},
  issue = {4},
  pages = {041101},
  numpages = {4},
  year = {2004},
  month = {Apr},
  publisher = {American Physical Society},
  doi = {10.1103/PhysRevE.69.041101},
  url = {https://link.aps.org/doi/10.1103/PhysRevE.69.041101}
}

@book{goodfellow2016deep,
  title={Deep learning},
  author={Goodfellow, Ian and Bengio, Yoshua and Courville, Aaron and Bengio, Yoshua},
  volume={1},
  year={2016},
  publisher={MIT press Cambridge}
}

@article{GreydanusDzambaYosinski2019,
  title={Hamiltonian neural networks},
  author={Greydanus, Samuel and Dzamba, Misko and Yosinski, Jason},
  journal={Advances in neural information processing systems},
  volume={32},
  year={2019}
}

@software{GeometricMachineLearning,
  author = {Benedikt Brantner and Michael Kraus},
  title = {{GeometricMachineLearning.jl}: v0.5.0},
  month = aug,
  version = {0.5.0},
  year = {2026},
  publisher = {Zenodo},
  doi = {10.5281/zenodo.19677961},
  url = {https://doi.org/10.5281/zenodo.19677961},
  note = {https://doi.org/10.5281/zenodo.19677961}
}

@book{HLWGeometric,
  title={Geometric Numerical Integration: Structure-Preserving Algorithms for Ordinary Differential Equations},
  author={Hairer, E. and Lubich, C. and Wanner, G.G.},
  isbn={9783540430032},
  lccn={2002023334},
  series={Springer Series in Computational Mathematics},
  url={http://books.google.com/books?id=O5CfNSGTP\_EC},
  year={2002},
  publisher={Springer, New York}
}

@book{HWODE1,
  title={Solving Ordinary Differential Equations {I}: Nonstiff Problems},
  author={Hairer, E. and N{\o}rsett, S.P. and Wanner, G.},
  series={Springer Series in Computational Mathematics},
  volume={8},
  edition={2nd},
  year={1993},
  publisher={Springer}
}

@book{HaleODE1969,
  author    = {Jack K. Hale},
  title     = {Ordinary Differential Equations},
  publisher = {Wiley-Interscience},
  address   = {New York},
  year      = {1969}
}

@article{HallLeokSpectral,
 author = {Hall, James and Leok, Melvin},
 title = {Spectral Variational Integrators},
 journal = {Numer. Math.},
 volume = {130},
 number = {4},
 month = {Aug},
 year = {2015},
 pages = {681--740}
}

@misc{HansenCelledoniTapley2025,
  title        = {Learning mechanical systems from real-world data using discrete forced {L}agrangian dynamics},
  author       = {Martine Dyring Hansen and Elena Celledoni and Benjamin Kwanen Tapley},
  howpublished = {arXiv preprint},
  year         = {2025},
  eprint       = {2505.20370},
  archivePrefix = {arXiv},
  primaryClass = {eess.SY},
  doi          = {10.48550/arXiv.2505.20370},
  url          = {https://arxiv.org/abs/2505.20370},
	note         = {Available at https://arxiv.org/abs/2505.20370}
}

@InProceedings{Havens2021,
  title = 	 {Forced Variational Integrator Networks for Prediction and Control of Mechanical Systems},
  author =       {Havens, Aaron and Chowdhary, Girish},
  booktitle = 	 {Proceedings of the 3rd Conference on Learning for Dynamics and Control},
  pages = 	 {1142--1153},
  year = 	 {2021},
  editor = 	 {Jadbabaie, Ali and Lygeros, John and Pappas, George J. and Parrilo, Pablo A. and Recht, Benjamin and Tomlin, Claire J. and Zeilinger, Melanie N.},
  volume = 	 {144},
  series = 	 {Proceedings of Machine Learning Research},
  month = 	 {07 -- 08 June},
  publisher =    {PMLR},
  url = 	 {https://proceedings.mlr.press/v144/havens21a.html}
}

@InProceedings{HeResNets2016,
author = {He, Kaiming and Zhang, Xiangyu and Ren, Shaoqing and Sun, Jian},
title = {Deep Residual Learning for Image Recognition},
booktitle = {Proceedings of the {IEEE} Conference on Computer Vision and Pattern Recognition ({CVPR})},
year = {2016},
pages = {770--778}
}

@Article{HolmTyranowskiGalerkin,
author="Holm, Darryl D. and Tyranowski, Tomasz M.",
title={Stochastic discrete {H}amiltonian variational integrators},
journal="BIT Numerical Mathematics",
year="2018",
volume="58",
number="4",
pages="1009--1048",
issn="1572-9125",
doi="10.1007/s10543-018-0720-2",
url="https://doi.org/10.1007/s10543-018-0720-2"
}

@article{HornKorenGHNN,
title = {A generalized framework of neural networks for {H}amiltonian systems},
journal = {Journal of Computational Physics},
volume = {521},
pages = {113536},
year = {2025},
doi = {https://doi.org/10.1016/j.jcp.2024.113536},
author = {Philipp Horn and Veronica {Saz Ulibarrena} and Barry Koren and Simon {Portegies Zwart}}
}

@misc{HornKorenPGHNN,
author = {Horn, Philipp and Koren, Barry},
title = {Parametric Generalized {H}amiltonian Neural Networks},
year = {2026},
note = {SSRN preprint},
url = {https://papers.ssrn.com/sol3/papers.cfm?abstract_id=6427896}
}

@phdthesis{Horn2026PHD,
  author       = {Horn, Philipp},
  title        = {Structure-Preserving Neural Networks for Hamiltonian Systems},
  school       = {Eindhoven University of Technology},
  year         = {2026},
  address      = {Eindhoven},
  type         = {PhD thesis},
  pages        = {109},
  isbn         = {978-90-386-6675-4},
  url          ={https://research.tue.nl/en/publications/structure-preserving-neural-networks-for-hamiltonian-systems/}
}

@article{HornikStinchcombeWhite1989,
  title={Multilayer feedforward networks are universal approximators},
  author={Hornik, Kurt and Stinchcombe, Maxwell and White, Halbert},
  journal={Neural networks},
  volume={2},
  number={5},
  pages={359--366},
  year={1989},
  publisher={Elsevier}
}

@article{HornikStinchcombeWhite1990,
title = {Universal approximation of an unknown mapping and its derivatives using multilayer feedforward networks},
journal = {Neural Networks},
volume = {3},
number = {5},
pages = {551-560},
year = {1990},
issn = {0893-6080},
doi = {https://doi.org/10.1016/0893-6080(90)90005-6},
url = {https://www.sciencedirect.com/science/article/pii/0893608090900056},
author = {Kurt Hornik and Maxwell Stinchcombe and Halbert White}
}

@article{Hornik1991,
title = {Approximation capabilities of multilayer feedforward networks},
journal = {Neural Networks},
volume = {4},
number = {2},
pages = {251-257},
year = {1991},
issn = {0893-6080},
doi = {https://doi.org/10.1016/0893-6080(91)90009-T},
url = {https://www.sciencedirect.com/science/article/pii/089360809190009T},
author = {Kurt Hornik}
}

@book{IkedaWatanabe1989,
  title={Stochastic {D}ifferential {E}quations and {D}iffusion {P}rocesses},
  author={Ikeda, N. and Watanabe, S.},
  isbn={9784062032315},
  lccn={89212869},
  series={Kodansha scientific books},
  year={1989},
  publisher={North-Holland}
}

@article{JinZhangZhuTangKarniadakis2020,
  title={Symp{N}ets: Intrinsic structure-preserving symplectic networks for identifying {H}amiltonian systems},
  author={Jin, Pengzhan and Zhang, Zhen and Zhu, Aiqing and Tang, Yifa and Karniadakis, George Em},
  journal={Neural Networks},
  volume={132},
  pages={166--179},
  year={2020},
  publisher={Elsevier}
}

@inbook{Jung2003,
author = {Jung, Younjoon and Barkai, Eli and Silbey, Robert J.},
publisher = {John Wiley \& Sons, Ltd},
title = {A Stochastic Theory of Single Molecule Spectroscopy},
booktitle = {Advances in Chemical Physics},
chapter = {4},
pages = {199-266},
doi = {10.1002/0471231509.ch4},
year = {2003}
}

@article{KaneMarsden2000,
author = {Kane, C. and Marsden, J. E. and Ortiz, M. and West, M.},
title = {Variational integrators and the {N}ewmark algorithm for conservative and dissipative mechanical systems},
journal = {International Journal for Numerical Methods in Engineering},
year = {2000},
volume = {49},
number = {10},
pages = {1295-1325},
doi = {10.1002/1097-0207(20001210)49:10<1295::AID-NME993>3.0.CO;2-W}
}

@book{KloedenPlatenSDE,
  title={{Numerical Solution of Stochastic Differential Equations}},
  author={Kloeden, P.E. and Platen, E.},
  series={{Applications of Mathematics : Stochastic Modelling and Applied Probability}},
  year={1995},
  publisher={Springer}
}

@inproceedings{KongSunZhang2020,
author = {Kong, Lingkai and Sun, Jimeng and Zhang, Chao},
title = {{SDE-N}et: equipping deep neural networks with uncertainty estimates},
year = {2020},
publisher = {JMLR.org},
booktitle = {Proceedings of the 37th International Conference on Machine Learning},
articleno = {501},
numpages = {11},
series = {ICML'20}
}

@article{KrausTyranowski2019,
    author = {Kraus, Michael and Tyranowski, Tomasz M.},
    title = {Variational integrators for stochastic dissipative {H}amiltonian systems},
    journal = {IMA Journal of Numerical Analysis},
    volume = {41},
    number = {2},
    pages = {1318-1367},
    year = {2020},
    issn = {0272-4979},
    doi = {10.1093/imanum/draa022},
    url = {https://doi.org/10.1093/imanum/draa022}
}

@article{Kubo1954,
author = {Kubo, Ryogo},
title = {Note on the Stochastic Theory of Resonance Absorption},
journal = {Journal of the Physical Society of Japan},
volume = {9},
number = {6},
pages = {935-944},
year = {1954},
doi = {10.1143/JPSJ.9.935}
}

@book{Kunita1997,
  title={Stochastic {F}lows and {S}tochastic {D}ifferential {E}quations},
  author={Kunita, H.},
  isbn={9780521599252},
  lccn={89070813},
  series={Cambridge {S}tudies in {A}dvanced {M}athematics},
  year={1997},
  publisher={Cambridge {U}niversity {P}ress}
}

@article{LagarisLikasFotiadis1998,
  title={Artificial neural networks for solving ordinary and partial differential equations},
  author={Lagaris, Isaac E and Likas, Aristidis and Fotiadis, Dimitrios I},
  journal={IEEE transactions on neural networks},
  volume={9},
  number={5},
  pages={987--1000},
  year={1998},
  publisher={IEEE}
}

@article{LeeKang1990,
  title={Neural algorithm for solving differential equations},
  author={Lee, Hyuk and Kang, In Seok},
  journal={Journal of computational physics},
  volume={91},
  number={1},
  pages={110--131},
  year={1990},
  publisher={Elsevier}
}

@article{LeokZhang,
  title={Discrete {H}amiltonian variational integrators},
  author={Leok, Melvin and Zhang, Jingjing},
  journal={IMA Journal of Numerical Analysis},
  volume={31},
  number={4},
  pages={1497--1532},
  year={2011}
}

@Article{LeokShingel,
author="Leok, Melvin and Shingel, Tatiana",
title="General techniques for constructing variational integrators",
journal="Frontiers of Mathematics in China",
year="2012",
volume="7",
number="2",
pages="273--303"
}

@article{LewAVI,
  title={Asynchronous variational integrators},
  author={Lew, Adrian and Marsden, Jerrold E and Ortiz, Michael and West, Matthew},
  journal={Archive for Rational Mechanics and Analysis},
  volume={167},
  number={2},
  pages={85--146},
  year={2003}
}

@unpublished{LiuXiao2019,
author = {Liu, Xuanqing and Xiao, Tesi and Si, Si and Cao, Qin and Kumar, Sanjiv and Hsieh, Cho-Jui},
year = {2019},
title = {Neural {SDE}: Stabilizing Neural {ODE} Networks with Stochastic Noise},
doi = {10.48550/arXiv.1906.02355},
note={Unpublished, ar{X}iv:1906.02355}
}

@article{LuMengTyranowski2025,
title = {High-order stochastic integration schemes for the {R}osenbluth-{T}rubnikov collision operator in particle simulations},
journal = {Journal of Computational Physics},
volume = {527},
pages = {113811},
year = {2025},
doi = {https://doi.org/10.1016/j.jcp.2025.113811},
author = {Zhixin Lu and Guo Meng and Tomasz Tyranowski and Alex Chankin}
}

@article{MaDing2015,
 author = {Ma, Qiang and Ding, Xiaohua},
 title = {Stochastic Symplectic Partitioned {R}unge-{K}utta Methods for Stochastic {H}amiltonian Systems with Multiplicative Noise},
 journal = {Appl. Math. Comput.},
 issue_date = {February 2015},
 volume = {252},
 number = {C},
 month = feb,
 year = {2015},
 pages = {520--534}
}

@article{MarsdenPatrickShkoller,
  title={Multisymplectic geometry, variational integrators, and nonlinear {PDE}s},
  author={Marsden, Jerrold E and Patrick, George W and Shkoller, Steve},
  journal={Communications in Mathematical Physics},
  volume={199},
  number={2},
  pages={351--395},
  year={1998}
}

@book{MarsdenRatiuSymmetry,
  title={Introduction to Mechanics and Symmetry},
  author={Marsden, Jerrold and Ratiu, Tudor},
  series={Texts in Applied Mathematics},
  volume={17},
  year={1994},
  publisher={Springer Verlag}
}

@article{MarsdenWestVarInt,
  title={Discrete mechanics and variational integrators},
  author={Marsden, Jerrold E and West, Matthew},
  journal={Acta Numerica},
  volume={10},
  number={1},
  pages={357--514},
  year={2001}
}

@book{McDuffSalamonBook,
    author = {McDuff, Dusa and Salamon, Dietmar},
    title = {Introduction to Symplectic Topology},
    publisher = {Oxford University Press},
    year = {2017},
    isbn = {9780198794899},
    doi = {10.1093/oso/9780198794899.001.0001}
}

@article{McLachlanQuispel,
  title={Geometric integrators for {ODE}s},
  author={McLachlan, Robert I and Quispel, G Reinout W},
  journal={Journal of Physics A: Mathematical and General},
  volume={39},
  number={19},
  pages={5251--5285},
  year={2006}
}

@article{MeadeFernandez1994,
  title={The numerical solution of linear ordinary differential equations by feedforward neural networks},
  author={Meade Jr, Andrew J and Fernandez, Alvaro A},
  journal={Mathematical and computer modelling},
  volume={19},
  number={12},
  pages={1--25},
  year={1994},
  publisher={Elsevier}
}

@book{MilsteinBook,
  title={Numerical Integration of Stochastic Differential Equations},
  author={Milstein, G.N.},
  series={Mathematics and Its Applications},
  year={1995},
  publisher={Springer Netherlands}
}

@article{MilsteinRepin,
title = "Numerical methods for stochastic systems preserving symplectic structures",
journal = "SIAM J. Numer. Anal.",
volume = "40",
number = "4",
pages = "1583 - 1604",
year = "2002",
author = "Milstein, G. N. and Repin, Yu. M. and Tretyakov, M. V."
}

@book{MukamelBook1995,
  title={Principles of nonlinear optical spectroscopy},
  author={Mukamel, S.},
  isbn={9780195092783},
  lccn={gb95087868},
  series={Oxford series in optical and imaging sciences},
  url={https://books.google.de/books?id=k\_7uAAAAMAAJ},
  year={1995},
  publisher={Oxford University Press}
}

@article{OberBlobaum2016,
author = {Ober-Bl{\"o}baum, Sina},
title = {Galerkin variational integrators and modified symplectic {R}unge-{K}utta methods},
journal = {IMA Journal of Numerical Analysis},
volume = {37},
number = {1},
pages = {375-406},
year = {2017},
doi = {10.1093/imanum/drv062},
}

@article{Pavlov,
  title={Structure-preserving discretization of incompressible fluids},
  author={Pavlov, Dmitry and Mullen, Patrick and Tong, Yiying and Kanso, Eva and Marsden, Jerrold E and Desbrun, Mathieu},
  journal={Physica D: Nonlinear Phenomena},
  volume={240},
  number={6},
  pages={443--458},
  year={2011}
}

@article{PengArai2022,
author = {Peng, Linyu and Arai, Noriyoshi and Yasuoka, Kenji},
title = {A stochastic {H}amiltonian formulation applied to dissipative particle dynamics},
year = {2022},
volume = {426},
number = {C},
doi = {10.1016/j.amc.2022.127126},
journal = {Appl. Math. Comput.},
numpages = {13}
}

@book{PolterovichBook2001,
  title     = {The Geometry of the Group of Symplectic Diffeomorphisms},
  author    = {Polterovich, Leonid},
  series    = {Lectures in Mathematics, ETH Z{\"u}rich},
  year      = {2001},
  publisher = {Birkh{\"a}user Basel},
  doi       = {10.1007/978-3-0348-8299-6},
  isbn      = {978-3-0348-8299-6}
}

@article{Ripoll2001,
author = {Ripoll,M. and Ernst,M. H.  and Espa\~{n}ol,P. },
title = {Large scale and mesoscopic hydrodynamics for dissipative particle dynamics},
journal = {The Journal of Chemical Physics},
volume = {115},
number = {15},
pages = {7271-7284},
year = {2001},
doi = {10.1063/1.1402989}
}

@inproceedings{RowleyMarsden,
  title={Variational integrators for degenerate {L}agrangians, with application to point vortices},
  author={Rowley, Clarence W and Marsden, Jerrold E},
  booktitle={Decision and Control, 2002, Proceedings of the 41st IEEE Conference on},
  volume={2},
  pages={1521--1527},
  year={2002},
  organization={IEEE}
}

@InProceedings{Saemundsson2020,
  title = 	 {Variational Integrator Networks for Physically Structured Embeddings},
  author =       {Saemundsson, Steindor and Terenin, Alexander and Hofmann, Katja and Deisenroth, Marc},
  booktitle = 	 {Proceedings of the Twenty Third International Conference on Artificial Intelligence and Statistics},
  pages = 	 {3078--3087},
  year = 	 {2020},
  editor = 	 {Chiappa, Silvia and Calandra, Roberto},
  volume = 	 {108},
  series = 	 {Proceedings of Machine Learning Research},
  month = 	 {26--28 Aug},
  publisher =    {PMLR},
  url = 	 {https://proceedings.mlr.press/v108/saemundsson20a.html}
}

@article{SanzSerna,
  title={Symplectic integrators for {H}amiltonian problems: an overview},
  author = {Sanz-Serna, J. M.},
  journal={Acta Numerica},
  volume={1},
  pages = {243--286},
  year={1992}
}

@incollection{Skeel1999,
author="Skeel, Robert D.",
editor="Ainsworth, Mark
and Levesley, Jeremy
and Marletta, Marco",
title="Integration Schemes for Molecular Dynamics and Related Applications",
bookTitle="The Graduate Student's Guide to Numerical Analysis '98: Lecture Notes from the VIII EPSRC Summer School in Numerical Analysis",
year="1999",
publisher="Springer Berlin Heidelberg",
address="Berlin, Heidelberg",
pages="119--176",
isbn="978-3-662-03972-4",
doi="10.1007/978-3-662-03972-4_4",
url="https://doi.org/10.1007/978-3-662-03972-4_4"
}

@article{Sosanya2022,
  title={Dissipative {H}amiltonian neural networks: {L}earning dissipative and conservative dynamics separately},
  author={Sosanya, Andrew and Greydanus, Sam},
  journal={arXiv preprint arXiv:2201.10085},
  year={2022}
}

@article{Sonnendrucker2015,
title = {A split control variate scheme for {PIC} simulations with collisions},
journal = "Journal of Computational Physics",
volume = "295",
pages = "402 - 419",
year = "2015",
issn = "0021-9991",
doi = "https://doi.org/10.1016/j.jcp.2015.04.004",
url = "http://www.sciencedirect.com/science/article/pii/S0021999115002442",
author = "Eric Sonnendr{\"u}cker and Abigail Wacher and Roman Hatzky and Ralf Kleiber"
}

@article{SternDesbrun,
  title={Variational integrators for {M}axwell's equations with sources},
  author={Stern, Ari and Tong, Yiying and Desbrun, Mathieu and Marsden, Jerrold E},
  journal={PIERS Online},
  volume={4},
  number={7},
  pages={711--715},
  doi={10.2529/PIERS071019000855},
  year={2008}
}

@article{SunWang2016,
title = {Stochastic symplectic methods based on the {P}ad\'e approximations for linear stochastic {H}amiltonian systems},
journal = {Journal of Computational and Applied Mathematics},
year = {2016},
note = {http://dx.doi.org/10.1016/j.cam.2016.08.011},
author = "Liying Sun and Lijin Wang"
}

@article{Turaev2002,
doi = {10.1088/0951-7715/16/1/308},
url = {https://doi.org/10.1088/0951-7715/16/1/308},
year = {2002},
volume = {16},
number = {1},
pages = {123-135},
author = {Dmitry Turaev},
title = {Polynomial approximations of symplectic dynamics and richness of chaos in non-hyperbolic area-preserving maps},
journal = {Nonlinearity}
}

@article{TyranowskiVlasovMaxwell,
author = {Tyranowski, Tomasz M. },
title = {Stochastic variational principles for the collisional {V}lasov-{M}axwell and {V}lasov-{P}oisson equations},
journal = {Proceedings of the Royal Society A: Mathematical, Physical and Engineering Sciences},
volume = {477},
number = {2252},
pages = {20210167},
year = {2021},
doi = {10.1098/rspa.2021.0167},
URL = {https://royalsocietypublishing.org/doi/abs/10.1098/rspa.2021.0167},
eprint = {https://royalsocietypublishing.org/doi/pdf/10.1098/rspa.2021.0167}
}

@unpublished{TyranowskiStochasticModelReduction,
title={Data-driven structure-preserving model reduction for stochastic {H}amiltonian systems},
author={Tyranowski, Tomasz M.},
year={2022},
note={Preprint ar{X}iv:2201.13391}
}

@article{TyranowskiDesbrunRAMVI,
AUTHOR = {Tyranowski, Tomasz M. and Desbrun, Mathieu},
TITLE = {R-Adaptive Multisymplectic and Variational Integrators},
JOURNAL = {Mathematics},
VOLUME = {7},
YEAR = {2019},
NUMBER = {7},
ARTICLE-NUMBER = {642},
URL = {https://www.mdpi.com/2227-7390/7/7/642},
ISSN = {2227-7390},
DOI = {10.3390/math7070642}
}

@article{TyranowskiDesbrunLinearLagrangians,
AUTHOR = {Tyranowski, Tomasz M. and Desbrun, Mathieu},
TITLE = {Variational Partitioned {R}unge-{K}utta Methods for {L}agrangians Linear in Velocities},
JOURNAL = {Mathematics},
VOLUME = {7},
YEAR = {2019},
NUMBER = {9},
ARTICLE-NUMBER = {861},
URL = {https://www.mdpi.com/2227-7390/7/9/861},
ISSN = {2227-7390},
DOI = {10.3390/math7090861}
}

@article{TyranowskiKraus2021,
author = {Tyranowski, Tomasz M. and Kraus, Michael},
title = {Symplectic model reduction methods for the {V}lasov equation},
journal = {Contributions to Plasma Physics},
volume = {},
number = {},
year={2022},
pages = {e202200046},
doi = {https://doi.org/10.1002/ctpp.202200046}
}

@article{VanKampen1976,
title = "Stochastic differential equations",
journal = "Physics Reports",
volume = "24",
number = "3",
pages = "171 - 228",
year = "1976",
doi = "https://doi.org/10.1016/0370-1573(76)90029-6",
author = "Van Kampen, N.G."
}

@article{WangYao2021,
AUTHOR = {Wang, Yongguang and Yao, Shuzhen},
TITLE = {Neural Stochastic Differential Equations with Neural Processes Family Members for Uncertainty Estimation in Deep Learning},
JOURNAL = {Sensors},
VOLUME = {21},
YEAR = {2021},
NUMBER = {11},
ARTICLE-NUMBER = {3708},
URL = {https://www.mdpi.com/1424-8220/21/11/3708},
PubMedID = {34073566},
ISSN = {1424-8220},
DOI = {10.3390/s21113708}
}

@phdthesis{WestPHD,
      author={West, Matthew},
      title={Variational Integrators},
      year={2004},
      school={California Institute of Technology}
}

@misc{XiaoZhangTang2024,
  title         = {Generalized {L}agrangian Neural Networks},
  author        = {Shanshan Xiao and Jiawei Zhang and Yifa Tang},
	howpublished  = {arXiv preprint},
  year          = {2024},
  eprint        = {2401.03728},
  archivePrefix = {arXiv},
  primaryClass  = {math.DS},
  doi           = {10.48550/arXiv.2401.03728},
  url           = {https://arXiv.org/abs/2401.03728},
	note          = {Available at https://arXiv.org/abs/2401.03728}
}

@article{Xu2024,
	author  = {Zhongshu  Xu and Yuan Chen and Qifan  Chen and Dongbin Xiu},
	title   = {Modeling Unknown Stochastic Dynamical System via Autoencoder},
	journal = {Journal of Machine Learning for Modeling and Computing},
	issn    = {2689-3967},
	year    = {2024},
	volume  = {5},
	number  = {3},
	URL     = {https://dl.begellhouse.com/journals/558048804a15188a,7dd2ba1c3481309f,1fa59aa90d1bc10a.html},
	pages   = {87--112},
	DOI     = {10.1615/JMachLearnModelComput.2024055773}
}

@article{ZhongDeyChakraborty2019,
  title={Symplectic ode-net: Learning hamiltonian dynamics with control},
  author={Zhong, Yaofeng Desmond and Dey, Biswadip and Chakraborty, Amit},
  journal={arXiv preprint arXiv:1909.12077},
  year={2019},
  note={Published as a conference paper at ICLR 2020}
}

@article{ZhongDeyChakraborty2020,
  title={Dissipative {SymODEN}: Encoding {H}amiltonian dynamics with dissipation and control into deep learning},
  author={Zhong, Yaofeng Desmond and Dey, Biswadip and Chakraborty, Amit},
  journal={arXiv preprint arXiv:2002.08860},
  year={2020},
  note={Published at ICLR 2020 Workshop on Integration of Deep Neural Models and Differential Equations (DeepDiffEq)}
}

@article{zhu2024dyngma,
  title={{DynGMA}: A robust approach for learning stochastic differential equations from data},
  author={Zhu, Aiqing and Li, Qianxiao},
  journal={Journal of Computational Physics},
  volume={513},
  pages={113200},
  year={2024},
  publisher={Elsevier}
}

\end{document}